\documentclass{article}

\usepackage{microtype}
\usepackage{graphicx}
\usepackage{subfigure}
\usepackage{booktabs} 

\usepackage{hyperref}

\usepackage{pifont}
\newcommand{\cmark}{\ding{51}}%
\newcommand{\xmark}{\ding{55}}%

\usepackage{algorithm}
\usepackage[noend]{algorithmic}

\newcommand{\add}[1]{{\color{black}#1}}
\usepackage[preprint]{neurips_2023}

\usepackage{amsmath}
\usepackage{amssymb}
\usepackage{mathtools}
\usepackage{amsthm}
\usepackage{multirow}
\usepackage[capitalize,noabbrev]{cleveref}

\theoremstyle{plain}
\newtheorem{theorem}{Theorem}[section]

\newtheorem{lemma}[theorem]{Lemma}
\newtheorem{corollary}[theorem]{Corollary}
\theoremstyle{definition}

\newtheorem{assumption}[theorem]{Assumption}
\theoremstyle{remark}
\newtheorem{remark}[theorem]{Remark}

\usepackage[textsize=tiny]{todonotes}

\usepackage{wrapfig}
\usepackage{tablefootnote}
\usepackage{threeparttable}
\usepackage{makecell}

\usepackage{colortbl}
\definecolor{bgcolor}{rgb}{0.8,1,1}
\definecolor{bgcolor2}{rgb}{0.8,1,0.8}
\usepackage{threeparttable}
\newcommand{\myred}[1]{{\color{red}#1}}

\usepackage{comment}
\usepackage{kamzolov}
\usepackage{enumitem}

\newcommand{\kappaf}{\kappa}
\newcommand{\DFP}{\text{DFP}}
\newcommand{\Broyd}{\text{Broyd}}
\newcommand{\BFGS}{\text{BFGS}}

\newcommand{\CRN}{\text{CRN }}
\newcommand{\QN}{\text{QN }}

\newcommand{\B}{\widetilde{B}}
\newcommand{\eqdef}{:=}

\def\bk{\bar \kappa}

\author{%
    Dmitry Kamzolov\\
    MBZUAI\thanks{Mohamed bin Zayed University of Artificial Intelligence, United Arab Emirates}\\
    \texttt{kamzolov.opt@gmail.com}
    \And
    Klea Ziu\\
    MBZUAI\footnotemark[1]\\
    \texttt{klea.ziu@mbzuai.ac.ae}
    \And
    Artem Agafonov\\
    MBZUAI\footnotemark[1]\\
    \texttt{kamzolov.opt@gmail.com}
    \And
    Martin Tak\'a\v{c}\\
    MBZUAI\footnotemark[1]\\
    \texttt{takac.MT@gmail.com}
    }
\title{Cubic Regularization is the Key! The First Accelerated Quasi-Newton Method with a Global Convergence Rate of $O(k^{-2})$ for Convex Functions.}

\begin{document}
\maketitle
\begin{abstract} %
In this paper, we propose the first Quasi-Newton method with a global convergence rate of $O(k^{-1})$ for general convex functions. Quasi-Newton methods, such as BFGS, SR-1, are well-known for their impressive practical performance. However, they may be slower than gradient descent for general convex functions, with the best theoretical rate of $O(k^{-1/3})$. This gap between impressive practical performance and poor theoretical guarantees was an open question for a long period of time. In this paper, we make a significant step to close this gap. We improve upon the existing rate and propose the Cubic Regularized Quasi-Newton Method with a convergence rate of $O(k^{-1})$. The key to achieving this improvement is to use the Cubic Regularized Newton Method over the Damped Newton Method as an outer method, where the Quasi-Newton update is an inexact Hessian approximation. Using this approach, we propose the first Accelerated Quasi-Newton method with a global convergence rate of $O(k^{-2})$ for general convex functions. In special cases where we can improve the precision of the approximation, we achieve a global convergence rate of $O(k^{-3})$, which is faster than any first-order method. To make these methods practical, we introduce the Adaptive Inexact Cubic Regularized Newton Method and its accelerated version, which provide real-time control of the approximation error. We show that the proposed methods have impressive practical performance and outperform both first and second-order methods.
\end{abstract}

\section{Introduction}
In this paper\footnote{The first version of the paper, titled "Accelerated Adaptive Cubic Regularized Quasi-Newton Methods," was made available online on February 10th.}, we consider the following optimization problem
\begin{equation}
\label{eq:problem}
    \min_{x\in \R^d} f(x).
\end{equation}
We assume that function $f(x)$ is convex and has $L_1$-Lipschitz-continuous gradient and $L_2$-Lipschitz-continuous Hessian. 
\begin{assumption}
Let $x^{\ast}$ be a minimizer of the function $f$. The function $f\in C^2$ is \textbf{convex}  if $\forall x \in \R^d$
\[
    \nabla^2 f(x) \succeq 0.
\]
\end{assumption}
\begin{assumption}
The  function $f(x)\in C$ has $L_1$-Lipschitz-continuous gradient if for any $x,y \in \R^d$
\begin{equation*}
\label{eq:L1}
    \|\nabla f(x) - \nabla f(y) \|_{\ast} \leq L_1\|x-y\|.
\end{equation*}
\end{assumption}
\begin{assumption}
The function $f(x)\in C^2$ has $L_2$-Lipschitz-continuous Hessian if for any $x,y \in \R^d$
\begin{equation*}
\label{eq:L2}
    \|\nabla^2 f(x) - \nabla^2 f(y) \| \leq L_2\|x-y\|.
\end{equation*}
\end{assumption}
Note, these assumptions are the most standard assumptions for the first and second-order methods. For $\mu$-strongly convex function, we denote the condition number $\kappa= \tfrac{L_1}{\mu}$. 
\\
Second-order methods play a significant role in modern optimization and have roots in the classical works of \citet{Newton}, \citet{Raphson}, and \citet{Simpson}. Over the years, they have been studied in depth, modified and improved in works of \cite{kantorovich1948functional, more1978levenberg, Griewank-cubic-1981,nesterov2006cubic}, and are now widely used in industrial and scientific computing. These methods typically achieve faster convergence than first-order algorithms, but at the same time, the per-iteration cost of second-order methods is significantly higher. For instance, the classical Newton method can be written as:
\begin{equation}
    \label{eq:classical_Newton}
    x_{t+1} = x_t - \left[\nabla^2 f(x_t) \right]^{-1} \nabla f(x_t)
\end{equation}
It has a quadratic local convergence, but each iteration requires computation of the full Hessian and matrix inversion, which is impractical for large-scale optimization problems.
    \subsection{Quasi-Newton Methods}
In order to find a balance between computational cost of iteration and the fast convergence of second-order methods, different variants of Quasi-Newton(QN) methods have been proposed in papers of~\cite{broyden1967quasi, fletcher1970new,goldfarb1970family,   wright1999numerical, shanno1970conditioning,byrd1996analysis,khalfan1993theoretical, conn1991convergence,berahas2016multi}. 
\begin{equation*}
    \label{eq:classical_Quasi-Newton}
    x_{t+1} = x_t - \left[B_t \right]^{-1} \nabla f(x_t) = x_t - H_t \nabla f(x_t)
\end{equation*}
These methods approximate the Hessian matrix and its inverse, denoted by $B_t$ and $H_t$, respectively, based on first-order (gradient) information from previous iterations. We refer to this approach as \textit{\QN with history}. Another variant, known as \textit{\QN with sampling}, approximates the Hessian along directions of Hessian vector products, as described by \cite{berahas2021quasi}. Note that \QN with history can be combined with \QN with sampling.
\\
\QN methods have cheaper iteration costs compared to second-order methods with full information, but they also have local superlinear rates that are faster than the linear rates of first-order methods. \QN methods have become modern optimization classics due to their effectiveness and practicality. Usually, the accuracy of the Hessian approximation improves with iterations. Theoretical results on the accuracy of approximation are available for certain algorithms such as BFGS, SR1, and DFP by~\cite{fletcher2013practical, nocedal1999numerical, dennis1977quasi}. It is well known that \QN methods achieve an asymptotic superlinear convergence rate in the local neighborhood of an optimal solution, but explicit convergence rate was not established until recent works of ~\cite{rodomanov2021rates, rodomanov2021greedy, rodomanov2021new, lin2021greedy}, which started a new wave of \QN methods. However, for a long period of time, fast \textit{global convergence} of \QN methods was an open problem. To the best of our knowledge, the best QN method convergence for strongly convex functions is $O\ls\kappaf^3\log(\varepsilon^{-1})\rs$ by \cite{meng2020fast,berahas2021quasi}, while for convex functions, there are even fewer results. It means that QN methods may be slower than gradient descent. The main goal of our paper is to close this gap and to propose \QN method that will be not slower than gradient descent. In Table \ref{tab1}, we present the state-of-the-art (SOTA) QN methods and their convergence rates and our proposed methods and their convergence rates. We want to highlight the following points and make some remarks:
\begin{itemize}[leftmargin=10pt,noitemsep,topsep=0pt] 
    \item The SOTA QN methods have slower global convergence rates than gradient descent for general convex functions. 
    \item The SOTA QN methods are only matching global convergence rates of gradient descent specifically for quadratic functions, where the Hessian matrix is a constant matrix and doesn't change through iterations.
    \item We propose Cubic QN method with history, which is the first QN method that matches global convergence rates of gradient descent for general convex functions. We also propose Cubic QN method with sampling, which is faster than gradient descent and matches convergence rates for Cubic Regularized Newton Method by \cite{nesterov2006cubic}.
    \item So far, Accelerated QN methods did not exist. In our paper, we propose the first Accelerated QN method. Accelerated QN method with history matches convergence rates of Accelerated Gradient Method(AGM) for convex function. We also propose Accelerated Cubic QN method with sampling, which is faster than AGM and matches convergence rates for Accelerated Cubic Regularized Newton Method by \cite{nesterov2008accelerating}.  
    \item One can obtain convergence rates for convex functions from convergence rate for strongly convex functions by using the regularization technique: $\min_x \lb f(x) + \tfrac{\varepsilon}{4 R^2}\|x\|^2\rb$. It means that if a QN method has the convergence rate $O\ls\kappaf^{3}\log(\varepsilon^{-1})\rs$ for strongly convex function, it also has the convergence rate $O(1) (\tfrac{L_1R^2}{\varepsilon})^{3}$ for convex functions.
    \item First-order lower bounds may be not applicable for QN methods, especially for QN with sampling. It is still an open problem to formulate such lower bounds. 
\end{itemize}

\begin{table}[h]
\centering
        \scriptsize
    \caption{Comparison of first-order methods and QN methods for convex and strongly-convex problems. {\em Notation:} $\varepsilon$ denotes solution accuracy for \eqref{eq:problem}. The `Assumption' column shows whether the function $f(x)$ is convex, strongly convex or quadratic. The `Convergence' column shows how many iterations (up to a constant factor) are enough to do to guarantee that $f(x_t)-f(x^{\ast})\leq \varepsilon$. Smaller power of $\varepsilon$ means faster convergence. }
    \label{tab1}
  {
    \begin{tabular}{|c|c|c|c|}
    \hline
    \textbf{Method} & \textbf{Reference} & \textbf{Assumption} & \textbf{Convergence}  
    \\ \hline
    \multirow{2}{*}{\texttt{Gradient Descent}}    
     & \multirow{2}{*}{\cite{nesterov2018lectures}} 
   &  convex function   & 
   $\tfrac{L_1R^2}{\varepsilon}$
    \\
    \cline{3-4} 
    & &
   $\mu$-strongly convex function
    & 
   {$\kappaf\log(\varepsilon^{-1})$}
    \\ \hline
    \multirow{2}{*}{\begin{tabular}{@{}c@{}}\texttt{Accelerated}  \\  \texttt{Gradient Method}
   \end{tabular}}      
     & \multirow{2}{*}{\cite{nesterov2018lectures}} 
   &  convex function   & 
   $\ls\tfrac{L_1R^2}{\varepsilon}\rs^{1/2}$
    \\
    \cline{3-4} 
    & &
   $\mu$-strongly convex function
    & 
   {$\kappaf^{1/2}\log(\varepsilon^{-1})$}
    \\ \hline
   \texttt{BFGS with history} & \cite{meng2020fast}
   & $\mu$-strongly convex function & $\kappaf^3\log(\varepsilon^{-1})$ 
    \\ \hline
    \texttt{BFGS with sampling} & 
   \cite{berahas2021quasi}
   & $\mu$-strongly convex function & $\kappaf^3\log(\varepsilon^{-1})$ 
    \\ \hline
    Quasi-Newton methods (QN) &  \begin{tabular}{@{}c@{}} \tiny{\cite{rodomanov2021greedy}} \\  \tiny{\cite{rodomanov2021rates}} \end{tabular}
   & \begin{tabular}{@{}c@{}} quadratic function \\  $f(x) = \frac{1}{2}x^{\top}Ax - b^{\top}x$ \end{tabular} & $\kappaf\log(\varepsilon^{-1})$
    \\ \hline 
    \multirow{2}{*}{\begin{tabular}{@{}c@{}}\texttt{First-order}  \\  \texttt{lower bound}
   \end{tabular}}      
     & \multirow{2}{*}{\cite{nesterov2018lectures}} 
   &  convex function   & 
   $\ls\tfrac{L_1R^2}{\varepsilon}\rs^{1/2}$
    \\
    \cline{3-4} 
    & &
   $\mu$-strongly convex function
    & 
   {$\kappaf^{1/2}\log(\varepsilon^{-1})$}
    \\ \hline
    \hline
    \cellcolor{bgcolor2}{
    \begin{tabular}{@{}c@{}} \texttt{Cubic L-BFGS} \\  \texttt{with history} \end{tabular}} & \cellcolor{bgcolor2}{This work} & \cellcolor{bgcolor2}{convex function}& \cellcolor{bgcolor2}{$\tfrac{L_1R^2}{\varepsilon}$} 
    \\ \hline
    \cellcolor{bgcolor2}{
    \begin{tabular}{@{}c@{}} \texttt{Cubic L-BFGS} \\  \texttt{with sampling} \end{tabular}} & \cellcolor{bgcolor2}{This work} & \cellcolor{bgcolor2}{convex function}& \cellcolor{bgcolor2}{$\ls\tfrac{L_2R^3}{\varepsilon}\rs^{1/2}$} 
    \\ \hline
    \cellcolor{bgcolor2}{
    \begin{tabular}{@{}c@{}} \texttt{Accelerated} \\ \texttt{Cubic L-BFGS} \\  \texttt{with history} \end{tabular}} & \cellcolor{bgcolor2}{This work} & \cellcolor{bgcolor2}{convex function}& \cellcolor{bgcolor2}{$\ls\tfrac{L_1R^2}{\varepsilon}\rs^{1/2}$} 
    \\ \hline
    \cellcolor{bgcolor2}{
    \begin{tabular}{@{}c@{}} \texttt{Accelerated} \\ \texttt{Cubic L-BFGS} \\  \texttt{with sampling} \end{tabular}} & \cellcolor{bgcolor2}{This work} & \cellcolor{bgcolor2}{convex function}& \cellcolor{bgcolor2}{$\ls\tfrac{L_2R^3}{\varepsilon}\rs^{1/3}$} 
    \\ \hline
    \end{tabular} 
    }
\end{table}
    \subsection{Modern second-order methods or how to improve Quasi-Newton (QN) methods}
    In this subsection, we discuss the intuition and reasons: why classical \QN methods may be slower than Gradient Descent and what is stopping us from getting good global convergence rates for \QN methods. From our perspective, \textbf{there are three main theoretical difficulties with \QN methods}: 
    \textit{\textbf{1)} The convergence of an outer (exact) method; \textbf{2)} The changing Hessian matrix between steps; \textbf{3)} The quality of the Hessian approximation.}
    \\
    Let us briefly address points $2$ and $3$ before moving on to the main point $1$. \cite{rodomanov2021rates, rodomanov2021greedy, rodomanov2021new} have made a breakthrough in \QN methods by explicitly analyzing the changing Hessian matrix and the quality of Hessian approximation. In the sections dedicated to convex functions, they proposed a pretty strict condition, such as strongly self-concordance for the function $f$, to track and bound the difference in the Hessian matrix between steps. However, even with such conditions, it was only possible to prove local superlinear convergence. In the sections dedicated to quadratic problems, it was shown that even when the Hessian matrix is constant, it is still challenging to approximate it using \QN updates. As a result, for quadratic functions, \QN methods have the same convergence rate as a Gradient Descent or solve the problem exactly with $t=d$ steps as Conjugate Gradient Descent, where $d$ is the dimension of $x\in \R^d$. These two problems are outside the scope of our paper, as well as local superlinear convergence. We hope that these problems can be addressed in future work. Now we move on to point $1$: "The convergence of an outer(exact) method", which we aim to solve in our paper. By the outer method, we mean the method for which we replace the exact Hessian matrix $\nabla^2 f(x_t)$ by the inexact \QN approximation $B_t$. The outer method used in most \QN methods is a Damped Newton method, although there are some variants where \QN updates are combined with Trust-Region Newton methods. In the next paragraph, we discuss why these methods may be problematic from a theoretical perspective.
    \\
    Despite the widespread use of second-order algorithms with quadratic local convergence rates, fast global convergence guarantees are not common for Newton-type methods. The reason is that a Newton step \eqref{eq:classical_Newton} is a minimizer of the second-order Taylor approximation
    \begin{equation*}
    x_{k+1} = \argmin_{x} \large\{  \Phi_{x_k}(x) \eqdef f(x_k) + \langle \nabla f(x_k), x - x_k \rangle   
    + \tfrac{1}{2}\langle \nabla^2 f(x_k)(x-x_k), x-x_k\rangle \large\},
    \end{equation*}
    which is not an upper bound on the objective function $f(x)$. This is one of the reasons, why the classical Newton method may even diverge from points that are far from the solution, as shown in Example 1.2.3 in \cite{nesterov2018lectures}.
    Different strategies such as line-search, trust regions, and damping/truncation are used for globalization~\cite{more1978levenberg,conn2000trust,nocedal1999numerical,martens2010deep}. However, to the best of our knowledge, these methods converge at a rate of $O(\varepsilon^{-2})$, which makes them slower than Gradient Descent. In \QN methods, we replace the exact Hessian matrix by an inexact approximation, which means we are losing information and slowing down the convergence. The solution to this problem is to use the Cubic Regularized Newton(CRN) method as an outer method by \cite{nesterov2006cubic}. The \CRN method converges globally with a convergence rate $O(\varepsilon^{-1/2})$, which is faster than Gradient Descent. Therefore, even with the loss of information from \QN approximation, the Cubic \QN can still converge globally and quickly enough. To demonstrate the practical performance difference, we present Figure \ref{fig:a9a_example}, where it can be seen that \CRN method is much faster compared to the Damped Newton method. Theoretical and practical results on the Cubic \QN methods are presented in the next sections and Appendix.

    \begin{wrapfigure}{r}{0.45\textwidth}
  \includegraphics[width=1\linewidth]{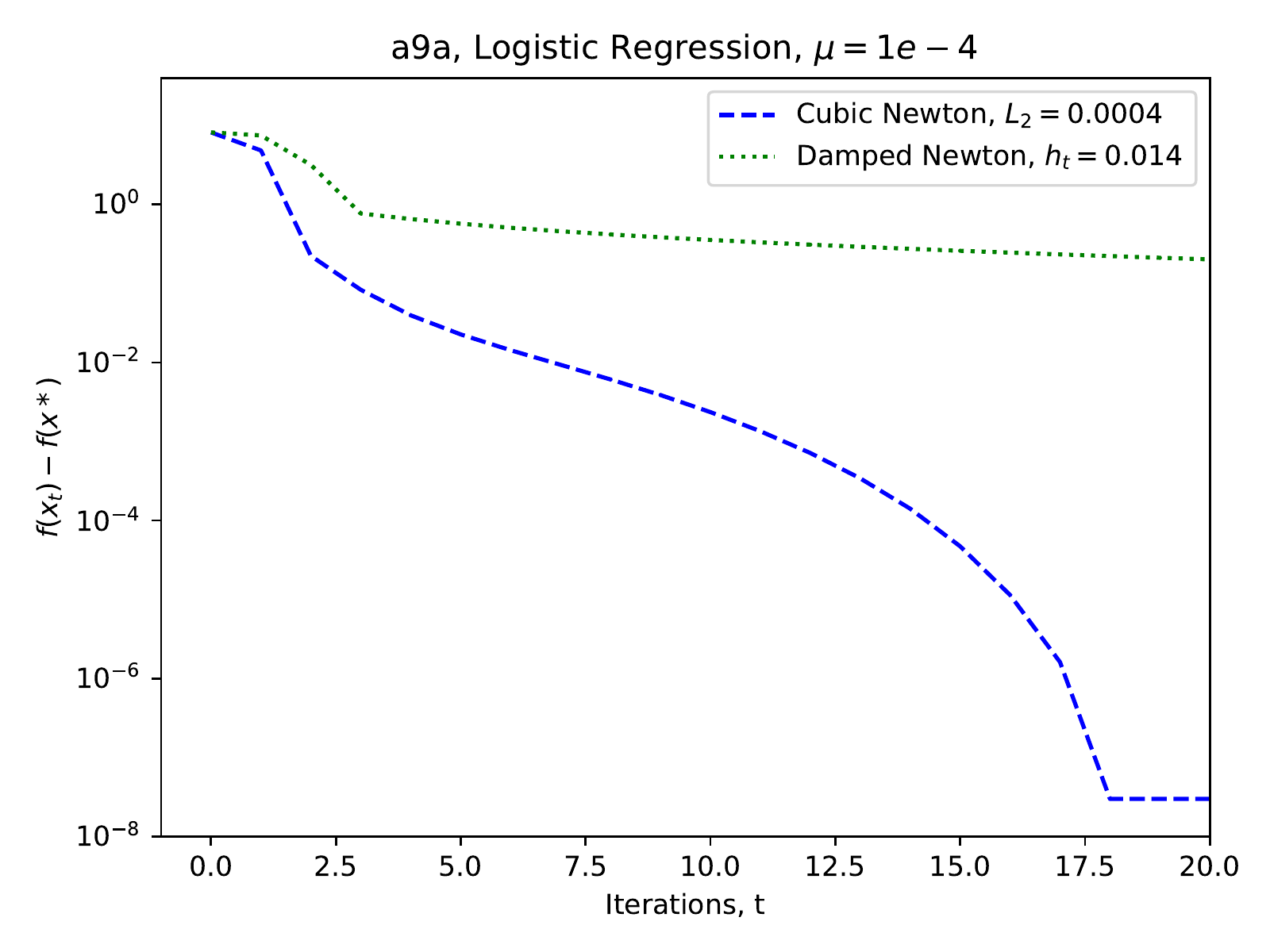}
    \caption{Comparison of \CRN method and Damped Newton method on regularized Logistic Regression for \texttt{a9a} dataset using the best tuned parameters from the starting point $x_0= 3*e$, where the regularizer $\mu = 1e-4$ and $e$ is a vector of all ones.}
    \label{fig:a9a_example}
    \end{wrapfigure}
    \subsection{Related literature}   
    \paragraph{Globalization via Cubic Regularization.}
    The Cubic Regularized Newton(CRN) method by \cite{nesterov2006cubic} is one of the main approaches to globalize the Newton method. The \CRN update has the following form
    \begin{equation}
        \label{eq:cubic_step_intro}
        \min \limits_{x}  \lb \Phi_{x_k}(x) + \tfrac{M}{6}\|x - x_k\|^3\rb.
    \end{equation}
    By choosing the regularization parameter $M$ greater than the Lipschitz-continuous Hessian parameter $L_2$, the cubic regularized Taylor approximation term majorizes the objective function $f(x)$.
    Therefore, this algorithm achieves global convergence with the convergence rate $O(\varepsilon^{-1/2})$. Various acceleration techniques can be applied, such as Nesterov acceleration with rate $O(\varepsilon^{-1/3})$ by \cite{nesterov2008accelerating}, near-optimal accelerations up to a logarithmic factor $\tilde{O}(\varepsilon^{-2/7})$ by \cite{monteiro2013accelerated, gasnikov2019near}, and optimal acceleration $O(\varepsilon^{-2/7})$ by \cite{kovalev2022first,carmon2022optimal}. Under higher smoothness assumption, superfast second-order methods can accelerate even more, up to the rate $O(\varepsilon^{-1/4})$ with Nesterov acceleration by \cite{nesterov2021superfast}, and  up to $O(\varepsilon^{-1/5})$ with near-optimal acceleration by \cite{nesterov2021inexact, kamzolov2020near}.
    The \CRN method allows for inexact Hessian approximations by \cite{ghadimi2017second} or stochastic Hessians by \cite{agafonov2020inexact, antonakopoulos2022extra}, making it applicable to distributed optimization by \cite{zhang2015disco,daneshmand2021newton, agafonov2021accelerated}. Moreover, all the results mentioned above about the \CRN method are also generalizable to higher-order (tensor) methods by \cite{nesterov2021implementable, nesterov2019inexact, grapiglia2021inexact,  agafonov2020inexact, doikov2020local, doikov2020dynamic, dvurechensky2021hyperfast, kamzolov2022exploiting}. 
    However, the main drawback of the \CRN method is the auxiliary problem~\eqref{eq:cubic_step_intro}, which typically requires running a separate optimization algorithm to solve it.
    Several papers have introduced regularization by the norm of the gradient to obtain an explicit regularized Newton step by \cite{polyak2009regularized, polyak2017complexity}. \cite{mishchenko2021regularized, doikov2021gradient} improved its convergence rate up to $O(\varepsilon^{-1/2})$, under higher assumptions on smoothness it accelerates up to $O(\varepsilon^{-1/3})$ by \cite{doikov2022super}.  Affine-Invariant \CRN method with local Hessian norms has the convergence rate $O(\varepsilon^{-1/2})$ and the same subproblem as a classical Newton step by \cite{hanzely2022newton}. Inexact \CRN methods are also well developed for nonconvex problems by \cite{cartis2011adaptive, cartis2011adaptive2, bellavia2019adaptive, lucchi2019stochastic, xu2020newton, doikov2022second}.
    \\
    {\bf Hessian approximation methods.}
    Another approach to reducing high iteration complexity involves the use of Newton-type methods operating in random low-dimensional subspaces by \cite{SDNA, RSN, RBCN, SSCN} and subsampled algorithms for empirical risk minimization by~\cite{NewtonSketch, xu2020newton, SN2019}. Stochastic \QN methods were also developed by \cite{Bordes2009a,OnlineBFGS,SBFGS,SQN2016,RBFGS2020}. In the context of distributed optimization, Hessian approximations are used not only to reduce the complexity of iteration but also the complexity of communication. These methods typically utilize compression and sketching techniques by \cite{NL2021, FedNL, BL2022, Newton-3PC, agafonov2022flecs, agafonov2022flecscgd}.

     \paragraph{Local Newton-type methods.}
     Another main research direction of second-order methods is local Newton methods for self-concordant problems~\cite{nesterov1994interior, nesterov2018lectures}. These methods have led to the development of interior-point methods, which have been a significant advancement in optimization and have been widely used in various fields such as engineering, finance, and machine learning \cite{alizadeh1995interior, rao1998application, koh2007interior, bertocchi2020deep}. This area of research is continuing to evolve currently  \cite{dvurechensky2018global, hildebrand2020optimal, doikov2022affine, nesterov2022set}.

\subsection{Contribution}
\label{sec:contribution}
In this paper, we present a novel approach to Quasi-Newton methods by integrating a Cubic Regularization technique. Our method utilizes the \QN approximation as an inexact Hessian within the \CRN algorithm. Furthermore, we highlight the main contributions of this paper.

\begin{itemize}[leftmargin=10pt] 
\vspace{-0.3 cm}
    \item \textbf{First Cubic Regularized Quasi-Newton method for convex optimization.} 
    We propose a Cubic Regularized \QN method that achieves global convergence rate $O(1) \max\lb \tfrac{L_1R^2}{\varepsilon}, \sqrt{\tfrac{MR^3}{\varepsilon}}\rb$ iterations for an appropriate choice of \QN Hessian approximation. It is the first \QN method that matches the global convergence of gradient descent.
\vspace{-0.1 cm}

    \item \textbf{First Accelerated Cubic Regularized Quasi-Newton method for convex optimization.}
    We introduce the first Accelerated \QN method with global convergence rate $O(1) \max\lb \ls \tfrac{\delta R^2}{\varepsilon}\rs^{1/2};  \ls\tfrac{MR^3}{\varepsilon}\rs^{1/3}\rb$ iterations. It is the first Accelerated \QN method that matches the global convergence of accelerated gradient descent.
    \vspace{-0.1 cm}
    
    \item \textbf{Novel concept of Hessian inexactness and corresponding convergence analysis.}
    We introduce the concept of directional $\delta_x^{y}$-inexact Hessian. This concept of inexactness depends only on the projection of Hessian and its approximation along the direction $y-x$. We show that several \QN algorithms satisfy these conditions. For simplicity, assuming that the error is fixed and equal $\delta$, we get the convergence rate $ O(1) \max\lb \ls \tfrac{L_1R^2}{\varepsilon}\rs^{1/2};  \ls\tfrac{MR^3}{\varepsilon}\rs^{1/3}\rb$ iterations for accelerated method and $O(1) \max\lb \tfrac{\delta R^2}{\varepsilon};  \sqrt{\tfrac{MR^3}{\varepsilon}}\rb$ iterations for non-accelerated one.
    \vspace{-0.1 cm}
    
    \item \textbf{Adaptivity.} We propose an adaptive variant of the (Accelerated) Inexact Cubic (Quasi) Newton method. In cases where inexactness is uncontrollable, the proposed algorithms adjust to the error. Additionally, adaptivity gives an understanding of inexactness level, so in the cases when it is possible to control the error adaptive methods allows us to achieve global $O(1) \ls\tfrac{MR^3}{\varepsilon}\rs^{1/3}$ convergence rate. For example, in the empirical risk minimization problem or stochastic optimization, if $\delta$ becomes too large, we can increase the batch size to reduce it \footnote{To the best of our knowledge, in stochastic optimization it is necessary to have Hessian batch size at least $\Tilde{O}(\e^{-2/3})$ to guarantee convergence in $O(\e^{-1/3})$ iterations from the paper of \cite{agafonov2020inexact}. That sample size is purely theoretical, and in practice it can be much smaller. Adaptive inexactness allows us to find the true $\delta$, which leads to smaller batch sizes.}.
\vspace{-0.1 cm}

    \item \textbf{Fast subproblem solution.} 
    We propose the way of solving Cubic Regualized Quasi-Newton method's subproblem via Woodbury matrix identity by~\cite{woodbury1949stability, woodbury1950inverting} and line-search with $O(m^2d + m^2\log \e^{-1})$ complexity, where $m$ is a user-defined memory size. Note, that the corresponding complexity of Cubic Newton is $O(d^3 + d\log \e^{-1})$.
    \vspace{-0.1 cm}
    
    \item \textbf{Competitive Numerical Experiments.} Our experiments show that Cubic \QN methods outperform \QN algorithms such as L-SR1 and L-BFGS in terms of iterations and gradient/Hessian-vector product computations. Additionally, the proposed Cubic \QN methods surpass the exact Cubic Newton method in terms of gradient/Hessian-vector product computations.
\end{itemize}

    \subsection{Organization}
    The rest of the paper is organized as follows. In Section \ref{sec:inexact_cubic}, we introduce a new improved assumption on Hessian inexactness and Adaptive Inexact \CRN method. In the next Section \ref{sec:accelerated_cubic}, we propose Adaptive Accelerated Inexact \CRN method.
    The Section \ref{sec:Quasi-Newton} is dedicated to the various Quasi-Newton appoximations and how to solve the cubic subproblem with \QN approximations. Finally, numerical experiments are provided in Section \ref{sec:experiments}. All proofs, additional experiments, alternative version of Inexact \CRN for wider class of problems can be found in Appendix.

\section{Adaptive Inexact Cubic Regularized Newton method for convex functions}
\label{sec:inexact_cubic}
In this section, we introduce a new method called Adaptive Inexact \CRN method. It serves as the main upper-level method in our approach, offering fast convergence and control over the inexactness of inner information by adaptive nature of the method. The method draws inspiration from the paper by \cite{ghadimi2017second}, as well as its generalization from the paper by \cite{agafonov2020inexact}. The proof technique is also influenced by \cite{nesterov2019inexact, nesterov2021superfast}. For this section, we assume that the function $f(x)$ is convex and has $L_2$-Lipschitz-continuous Hessian.
\\
Let us now introduce a generalized assumption on inexact Hessian approximation. specifically the Hessian inexactness along the given (step) direction.  
\begin{assumption}
\label{ass:direct_inexactness}
For a function $f(x)$ and points $x \in \mathbb{R}^d$ and $y \in \mathbb{R}^d$, a positive semidefinite matrix $B_x \in \mathbb{R}^{d \times d}$ is considered a $\delta_{x}^{y}$-inexact Hessian if it satisfies the inequality
\begin{equation}
\label{eq:delta_norm}
    \|(\nabla^2 f(x) - B_x)(y-x)\| \leq \delta_{x}^{y}\|y-x\|.
\end{equation}
\end{assumption}
It is worth noting that this assumption focuses solely on the inexactness along the direction $y-x$, which represents the step direction. Consequently, the error can be much smaller than $\|\nabla^2 f(x) - B_x\|_2$ from previous papers. More detailed information and comparisons regarding existing types of inexactness are provided in the Appendix.
\\
Now, we can proceed to the formulation of the method. Firstly, we introduce \textit{the exact Taylor approximation} defined as:
\begin{equation*}
    \label{eq:exact_taylor}
    \Phi_{x}(y) \stackrel{\text{def}}{=} f(x) + \la \nabla f(x), y-x\ra + \tfrac{1}{2} \la \nabla^2 f(x) (y-x), y-x \ra,
\end{equation*}
and \textit{the inexact Taylor approximation} as:
\begin{equation*}
    \label{eq:inexact_taylor}
    \phi_{x}(y) \stackrel{\text{def}}{=} f(x) + \la \nabla f(x), y-x\ra + \tfrac{1}{2} \la B_x (y-x), y-x \ra,
\end{equation*}

Next, we introduce the Inexact \CRN operator
\begin{equation}
    \label{eq:inexact_cubic_operator}
    S_{M, \delta}(x) = x + \textstyle{\argmin}_{ h \in \R^d} \left\lbrace f(x)+ \la\nabla f(x),h\ra   + \tfrac{1}{2}\la B_{x} h,h\ra +
    \tfrac{M}{6}\|h\|^3 + \tfrac{\delta}{2}\|h\|^2\right\rbrace,
\end{equation}
where $M\geq 2L_2$. Then, the step of the method is given by $x_{t+1} = S_{M,\delta}(x_{t})$. To exploit its best properties, we propose an adaptive version of the method that adapts to the value of $\delta$.

\begin{algorithm} 
		\caption{Adaptive Inexact Cubic Regularized Newton method} 
        \label{alg:adaptive_cubic}
		\begin{algorithmic}[1]
			\STATE \textbf{Requires:} Initial point $x_0 \in \R^d$, constant $M$ s.t. $M \geq L>0$, initial inexactness $\delta_0$, increase multiplier $\gamma_{inc}$.
			\FOR {$t=0,1,\ldots, T$}
                \STATE $x_{t+1} = S_{M,\delta_t}(x_{t})$
                \WHILE{$\langle \nabla f(x_{t+1}), v_t - x_{t+1} \rangle 
        \leq 
        \min \left\{ \tfrac{\|\nabla f(x_{t+1})\|_{\ast}^2}{4\delta_t }, 
         \tfrac{\|\nabla f(x_{t+1})\|_{\ast}^\frac{3}{2}}{(3M)^\frac{1}{2}}\right\}$}
			     \STATE $\delta_t= \delta_t\gamma_{inc}$
			     \STATE $x_{t+1} = S_{M,\delta_t}(v_{t})$
                \ENDWHILE
			\ENDFOR
            \STATE \textbf{Return} $x_{T+1}$
		\end{algorithmic}
	\end{algorithm}

Finally, we present the convergence theorem for convex functions.

\begin{theorem}
\label{thm:cubic_convergence}
Let $f(x)$ be a convex function with respect to the global minimizer $x^{\ast}$, and let $f(x)$ have an $L_2$-Lipschitz-continuous Hessian. Suppose $B_{t}$ is a $(\delta_x^y)$-inexact Hessian, and $M\geq 2L_2$. The Adaptive Inexact \CRN method performs $T\geq 1$ iterations to find an $\varepsilon$-solution $x_T$ such that $f(x_{T}) - f(x^{\ast})\leq \varepsilon$. The value of $T$ is bounded by: 
\begin{equation*}
    T = O(1) \max\lb \tfrac{\delta_{T}R^2}{\varepsilon};  \sqrt{\tfrac{MR^3}{\varepsilon}}\rb,
\end{equation*}
where $R~=~\max\limits_{x \in \mathcal{L}} \|x - x^{\ast}\|$ represents the diameter of the level set $\mathcal{L}~=~\lb x \in \R^d : f(x) \leq f(x_0) \rb$.
\end{theorem}
\vspace{-5pt}

Let us discuss the results of the theorem. From the proof, one can show that the initial iterations are performed at the rate of \CRN method $O(1)\sqrt{\tfrac{MR^3}{\varepsilon}}$, until the moment when $\delta_t\geq \sqrt{\tfrac{9M\varepsilon}{16R}} $. At that point, the method switches to the convergence rate of the classical Gradient Descent, i.e., $O(1)\tfrac{\delta_{T}R^2}{\varepsilon}$.  
In case, where inexactness can be controlled and the method can improve the Hessian approximation as precise as it needs, it is possible to maintain the convergence rate of the \CRN method.

\begin{corollary}
\label{cor:controllable_inexactness}
The Adaptive Inexact \CRN method with \textbf{Controllable Inexactness}, allowing for the improvement of the approximation $B_t$ up to the desired accuracy, performs $T\geq 1$ iterations to find an $\varepsilon$-solution $x_T$ such that $f(x_{T}) - f(x^{\ast})\leq \varepsilon$. The value of $T$ is bounded by 
$
    T = O(1) \sqrt{\tfrac{MR^3}{\varepsilon}}.
$
\end{corollary}
Therefore, we achieve the same convergence rate as the classical \CRN method but with potentially much cheaper costs of iterations.

\vspace{-5pt}
In summary, we propose the new Adaptive Inexact \CRN method based on new inexactness assumptions. It introduces the flexibility of choosing the approximation $B_{x}$ and controlling its inexactness at each step. All proofs are provided in the Appendix. Additionally, an alternative version of the method for a wider class of functions is presented in the Appendix as well.

\vspace{-0.3 cm}   
\section{Adaptive Accelerated Inexact Cubic Newton for convex functions}
\label{sec:accelerated_cubic}

In this section, we introduce a novel algorithm called the Adaptive Accelerated Inexact \CRN method. Our method draws inspiration from the work of \cite{nesterov2008accelerating} and its extensions for inexact Hessian computations as presented in papers of \cite{ghadimi2017second, agafonov2020inexact}. We enhance these methods by developing an adaptive version that allows us to estimate and control the level of inexactness in the inner step. For this section, we assume that the function $f(x)$ is convex and has $L_2$-Lipschitz-continuous Hessian.
\\
The Nesterov's type acceleration is based on the estimating sequence technique, where we aggregate linear models of the form $l(x,y)=f(y)+\la \nabla f(y), x-y\ra$ within the function 
\begin{equation*}
    \psi_t(x) =  \textstyle{\sum \limits_{i=2}^{3} \tfrac{ \bk^{t}_i}{i}\|x -x_0\|^i + \sum \limits_{j = 0}^{t - 1} \tfrac{\alpha_j}{A_j}l(x,x_{j+1}).}
\end{equation*}
Now, we are ready to present the method and state the convergence theorem. 
\begin{algorithm}
  \caption{Adaptive Accelerated Inexact Cubic Regularized Newton}\label{alg:inexact_acc}
  \begin{algorithmic}[1]
      \STATE \textbf{Input:} $y_0 = x_0$ is starting point; constants $M \geq 2L_2$; increase multiplier $\gamma_{inc}$; 
     starting inexactness $\delta_0  \geq 0$; non-negative non-decreasing sequences $\{\kappa_2^t\}_{t \geq 0}$, $\{\kappa_3^t\}_{t \geq 0}$, $\{\alpha_t\}_{t \geq 0}$, and $\{A_t\}_{t \geq 0}$.
    \FOR{$t \geq 0$} 
        \STATE 
            $v_t = (1 - \alpha_t)x_t + \alpha_t y_t$,
        \STATE $x_{t+1} = S_{M,\delta_t}(v_{t})$
                \WHILE{$\langle \nabla f(x_{t+1}), v_t - x_{t+1} \rangle 
        \leq 
        \min \left\{ \tfrac{\|\nabla f(x_{t+1})\|_{\ast}^2}{4\delta_t }, 
         \tfrac{\|\nabla f(x_{t+1})\|_{\ast}^\frac{3}{2}}{(3M)^\frac{1}{2}}\right\}$}
			     \STATE $\delta_t= \delta_t\gamma_{inc}$
			     \STATE $x_{t+1} = S_{M,\delta_t}(v_{t})$
                \ENDWHILE
        \STATE Compute 
            $
                y_{t+1}=\arg \min _{x \in \R^d} \psi_{t+1}(x)
            $
    \ENDFOR
  \end{algorithmic}
\end{algorithm}

\begin{theorem}\label{thm:acc_convergence}
Let $f(x)$ be a convex function, $f(x)$ has $L_2$-Lipschitz-continuous Hessian, $B_{t}$ is a $\delta_t$-inexact Hessian, and $M\geq 2L_2$. 
Adaptive Accelerated Inexact \CRN method performs $T\geq 1$ iterations with parameters
    \vspace{-5pt}
    \begin{gather*}
        \bk_2^{{t+1}} =  \tfrac{2\delta_t\al_t^2}{A_t}, \quad  \bk_{3}^{t} = \tfrac{8M}{3} \tfrac{\alpha_{t+1}^{3}}{A_{t+1}}, \, \alpha_t = \tfrac{3}{t+3}, \,\, A_t = \tfrac{6}{(t+1)(t+2)(t+3)}
    \end{gather*}
    to find $\varepsilon$-solution $x_T$ such that $f(x_{T}) - f(x^{\ast})\leq \varepsilon$. The value of $T$ is bounded by  
    \begin{equation*}
     T = O(1) \max \lb \ls \tfrac{\delta_{T} R^2}{\varepsilon}\rs^{1/2};  \ls\tfrac{MR^3}{\varepsilon}\rs^{1/3}\rb, \text{where }R= \|x_0-x^{\ast}\|.
    \end{equation*}
\end{theorem}
Note, that it is a shortened version of the method, the full method and all proofs are provided in the Appendix. Similar to the previous section, if we have the ability to control the level of inexactness and improve the accuracy of the Hessian approximation as needed, it becomes possible to match the convergence rate of the Accelerated \CRN method with possibly much cheaper computational costs of iterations.
\begin{corollary}
The Adaptive Accelerated Inexact \CRN method with \textbf{Controllable Inexactness}, allowing for the improvement of the approximation $B_t$ up to the desired accuracy, performs $T\geq 1$ iterations to find an $\varepsilon$-solution $x_T$ such that $f(x_{T}) - f(x^{\ast})\leq \varepsilon$. The value of $T$ is bounded by 
$
    T = O(1) \ls\tfrac{MR^3}{\varepsilon}\rs^{\frac{1}{3}}.
$
\end{corollary}

\section{Quasi-Newton(QN) Approximation}\label{sec:Quasi-Newton}
In this section, we propose an approach for creating an inexact Hessian by \QN approximations.
The main idea is straightforward: we calculate $B_x$ as a Quasi-Newton approximation and  use it for the step $y = S_{M,\delta}(x)$. 
Firstly, we discuss various \QN approximations with low-rank $B_t$.
We assume that the approximation $B_t$ takes the form:
\begin{equation}
\label{eq:low_rank}
    B_t = B^m_t = B^0 + \sum_{i=0}^{m-1} \xi_i u_i u_i^T + \sum^{m-1}_{i=0} \beta_i v_i v_i^T.
\end{equation}
\textbf{L-BFGS} is one of the most popular and effective \QN approximations, and it can be expressed as follows:
\begin{equation}
\label{eq:l-bfgs}
B^{m+1}_{t} = B^{m}_{t} +  \tfrac{y_m y_m^{\top}}{y_m^{\top} s_m} -  \tfrac{B_t^m s_m (B^m_t s_m)^{\top}}{s_m^{\top} B_t^m s_m}.
\end{equation}
In Equation \eqref{eq:l-bfgs}, we have $u_i = y_i$, $v_i = B^i_t s_i$, $\xi_i= \frac{1}{y_i^{\top} s_i}$, and $\beta_i= -\frac{1}{s_i^{\top} B_t^i s_i}$. For memory-size $m$, L-BFGS is $2m$-rank update.
\\
 Note, that for \eqref{eq:l-bfgs}, we can define matrix $B_t^{m}(Y,S)$, where $Y$ and $S$ are stacked set of vectors $Y = [y_0, \ldots, y_{m_1}]$ and $S = [s_0, \ldots, s_{m_1}]$. This allows us to calculate the matrix $B_t^{m}(Y, S)$ for any given $Y$ and $S$.
\\
 Now, we focus on the choice of $Y$ and $S$ for the \QN approximation. There are two main variants:
 \\
 \textbf{\QN with history.} This variant is well-known and classic. It involves the following update rules:
 \begin{equation*}
     \label{eq:history}
     s_i = z_{i+1}-z_{i},\qquad
     y_i = \nabla f(z_{i+1}) - \nabla f(z_{i}).
 \end{equation*}
 This approach is computationally efficient as it does not require additional gradient calculations. However, its main drawback is that it cannot increase the accuracy at the current point. In the paper by \cite{berahas2021quasi}, it was shown that to ensure $\delta_{T}\leq L_1$, we need to divide $\xi_i$ and $\beta_i$ by $m$.
 \\
\textbf{\QN with sampling.}
This variant is based on fast computation of Hessian-vector products (HVP):
 \begin{equation*}
     \label{eq:sample}
     y_i = \nabla^2 f(x_{t})s_i,
 \end{equation*}
where $s_i$ is a vector from $1$-sphere such that $\|s_i\|=1$. For $m<<d$, each $s_i$ is linearly independent with high probability.  This variant requires $m$ gradient/HVP computations per step, which is significantly less than computing the full Hessian ($d$ HVPs). The advantage of using information from the current Hessian is that it provides more control over inexactness. We can increase memory to sample a more precise approximation. For the sampled $B_t^m$, it is possible to guarantee that $\delta_{T}\leq L_1$. Further details and proofs are provided in the Appendix.
\\
In summary, we have two different policies for choosing $Y$ and $S$. The \QN with history is the most computationally efficient but cannot improve the accuracy of the Hessian. \QN with sampling costs $m$ computations of HVPs and can be used to increase the accuracy of the approximation. It is a good idea to combine these policies and use history as a basic approximation and sampling to increase the accuracy of the Hessian approximation.
\\
Finally, we discuss how to efficiently solve the subproblem \eqref{eq:inexact_cubic_operator} $x_{t+1} = S_{M,\delta}(x_t)$ with low-rank $B_t$ from \eqref{eq:low_rank}.
    
\begin{equation*}
x_{t+1} = x_{t} + \argmin\limits_{ h \in \E} \lb \la\nabla f(x_t),h\ra + \frac{1}{2}\la B_{t} h,h\ra+ \frac{M}{6}\|h\|^3 + \frac{\delta}{2}\|h\|^2\rb.
\end{equation*}
 
The subproblem's first derivative with regard to $h$:
\begin{equation*}
    \nabla f(x_t) + (B_t + \delta I )h^{\ast} + \frac{M}{2}\|h^{\ast}\| h^{\ast} = 0.
\end{equation*} 
The solution to the subproblem can be obtained as follows:
\begin{equation*}
\label{eq:cubic_subproblem_solution}
   h^{\ast} = -\ls B_t + \delta I  + \tfrac{M}{2}\|h^{\ast}\|I\rs^{-1}\nabla f(x).
\end{equation*}
To find $\|h^{\ast}\|$, we formulate and solve the one-dimensional dual problem using ray-search:
\begin{equation*}
 \|h^{\ast}\| = \arg \min_{\tau}\left\lbrace \la\left(B_t + \delta I + \tfrac{M}{2}\tau I\right)^{-1}\nabla f(x_t),\nabla f(x_t)\ra+ \tfrac{M}{6}\tau^2\right\rbrace,
\end{equation*}
Since $B_t$ is a low-rank matrix with a sum structure, we can effectively invert it using the Woodbury matrix identity: 
\begin{equation*}
(A + UCV)^{-1} = A^{-1} - A^{-1}U(C^{-1} + V A^{-1}U)^{-1}V A^{-1}.
\end{equation*}

Thus, the inversion takes $O(m^3)$ instead of $O(d^3)$ as for \CRN. By applying SVD to $VA^{-1}U$, the multiple inversions for different $\tau$ would cost only $O(m^3 + m^2 \log(\varepsilon^{-1}))$ instead of $O(m^3 \log(\varepsilon^{-1}))$ as before. Therefore, the total computational cost of ray-search procedure for low-rank matrix $B_t$ is $O(m^2d + m^2 \log(\varepsilon^{-1}))$. It is worth noting that for Cubic Regularized Newton method, it typically takes $O(d^3 + d \log(\varepsilon^{-1}))$
More details regarding the computational complexity of inversion and the solution to the subproblem can be found in the Appendix.
\vspace{-5pt}

We summarize the convergence of Cubic \QN Methods and Accelerated Cubic \QN methods in the following theorems.
\begin{theorem}
\label{thm:cubic_bfgs}
Let $f(x)$ be a convex function with respect to the global minimizer $x^{\ast}$, and let $f(x)$ have an $L_1$-Lipschitz-continuous gradient and an $L_2$-Lipschitz-continuous Hessian. Suppose $B_{t}$ is a $m$-memory L-BFGS approximation, and $M\geq 2L_2$. The Adaptive Inexact Cubic Quasi-Newton method performs $T\geq 1$ iterations to find an $\varepsilon$-solution $x_T$ such that $f(x_{T}) - f(x^{\ast})\leq \varepsilon$. The value of $T$ is bounded by: 
\begin{equation*}
    T = O(1) \max\lb \tfrac{L_1R^2}{\varepsilon};  \sqrt{\tfrac{MR^3}{\varepsilon}}\rb,
\end{equation*}
where $R~=~\max\limits_{x \in \mathcal{L}} \|x - x^{\ast}\|$ represents the diameter of the level set $\mathcal{L}~=~\lb x \in \R^d : f(x) \leq f(x_0) \rb$.
\end{theorem}

\begin{theorem}
\label{thm:accelerated_cubic_bfgs}
Let $f(x)$ be a convex function with respect to the global minimizer $x^{\ast}$, and let $f(x)$ have an $L_1$-Lipschitz-continuous gradient and an $L_2$-Lipschitz-continuous Hessian. Suppose $B_{t}$ is a $m$-memory L-BFGS approximation, and $M\geq 2L_2$. The Adaptive Inexact Cubic Quasi-Newton method performs $T\geq 1$ iterations to find an $\varepsilon$-solution $x_T$ such that $f(x_{T}) - f(x^{\ast})\leq \varepsilon$. The value of $T$ is bounded by: 
\begin{equation*}
    T = O(1) \max\lb \ls \tfrac{L_1R^2}{\varepsilon}\rs^{1/2};  \ls\tfrac{MR^3}{\varepsilon}\rs^{1/3}\rb, \text{where }R= \|x_0-x^{\ast}\|.
    \end{equation*}
\end{theorem}
    
\section{Experiments}\label{sec:experiments}

In this section, we present numerical experiments conducted to demonstrate the efficiency of our proposed methods. We consider $l_2$-regularized logistic regression problems of the form:
\begin{equation*}
    f(x)= \tfrac{1}{n}{\textstyle \sum_{i=1}^n} \log(1+ \exp(-b_i a_i^{\top}x)) + \tfrac{\mu}{2}\|x\|^2,
\end{equation*}
where $\{(a_i,b_i)\}_{i=1}^n$ are the training examples described by features $a_i$ and class labels $b_i \in \{-1,1\}$, and $\mu \geq 0$ is the regularization parameter.

\vspace{-0.3 cm}
\paragraph{Setup.}
We present results on the \texttt{MNIST} dataset ($d=784$) by \cite{lecun1998mnist}.
We compare the performance of the history and sampled Adaptive Cubic \QN method, sampled Accelerated Adaptive Cubic LBFGS, Gradient Descent (GD), \CRN method, and classical \QN methods (LBFGS, LSR1). For all \QN algorithms, we set the memory size $m=10$. We present experiments for theoretical hyperparameters in the convex case with $\mu=0$ in Figure~\ref{fig:mnist_theo_cv} and in strongly convex case with $\mu = 10^{-4}$ in Figure~\ref{fig:mnist_theo}.
To demonstrate the globalization properties of the methods, we consider the starting point $x_0$ far from the solution, specifically $x_0 = 1\cdot e $, where $e$ is the all-one vector. The classical Newton method diverges from that point. 
Additional experiments with different hyperparameters are provided in the Appendix, including datasets  \texttt{a9a} ($d=123$), \texttt{real-sim} ($d=20958$), \texttt{gisette} ($d=5000$) and  \texttt{CIFAR-10} ($d=1024$) datasets by \cite{chang2011libsvm}.

\vskip-10pt
\begin{figure}
  \centering
  \includegraphics[width=0.45\linewidth]{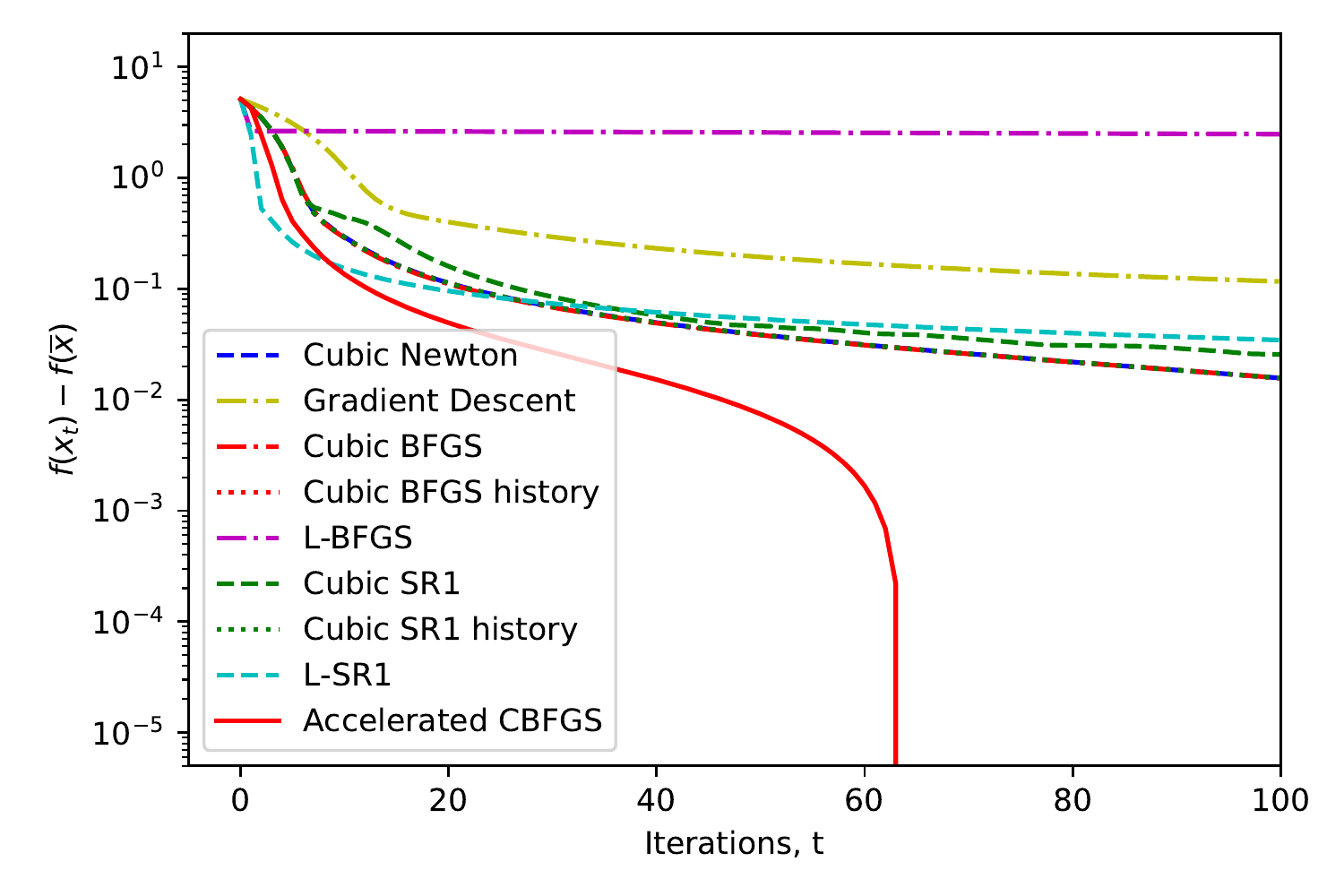}
  \includegraphics[width=0.45\linewidth]{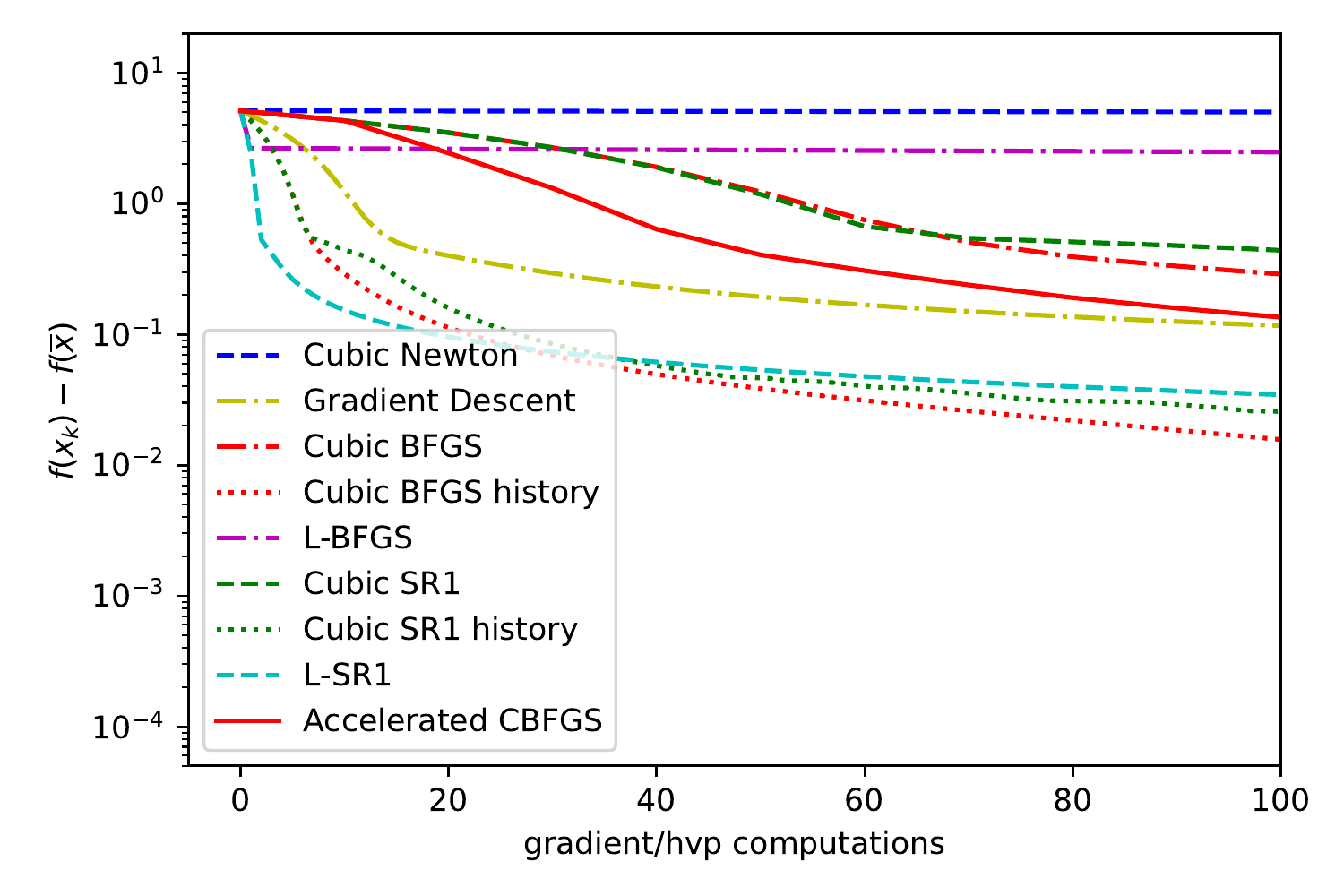}
  \caption{Comparison of \QN methods and \CRN methods for \texttt{MNIST} dataset using theoretical parameters in the convex case.}
  \label{fig:mnist_theo_cv}
\end{figure}
\begin{figure}
\vspace{-0.35cm}
  \centering
  \includegraphics[width=0.45\linewidth]{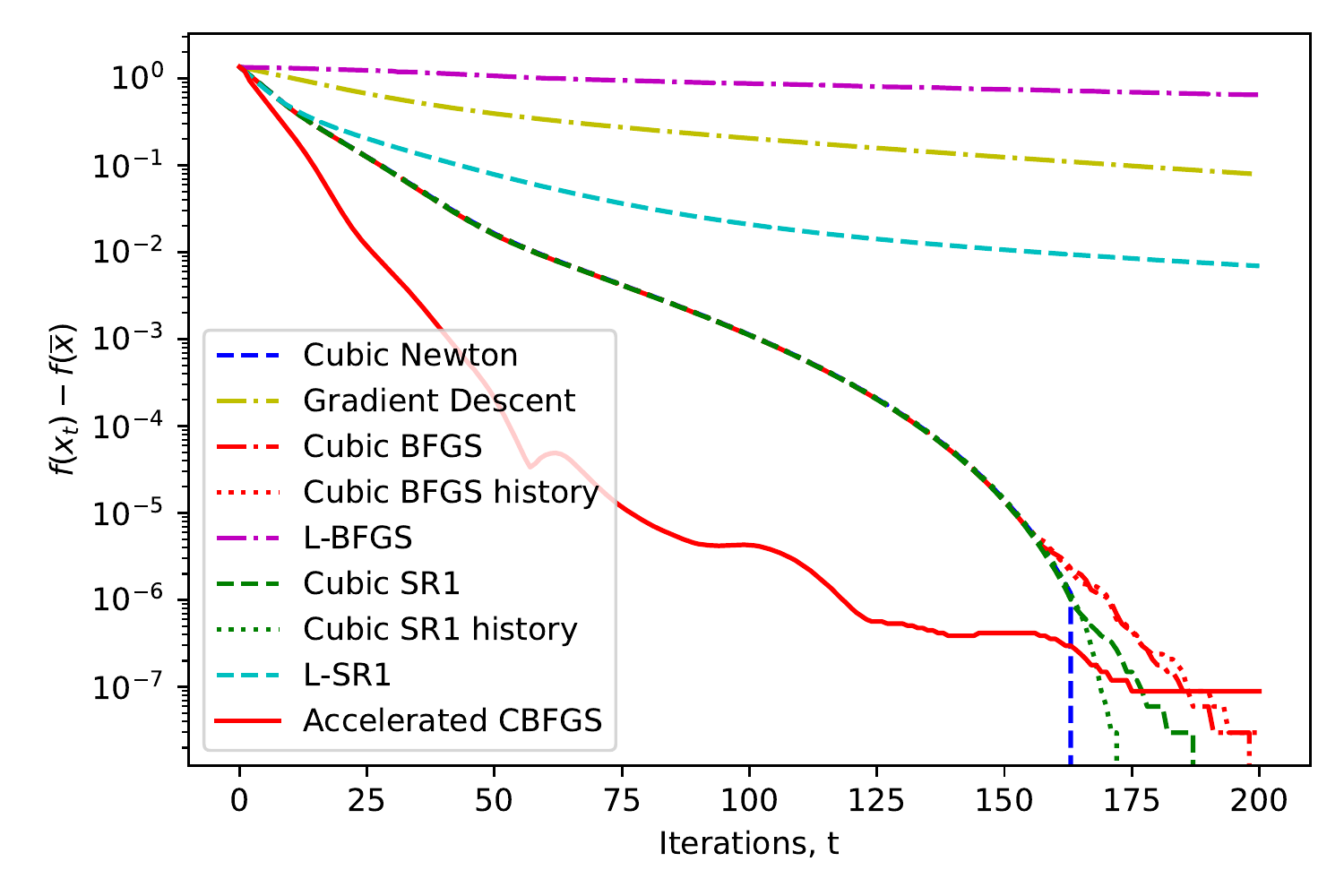}
  \includegraphics[width=0.45\linewidth]{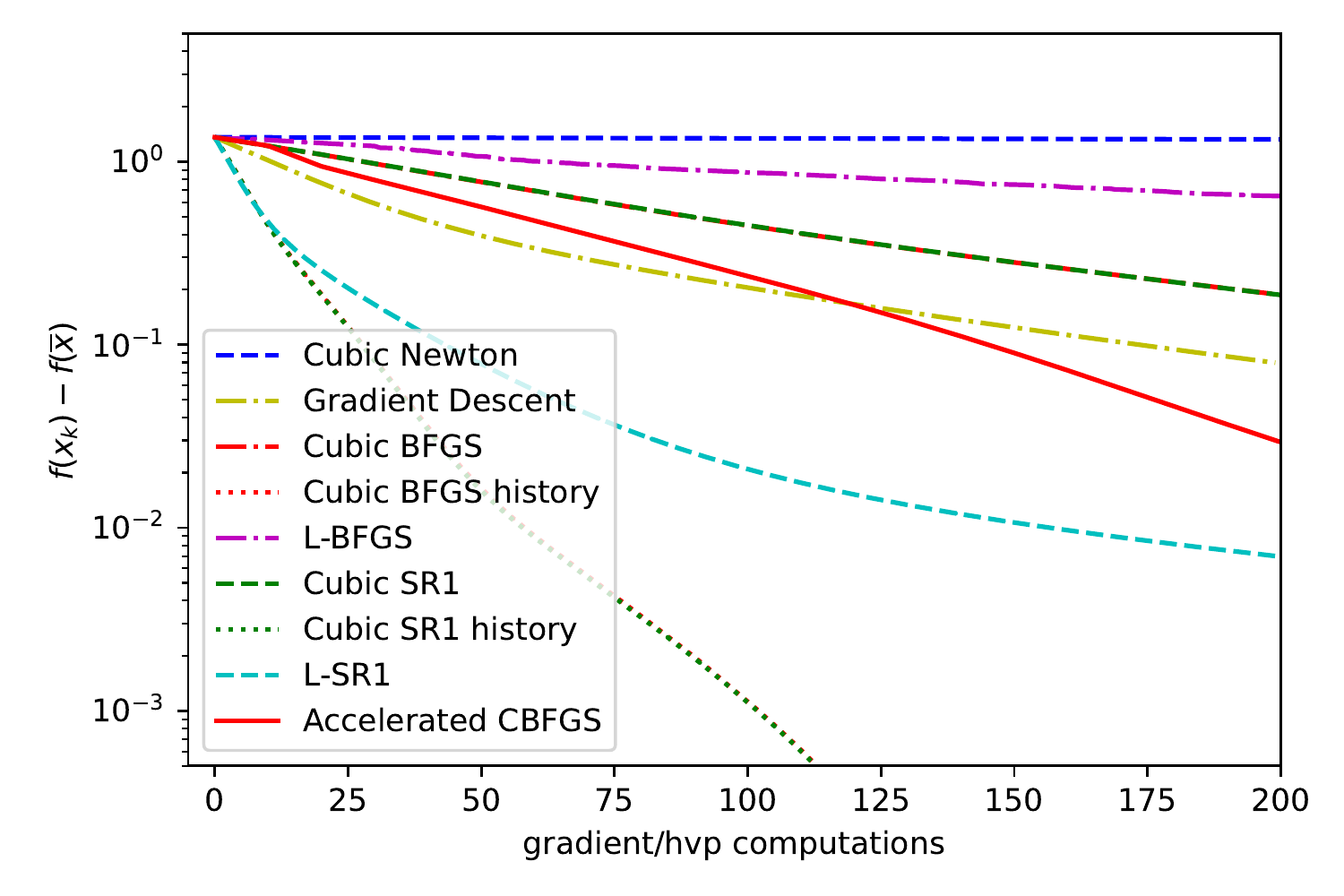}

  \caption{Comparison of \QN methods and \CRN methods for \texttt{MNIST} dataset using theoretical parameters in strongly convex case.}
  \label{fig:mnist_theo}
\vspace{-0.2cm}
\end{figure}

\paragraph{Results.}
Proposed Cubic \QN methods outperform classical non-regularized \QN methods, which suffer from the starting point far from the solution.  For convex case in Figure~, Cubic \QN iterations almost exactly match \CRN iterations. However, \CRN requires $d$ HVP per iteration, which makes it the slowest method in terms of gradient/HVP computations. For strongly convex case in Figure~\ref{fig:mnist_theo}, one can see, that convergence of the Cubic \QN method matches the convergence of exact \CRN until the very last iterations. This means that the discrepancies between the methods begin only in the area of quadratic convergence, that hard to match. We highlight that this behavior perfectly supports the theory, that Cubic \QN has the same convergence rate as \CRN on the initial iterations.
Sample methods perform slightly better in terms of iteration, but are much worse in gradient/hessian-vector product computations. Indeed, sampled methods require extra $m$ Hessian-vector products, which makes iteration more expensive. We suppose that the combination of these two approaches might be better in practice. At first, one can use cheap history \QN updates and sample HVP for a more precise approximation.
Figures~\ref{fig:mnist_theo_cv},\ref{fig:mnist_theo} show that Accelerated \QN is much faster than non-accelerated methods.

\subsection*{Acknowledgement}
The authors are grateful to Alexander Gasnikov and Pavel Dvurechensky for valuable comments and discussions. 

\bibliographystyle{icml2023}
\bibliography{bibliography, kamzolov_bib}


\newpage
\appendix

\section{Proofs of Section \ref{sec:inexact_cubic}}
In this Section, we present proofs for the Inxact \CRN method. We begin from two helpful lemmas on closeness of the function $f(x)$ to inexact Taylor approximation $\phi_{x}(y)$.
\begin{lemma}
\label{lem:cubic_bound_acc}
For the function $f(x)$ with $L_2$-Lipschitz-continuous Hessian and $B_x$ is $\delta_x^y$-inexact Hessian for $x,y \in \R^d$, we have
\begin{equation}
\label{eq:grad_model_bound}
    \|\nabla \phi_{x}(y) - \nabla f(y)\| 
        \leq
        \delta_x^y \|y - x\| + \frac{L_2}{2}\|y - x\|^{2}.
\end{equation}
\end{lemma}
\begin{proof}
    \begin{align*}
         \|\nabla \phi_{x}(y) - \nabla f(y)\| &= \|\nabla \phi_{x}(y) - \nabla \Phi_{x}(y) +\nabla \Phi_{x}(y)  - \nabla f(y)\| \\
         &\leq \|\nabla \phi_{x}(y) - \nabla\Phi_{x}(y)\| +\|\nabla\Phi_{x}(y)  - \nabla f(y)\| \\
         &=\| (\nabla^2 f(x) - B_{x})(x- x)\| +\|\Phi_{x}(y)  - \nabla f(y)\| \\
            &\stackrel{\eqref{eq:delta_norm}}{\leq} \delta_x^y \|y - x\| + \frac{L_2}{2}\|y - x\|^{2}
    \end{align*}
\end{proof}

\begin{lemma}\label{lem:monotonicity}
    For the function $f(x)$ with $L_2$-Lipschitz-continuous Hessian, for the Inexact \CRN operator \eqref{eq:inexact_cubic_operator}
    \begin{equation}
    \label{eq:RCN_subproblem_appendix}
        y = \argmin_{ z \in \R^d} \left\lbrace f(x)+ \la\nabla f(x),z-x\ra   + \tfrac{1}{2}\la B_{x} (z-x),z-x\ra + \tfrac{M}{6}\|z-x\|^3 + \tfrac{\delta}{2}\|z-x\|^2\right\rbrace
    \end{equation}
    with $B_x$ as $\delta_x^y$-inexact Hessian 
    \begin{equation}
    \label{eq:subproblem_globalization}
        |f(y) - \phi_{x}(y)| \leq \frac{L_2}{6}\|y-x\|^3 + \frac{\delta_x^{y}}{2}\|y-x\|^2
    \end{equation}
    The Inexact \CRN step $y=S_{M,\delta}(x)$ with $\delta\geq \delta_x^y$, and $M\geq L_2$ is  monotone.
        \begin{equation}
        \label{eq:monotone}
        f(y)\leq f(x)
    \end{equation}
\end{lemma}
\begin{proof}
We start with the proof of \eqref{eq:subproblem_globalization}:
    \begin{equation*}
\begin{gathered}
    |f(y) - \phi_{x}(y)| \leq 
    |f(y) - \Phi_{x}(y) + \Phi_{x}(y) - \phi_{x}(y) | \leq \frac{L_2}{6}\|y-x\|^3 + |\Phi_{x}(y) - \phi_{x}(y)| \\
    \leq \frac{L_2}{6}\|y-x\|^3 + \tfrac{1}{2} \la (\nabla^2 f(x) - B_x) (y-x), y-x \ra \\
    \leq \frac{L_2}{6}\|y-x\|^3 + \tfrac{1}{2} \left\| (\nabla^2 f(x) - B_x) (y-x)\right\| \|y-x\|\\
    \stackrel{\eqref{eq:delta_norm}}{\leq} 
    \frac{L_2}{6}\|y-x\|^3 + \frac{\delta_x^{y}}{2}\|y-x\|^2.
\end{gathered}
\end{equation*}
Hence, we get an upper-bound for function $f(y)$. This allows us to show that the step $y=S_{M,\delta}(x)$ is monotone.
    \begin{equation*}
\begin{aligned}
    f(y) &\leq \phi_{x}(y) + \frac{L_2}{6}\|y-x\|^3 + \frac{\delta_x^{y}}{2}\|y-x\|^2 \\
    &\leq
    \phi_{x}(y)+ \frac{M}{6}\|y-x\|^3 + \frac{\delta}{2}\|y-x\|^2 \\
 &\stackrel{\eqref{eq:RCN_subproblem_appendix}}{\leq}  \phi_{x}(z)+ \frac{M}{6}\|z-x\|^3 + \frac{\delta}{2}\|z-x\|^2
    \stackrel{z:=x}{\leq}  f(x).
\end{aligned}
\end{equation*}
The last two inequalities hold because $y$ is a minimum of the subproblem \eqref{eq:RCN_subproblem_appendix}.
Finally, we proved that the Inexact \CRN step is monotone.
\end{proof}

The next Lemma characterizes the progress of the inexact cubic step $y=S_{M,\delta}(x)$ from \eqref{eq:inexact_cubic_operator}.

\begin{lemma}\label{lem:inexact_cubic_step_improvement}
    For the convex function $f(x)$ with $L_2$-Lipschitz-continuous Hessian, for the Inexact \CRN operator \eqref{eq:RCN_subproblem_appendix}
    with $B_x$ as $\delta_x^y$-inexact Hessian, $\delta\geq \delta_x^y$, and $M\geq 2L_2$, the following holds
        \begin{equation}
        \label{eq:cubic_step_improvement}
        \langle \nabla f(y), x - y \rangle 
        \geq 
         \min \left\{ \|\nabla f(y)\|^2\left( \tfrac{1}{4\delta }\right), \|\nabla f(y)\|^\frac{3}{2}
        \left( \tfrac{1}{3M}\right)^\frac{1}{2}\right\}.
    \end{equation}
\end{lemma}
\begin{proof}
    For simplicity, we denote $r = \|y - x\|$ and
    \begin{equation}
        \label{big_sum_lemma6}
        \zeta =  \delta + \frac{M}{2}\|y - x\|.
    \end{equation}
    
    By the optimality condition for $y = S_{M,\delta}(x)$ and $\phi_x{y}$ from 
 $\eqref{eq:inexact_taylor}$, we get
    
    \begin{equation}
   \label{eq:opt_cnd_acc}
    \begin{gathered}
    0 = \nabla \phi_{x}(y)+ \delta (y - x)  + \frac{M}{2}\|y - x\|(y - x) \\ \stackrel{\eqref{big_sum_lemma6}}{=} \nabla \phi_{x}(y) + \zeta (y - x).
    \end{gathered}
    \end{equation}
    
Next, we consider $2$ cases depending on which term dominates in the $\zeta$. 

\begin{itemize}
    \item If $\delta \geq \frac{M}{2}\|y-x\|$, then we get the following result.
    We start with getting an upper bound for $\|\nabla \phi_{x}(y) - \nabla f(y)\|$. 
    \begin{equation*}
        \|\nabla \phi_{x}(y) - \nabla f(y)\| 
        \stackrel{\eqref{eq:grad_model_bound}}{\leq} 
         \delta_x^y \|y - x\| + \frac{L_2}{2}\|y - x\|^{2}
         \leq \zeta\|y - x\|.
    \end{equation*}
   Next, from the previous inequality and optimality condition \eqref{eq:opt_cnd_acc}, we get
    \begin{gather*}
        \zeta^2\|y - x\|^2 \geq \|\nabla \phi_{x}(y) - \nabla f(y)\|^2 
        \stackrel{\eqref{eq:opt_cnd_acc}}{=} \left\|\nabla  f(y) +  \zeta (y - x) \right\|^2 \\
        = 2 \langle \nabla f(y), y - x \rangle \zeta + \|\nabla f(y) \|^2  + \zeta^2\|y - x\|^2.
    \end{gather*}
    Hence, 
    \begin{gather*}
        \langle \nabla f(y), x - y \rangle 
        \geq
        \frac{1}{2\zeta}\|\nabla f(y) \|^2\geq \frac{1}{4\delta}\|\nabla f(y) \|^2.
    \end{gather*}
    
    \item If $\delta < \frac{M}{2}\|y-x\|$ dominates the others, then similarly to previous case, we get
    \begin{gather*}
        \ls \delta_x^y + \tfrac{L_2}{2}r \rs^2 r^2 \geq \|\nabla \phi_{x}(y) - \nabla f(y)\|^2 
        \stackrel{\eqref{eq:opt_cnd_acc}}{=} \left\|\nabla  f(y) +  \zeta (y - x) \right\|^2 \\
        = 2 \langle \nabla f(y), y - x \rangle \zeta + \|\nabla f(y) \|^2  + \zeta^2 r^2.
    \end{gather*}
    As a result, we get
    \begin{gather*}
        \langle \nabla f(y), x - y \rangle 
        \geq
        \frac{\|\nabla f(y) \|^2}{2\zeta} + 
       \ls\ls \delta + \tfrac{M}{2}r\rs^2- \ls \delta_x^y + \tfrac{L_2}{2}r\rs^2 \rs \frac{r^2}{2\zeta}\\
       = \frac{\|\nabla f(y) \|^2}{2\zeta}
        + \ls \delta - \delta_x^y + \tfrac{M-L_2}{2}r\rs \ls \delta_x^y + \delta + \tfrac{L_2+M}{2}r\rs  \frac{r^2}{2\zeta}\\
        \geq \frac{\|\nabla f(y) \|^2}{2Mr}
        + \frac{M^2-L_2^2}{4} \frac{r^3}{2M}
        \geq \frac{\|\nabla f(y) \|^2}{2Mr}
        + \frac{3M}{32} r^3 
        \geq \ls\frac{1}{3M}\rs^{\frac{1}{2}}\|\nabla f(y) \|^{\frac{3}{2}},
    \end{gather*}
    where for the last inequality, we use $\tfrac{\alpha}{r} + \tfrac{\beta r^3}{3} \geq \frac{4}{3}\beta^{1/4}\al^{3/4}$.
\end{itemize} 
\end{proof}
Finally, we present the convergence theorem. It is an extended version of Theorem \ref{thm:cubic_convergence}.
\begin{theorem}
Let $f(x)$ be a convex function with respect to the global minimizer $x^{\ast}$, and let $f(x)$ have an $L_2$-Lipschitz-continuous Hessian. Suppose $B_{t}$ is a $(\delta_{x_t}^{x_{t+1}})$-inexact Hessian, and $M\geq 2L_2$, $\delta_{x_t}^{x_{t+1}}\leq \delta_{max}$. The Adaptive Inexact \CRN method performs $T\geq 1$ iterations to find an $\varepsilon$-solution $x_T$ such that $f(x_{T}) - f(x^{\ast})\leq \varepsilon$. The first $T_1$ iterations are performed with \CRN rate while
\begin{equation}
\label{eq:switch-moment-cubic}
    \frac{3M \ls f(x_{t+1}) - f(x^{\ast}) \rs}{16R} \geq \delta_t^2.
\end{equation}
The last $T_2$ iterations are performed with gradient rate, where $T_1+T_2 = T$.
Note, that $f(x_{t+1})$ monotonically decrease and $\delta_t$ monotonically increase, hence the switch happens only once.
The total value of $T$ is bounded by: 
\begin{equation*}
    T = O(1) \max\lb \tfrac{\delta_{T}R^2}{\varepsilon};  \sqrt{\tfrac{MR^3}{\varepsilon}}\rb,
\end{equation*}
where $R~=~\max\limits_{x \in \mathcal{L}} \|x - x^{\ast}\|$ represents the diameter of the level set $\mathcal{L}~=~\lb x \in \R^d : f(x) \leq f(x_0) \rb$.
\end{theorem}
\begin{proof}
    First, from Lemma \ref{lem:inexact_cubic_step_improvement} and Equation \eqref{eq:cubic_step_improvement}, the condition on Line $4$ of Algorithm \ref{alg:adaptive_cubic} is reachable and the line-search has a finite number of iterations bounded by $\log_{\gamma_{inc}}\ls\frac{\delta_{max}}{\delta_0}\rs$. Hence, by Equation \eqref{eq:cubic_step_improvement} and condition on Line $4$ of Algorithm \ref{alg:adaptive_cubic},we get the next bound for each step:
    \begin{equation}
    \label{eq:hehe31}
    \langle \nabla f(x_{t+1}), x_{t} - x_{t+1} \rangle 
        \geq 
         \min \left\{ \|\nabla f(x_{t+1})\|^2\left( \tfrac{1}{4\delta_t }\right), \|\nabla f(x_{t+1})\|^\frac{3}{2}
        \left( \tfrac{1}{3M}\right)^\frac{1}{2}\right\}.
    \end{equation}
By convexity, we get 
\begin{equation}
\label{eq:cubic_mono}
\begin{aligned}
    f(x_t)-f(x_{t+1}) &\geq\langle \nabla f(x_{t+1}), x_{t} - x_{t+1} \rangle\\
        &\geq 
         \min \left\{ \|\nabla f(x_{t+1})\|^2\left( \tfrac{1}{4\delta_t }\right), \|\nabla f(x_{t+1})\|^\frac{3}{2}
        \left( \tfrac{1}{3M}\right)^\frac{1}{2}\right\}
        \geq 0.
\end{aligned}
\end{equation}
It proofs that the convergence is monotone.

Furthermore, by convexity, we get
\begin{align*}
    f(x^{\ast})&\geq f(x_{t+1}) + \langle \nabla f(x_{t+1}), x^{\ast} - x_{t+1} \rangle\\ &\geq f(x_{t+1}) - \| \nabla f(x_{t+1})\| \cdot \| x^{\ast} - x_{t+1} \|.
\end{align*}
Hence
\begin{equation}
\label{eq:hehe2}
\| \nabla f(x_{t+1})\| \geq \frac{f(x_{t+1}) - f(x^{\ast})}{\| x^{\ast} - x_{t+1} \|} \geq \frac{f(x_{t+1}) - f(x^{\ast})}{R},
\end{equation}
where in the last inequality we used the fact that $\|x_{t+1}-x^{\ast}\| \leq R$, because $R~=~\max\limits_{x \in \mathcal{L}} \|x - x^{\ast}\|$ represents the diameter of the level set $\mathcal{L}~=~\lb x \in \R^d : f(x) \leq f(x_0) \rb$. 

By uniting \eqref{eq:cubic_mono} and \eqref{eq:hoho}, we show that 
\begin{equation}
\label{eq:hoho}
\begin{aligned}
    f(x_t)-f(x_{t+1}) &\geq\langle \nabla f(x_{t+1}), x_{t} - x_{t+1} \rangle\\
        &\geq 
         \min \left\{ \|\nabla f(x_{t+1})\|^2\left( \tfrac{1}{4\delta_t }\right), \|\nabla f(x_{t+1})\|^\frac{3}{2}
        \left( \tfrac{1}{3M}\right)^\frac{1}{2}\right\}\\
        &\geq 
         \min \left\{ \ls\frac{f(x_{t+1}) - f(x^{\ast})}{R}\rs^2\left( \tfrac{1}{4\delta_t }\right), \ls \frac{f(x_{t+1}) - f(x^{\ast})}{R}\rs^\frac{3}{2}
        \left( \tfrac{1}{3M}\right)^\frac{1}{2}\right\}
        \end{aligned}
    \end{equation}
Now, we get to cases depending on $\delta_t$ and $f(x_{t+1})-f(x^{\ast})$ values.

 \textbf{Case 1.} \textit{The first $T_1$ iterations.}  For all $T_1\geq t\geq 0$, we have the next condition satisfied 
    \begin{equation*}
        \frac{3M \ls f(x_{t+1}) - f(x^{\ast}) \rs}{16R} \geq \delta_t^2.    
    \end{equation*}
When we apply the condition to \eqref{eq:hoho}, we get:
\begin{align*}
    f(x_t)-f(x_{t+1}) 
        &\geq 
         \min \left\{ \ls\frac{f(x_{t+1}) - f(x^{\ast})}{R}\rs^2\left( \tfrac{1}{4\delta_t }\right), \ls \frac{f(x_{t+1}) - f(x^{\ast})}{R}\rs^\frac{3}{2}
        \left( \tfrac{1}{3M}\right)^\frac{1}{2}\right\}\\
        &\geq 
          \ls \frac{f(x_{t+1}) - f(x^{\ast})}{R}\rs^\frac{3}{2}
        \left( \tfrac{1}{3M}\right)^\frac{1}{2}.
 \end{align*}
So, for the first $T_1$ iterations, we have \CRN rate.
\begin{gather*}
    f(x_t)-f(x_{t+1}) \geq \ls \frac{f(x_{t+1}) - f(x^{\ast})}{R}\rs^\frac{3}{2}
        \left( \tfrac{1}{3M}\right)^\frac{1}{2}\\
        \ls f(x_t) - f(x^{\ast}) \rs - \ls f(x_{t+1}) - f(x^{\ast}) \rs \geq \ls \frac{f(x_{t+1}) - f(x^{\ast})}{R}\rs^\frac{3}{2}
        \left( \tfrac{1}{3M}\right)^\frac{1}{2}\\
        \frac{f(x_t) - f(x^{\ast})}{3MR^3}  - \frac{f(x_{t+1}) - f(x^{\ast})}{3MR^3} \geq \ls \frac{f(x_{t+1}) - f(x^{\ast})}{3MR^3}\rs^\frac{3}{2}\\
        \xi_t - \xi_{t+1}\geq \xi_{t+1}^{1+\frac{1}{2}},
\end{gather*}
where $\xi_t = \frac{f(x_t) - f(x^{\ast})}{3MR^3}$.

Lemma A.1 from \cite{nesterov2019inexact} guarantee that if we have the following condition for $\xi_t$: 
\begin{equation}
\label{eq:nesterov_sequence_trick}
    \xi_t - \xi_{t+1}\geq \xi_{t+1}^{1+\al}
\end{equation}
then
\begin{equation*}
    \xi_{t}\leq \left[ \ls 1 + \frac{1}{\al} \rs \ls (1 +\xi_0^\al \rs \cdot \frac{1}{t}\right]^{\frac{1}{\al}}.
\end{equation*}

By applying it to our case with $\al = \frac{1}{2}$, we get 
    \begin{gather*}
    \xi_{t+1}\leq \left[ 3  \ls 1 +\ls\frac{f(x_1) - f(x^{\ast})}{3MR^3}\rs^{\frac{1}{2}} \rs \cdot \frac{1}{t}\right]^{2}\\
    \frac{f(x_{t+1}) - f(x^{\ast})}{3MR^3} \leq  18  \ls 1 +\ls\frac{f(x_1) - f(x^{\ast})}{3MR^3}\rs \rs \cdot \frac{1}{t^2}\\
    f(x_{t+1}) - f(x^{\ast}) \leq   \ls 3MR^3 +\ls f(x_1) - f(x^{\ast})\rs \rs \cdot \frac{18}{t^2}
\end{gather*}
By \eqref{eq:hehe2} and \eqref{eq:hehe31}, one can show that $f(x_1) - f(x^{\ast}) \leq 3MR^3$, hence 
\begin{gather*}
    f(x_{T_1+1}) - f(x^{\ast}) \leq  \frac{108MR^3}{t^2}
\end{gather*}
or
\begin{equation}
    \label{eq:final_case1_cubic}
    T_1 = O\ls \sqrt{\frac{MR^3}{\varepsilon}}\rs 
\end{equation}

\textbf{Case 2.} \textit{The last $T_2$ iterations.}  For all $T_2\geq t\geq T_1\geq0$, we have the next condition satisfied 
    \begin{equation*}
        \frac{3M \ls f(x_{t+1}) - f(x^{\ast}) \rs}{16R} \leq \delta_t^2.  
    \end{equation*}
When we apply the condition to \eqref{eq:hoho}, we get:
\begin{align*}
    f(x_t)-f(x_{t+1}) 
        &\geq 
         \min \left\{ \ls\frac{f(x_{t+1}) - f(x^{\ast})}{R}\rs^2\left( \tfrac{1}{4\delta_t }\right), \ls \frac{f(x_{t+1}) - f(x^{\ast})}{R}\rs^\frac{3}{2}
        \left( \tfrac{1}{3M}\right)^\frac{1}{2}\right\}\\
        &\geq 
          \ls\frac{f(x_{t+1}) - f(x^{\ast})}{R}\rs^2\left( \tfrac{1}{4\delta_t }\right)\\
        &\geq 
          \ls\frac{f(x_{t+1}) - f(x^{\ast})}{R}\rs^2\left( \tfrac{1}{4\delta_T }\right)  .
 \end{align*}
So, for the last $T_2$ iterations, we have gradient method's rate.
\begin{gather*}
    f(x_t)-f(x_{t+1}) \geq \ls\frac{f(x_{t+1}) - f(x^{\ast})}{R}\rs^2\left( \tfrac{1}{4\delta_T }\right)\\
        \ls f(x_t) - f(x^{\ast}) \rs - \ls f(x_{t+1}) - f(x^{\ast}) \rs \geq \ls\frac{f(x_{t+1}) - f(x^{\ast})}{R}\rs^2\left( \tfrac{1}{4\delta_T }\right)\\
        \frac{f(x_t) - f(x^{\ast})}{4\delta_T R^2}  - \frac{f(x_{t+1}) - f(x^{\ast})}{4\delta_T R^2} \geq \ls\frac{f(x_{t+1}) - f(x^{\ast})}{4\delta_T R^2}\rs^2\\
        \xi_t - \xi_{t+1}\geq \xi_{t+1}^{1+1},
\end{gather*}
where $\xi_t = \frac{f(x_t) - f(x^{\ast})}{4\delta_T R^2}$.

By applying \eqref{eq:nesterov_sequence_trick} with $\al = 1$, we get 
    \begin{gather*}
    \xi_{T}\leq 2  \ls 1 +\frac{f(x_{T_1}) - f(x^{\ast})}{4\delta_{T} R^2} \rs \cdot \frac{1}{T_2}\\
    \frac{f(x_{T}) - f(x^{\ast})}{4\delta_{T} R^2} \leq  2  \ls 1 +\frac{f(x_{T_1}) - f(x^{\ast})}{4\delta_{T} R^2} \rs \cdot \frac{1}{T_2}\\
    f(x_{T}) - f(x^{\ast}) \leq    4\delta_{T} R^2 +\ls f(x_{T_1}) - f(x^{\ast})\rs  \cdot \frac{2}{T_2}
\end{gather*}
By \eqref{eq:hehe2} and \eqref{eq:hehe31}, one can show that $f(x_{T_1}) - f(x^{\ast}) \leq 4\delta_{T} R^2$, hence 
\begin{gather*}
    f(x_{T}) - f(x^{\ast}) \leq  \frac{16\delta_{T} R^2}{T_2}
\end{gather*}
or
\begin{equation}
    \label{eq:final_case2_cubic}
    T_2 = O\ls \frac{\delta_{T} R^2}{\varepsilon}\rs 
\end{equation}
Finally, by uniting \eqref{eq:final_case1_cubic} and \eqref{eq:final_case2_cubic}, we prove the Theorem.
\end{proof}

Let us discuss this result. To reach results from Corollary \ref{cor:controllable_inexactness}, we have to stay in Case 1 for all iterations. From the proof, one can see, that if $\delta_t=O\ls\frac{MR}{t}\rs$ then we stay in Case 1 for all iteration of the method. It means that first iterations could be done less precise than the last one.

\section{Adaptive Accelerated Inexact Cubic Newton for convex functions}

\begin{algorithm}
  \caption{Adaptive Accelerated Inexact Cubic Newton, Full version}\label{alg:inexact_acc_detailed}
  \begin{algorithmic}[1]
      \STATE \textbf{Input:} $y_0 = x_0$ is starting point; constants $M \geq 2L_2$; increase multiplier $\gamma_{inc}$; 
     starting inexactness $\delta_0  \geq 0$; non-negative non-decreasing sequences $\{\kappa_2^t\}_{t \geq 0}$, $\{\kappa_3^t\}_{t \geq 0}$, and
      \begin{equation}\label{eq:alphas}
          \alpha_t = \frac{3}{t + 3}, ~~~ A_t = \prod \limits_{j=1}^t(1 -\alpha_j), ~~~ A_0 = 1,
      \end{equation}
      \begin{equation}\label{eq:psi_0}
        \psi_{0}(x):=\sum \limits_{i = 2}^{3} \frac{\bk_i^{0}}{i}\|x - x_0\|^i.
    \end{equation}
    \WHILE{$t \geq 0$} 
        \STATE 
            \begin{equation}\label{eq:u_t}
                v_t = (1 - \alpha_t)x_t + \alpha_t y_t,
            \end{equation}
            
        \STATE $x_{t+1} = S_{M,\delta_t}(v_{t})$
                \WHILE{$\langle \nabla f(x_{t+1}), v_t - x_{t+1} \rangle 
        \leq 
        \min \left\{ \frac{\|\nabla f(x_{t+1})\|^2}{4\delta_t }, 
         \frac{\|\nabla f(x_{t+1})\|^\frac{3}{2}}{(3M)^\frac{1}{2}}\right\}$}
			     \STATE $\delta_t= \delta_t\gamma_{inc}$
			     \STATE $x_{t+1} = S_{M,\delta_t}(v_{t})$
                \ENDWHILE
        \STATE $\bk^{t+1}_2 = \frac{4\delta_t\al_t^2}{A_t}$
        \STATE Compute 
            \begin{equation}\label{eq:estimating_seq}
            \begin{aligned}
                y_{t+1}=\arg \min _{x \in \mathbb{R}^{n}}\left\{\psi_{t+1}(x):=\psi_{t}(x)+ \sum \limits_{i = 2}^{3} \frac{\bk^{t+1}_i - \bk^{t}_i}{i}\|x - x_0\|^i
                +\frac{\alpha_{t}}{A_{t}} 
                l(x,x_{t+1}) \right\}.
                \end{aligned}
            \end{equation}
        \IF{$\|y_{t+1} - y_{t}\| > \|y_{t+1} - x_0\|$ or $\bk^{t}_2 \geq \frac{2\delta_t\al_t^2}{A_t}$} 
            \STATE $\delta_{t-1}= \delta_t$
            \STATE Move back to line $8$ of previous step and recompute everything from $y_t$; $t=t-1$
        \ENDIF
        \STATE $\delta_{t+1} = \delta_t$
        \STATE $t=t+1$
    \ENDWHILE
  \end{algorithmic}
\end{algorithm}
Algorithm \ref{alg:inexact_acc_detailed} is detailed version of Algorithm \ref{alg:inexact_acc}.It is helping to better understand the algorithm and proof.
In full version, we add some additional safety conditions to check: lines $10-12$. Due to the complexity of proofs and theory, it is hard to prove that condition
\begin{equation}
    \label{eq:condition_acc}
    \|y_{t+1} - y_{t}\| \geq \|y_{t+1} - x_0\| \quad \textit{or} \quad  \bk^{t}_2 \geq \frac{2\delta_t\al_t^2}{A_t}
\end{equation}
holds for all $t$. We've add it as a check to make method computationally faster. It comes from adaptivity part and does not change the theoretical convergence rate. Moreover, in our experiments, the condition \eqref{eq:condition_acc} is always true. So, we believe that it is fair to add such check in the Algorithm and use it in the proof. 

Algorithm \ref{alg:inexact_acc_detailed} is an adaptive and inexact version of Nesterov's acceleration with the estimating sequence technique. By getting an upper and a lower bound for the estimating sequence $\psi_t(x)$, we prove the convergence theorem (Theorem \ref{thm:acc_convergence}). So, the full proof is organized as follows:
\begin{itemize}
    \item Lemma \ref{lem:upper_seq} provides an upper bound for the estimating sequence  $\psi_t(x)$;
    \item From Lemma \ref{lem:inexact_cubic_step_improvement}, we get the efficiency of Inexact Cubic Newton step $x_{t+1} = S_{M,\delta_t}(v_{t})$, and a guarantee of reaching adaptive condition. 
    \item Lemma \ref{lem:step} provides a lower bound on $\psi_t(x)$ based on results of technical Lemmas
    \ref{lem:dual}- \ref{lm:argmin}; 
    \item Everything is combined together in Theorem \ref{thm:acc_convergence} in order to prove convergence and obtain convergence rate.
\end{itemize}

Let us remind you the full formula for $\psi_t(x)$,
\begin{equation}
    \label{eq:psi_bar}
    \psi_0(x) = \sum \limits_{i=2}^{3} \frac{ \bk^{0}_i}{i}\|x - x_0\|^i; \quad \psi_t(x) =  \sum \limits_{i=2}^{3} \frac{ \bk^{t}_i}{i}\|x - x_0\|^i + \sum \limits_{j = 0}^{t - 1} \frac{\alpha_j}{A_j}l(x,x_{j+1}), \, \forall t\geq 1.
\end{equation}

The following Lemma shows that the sequence of functions $\psi_t(x)$ can be upper bounded by the properly regularized objective function.

\begin{lemma}\label{lem:upper_seq}
    For convex function $f(x)$ with solution solution $x^{\ast} \in \mathbb{R}^n$ and $\psi_t(x)$ from \eqref{eq:psi_bar}, we have
    \begin{equation}
        \label{eq:acc_upper_bound}
        \psi_t(x^{\ast}) \leq \frac{f(x^{\ast})}{A_{t - 1}}  + \frac{ \bk_2^{t}}{2}\|x^{\ast} - x_0\|^2 + \frac{\bk^{t}_{3}}{3}\|x^{\ast} - x_0\|^3.
        \end{equation}
\end{lemma}

\begin{proof}
For $t=0$, let us define $A_{-1}$ such that $\tfrac{1}{A_{-1}}=0$ then $\tfrac{f(x^{\ast})}{A_{-1}}=0$ and
   $$
        \psi_0(x^{\ast}) \leq \sum \limits_{i = 2}^{3} \frac{\bk^0_i}{i}\|x^{\ast} - x_0\|^i$$
    From \eqref{eq:psi_bar}, we have
    \begin{equation}
        \psi_t(x^{\ast}) 
        = \sum \limits_{i = 2}^{3} \frac{\bk^{t }_i}{i}\|x^{\ast} - x_0\|^i + \sum \limits_{j = 0}^{t - 1} \frac{\alpha_j}{A_j} l(x^{\ast},x_{j+1}).
        \label{eq:lem_upper_seq_pr2}
    \end{equation}
From \eqref{eq:alphas}, we have that, for all $j \geq 1$, $A_j = A_{j-1}(1-\alpha_j)$, which leads to $\frac{\alpha_j}{A_j}=\frac{1}{A_j}-\frac{1}{A_{j-1}}$. Hence, we have $\sum \limits_{j = 0}^{t - 1} \frac{\alpha_j}{A_j} = \frac{1}{A_{t-1}} - \frac{1}{A_{-1}}=\frac{1}{A_{t-1}}$ and, using the convexity of the objective $f$, we get
    \begin{gather}
        \sum \limits_{j = 0}^{t - 1} \frac{\alpha_j}{A_j} l(x^{\ast},x_{j+1}) \leq f(x^{\ast})\sum \limits_{j = 0}^{t - 1} \frac{\alpha_j}{A_j} = \frac{f(x^{\ast})}{A_{t-1}} . \label{eq:lem_upper_seq_pr3}
    \end{gather}
    Finally, combining all the inequalities from above, we obtain 
    \begin{gather}
        \psi_t(x^{\ast})  \stackrel{\eqref{eq:lem_upper_seq_pr2},\eqref{eq:lem_upper_seq_pr3}}{\leq}  \frac{f(x^{\ast})}{A_{t - 1}} + \frac{ \bk_2^{t}}{2} \|x^{\ast} - x_0\|^2   + \frac{\bk^{t}_{3}}{3}\|x^{\ast} - x_0\|^3.
    \end{gather}
\end{proof}

The next step is to provide a lower bound $\psi_t(x) \geq \psi_t^{\ast} := \min_x \psi_t(x) \geq \frac{f(x_t)}{A_{t-1}}$ for all $x$. The convergence rate of Algorithm \ref{alg:inexact_acc} will follow from this bound and Lemma \ref{lem:upper_seq}. The proof of the desired lower bound is quite technical and requires several auxiliary lemmas. After that we combine all the technical results together to obtain convergence rate in the proof of Theorem \ref{thm:acc_convergence}. We start the technical derivations with the following result.

\begin{lemma}\label{lm:argmin}
    Let $h(x)$ be a convex function, $x_0 \in \mathbb{R}^n$, $\theta_2, \theta_3 \geq 0$, and 
    $$\bar{x} = \arg \min \limits_{x \in \mathbb{R}^n} \{\bar{h}(x) = h(x) + \frac{\theta_2}{2} \|x -x_0\|^2 + \frac{\theta_3}{3} \|x -x_0\|^3\}.$$ Then, for all $x\in \mathbb{R}^n$,
    
    $$\bar{h}(x) \geq \bar{h}(\bar{x}) + 
    \frac{\theta_2}{2} \|x - \bar{x}\|^2 + \frac{\theta_3}{6} \|x - \bar{x}\|^3.$$
\end{lemma}
\begin{proof}
    By using Lemma 4 from \cite{nesterov2008accelerating}, one can see that, for all $x, y \in \mathbb{R}^n$ and for any $i \geq 2$, 
    \begin{equation*}
        d_i(x) - d_i(y) - \langle \nabla d_i(y), x - y \rangle \geq \left(\frac 12 \right)^{i-2}d_i(x - y),
    \end{equation*}
    where $d_i(x)= \tfrac{1}{i}\|x\|^{i}$.
    
    Using the convexity of $h$, we have
    \begin{gather*}
        \bar{h}(x) = h(x) + \sum \limits_{i = 2}^{3} \theta_i d_i(x - x_0) \geq h(\bar{x}) + \langle\nabla h(\bar{x}), x - \bar{x}\rangle + \sum \limits_{i = 2}^{3} \theta_i d_i(x - x_0) \\
        \geq h(\bar{x}) + \langle \nabla h(\bar x), x - \bar x\rangle + \sum \limits_{i = 2}^{3} \theta_i\left( d_i(\bar{x} - x_0) + \langle\nabla d_i(\bar{x} - x_0), x - \bar{x} \rangle + \left(\frac{1}{2}\right)^{i - 2}d_i(x - \bar x)\right)\\
        = \bar h(\bar x) + \langle \nabla \bar h (\bar x), x - \bar x\rangle + \sum \limits_{i = 2}^{3} \left(\frac{1}{2}\right)^{i - 2}\theta_i d_i(x - \bar x) \geq  \bar h(\bar x) +  \sum \limits_{i = 2}^{3} \left(\frac{1}{2}\right)^{i - 2}\theta_i d_i(x - \bar x),
    \end{gather*}
    where the last inequality holds by optimality condition since $\bar h(x)$ is convex.
\end{proof}

We will also use the next technical lemma \cite{nesterov2008accelerating,ghadimi2017second} on Fenchel conjugate for the $p$-th power of the norm.
\begin{lemma}\label{lem:dual}
    Let $g(z)=\frac{\theta}{p}\|z\|^{p}$ for $p \geq 2$ and $g^{*}$ be its conjugate function i.e., $g^{*}(v)=\sup _{z}\{\langle v, z\rangle-$ $g(z)\} .$ Then, we have
$$
g^{*}(v)=\frac{p-1}{p}\left(\frac{\|v\|^{p}}{\theta}\right)^{\frac{1}{p-1}}
$$
Moreover, for any $v, z \in \mathbb{R}^{n}$, we have $g(z)+g^{*}(v)-\langle z, v\rangle \geq 0 .$
\end{lemma}

Finally, the last step is the next Lemma which prove that $\frac{f(x_t)}{A_{t-1}} \leq  \min \limits_x \psi_t(x) = \psi_t^{\ast}$.

\begin{lemma}\label{lem:step}
    Let $\{x_t, y_t\}_{t \geq 1}$ be generated by Algorithm \ref{alg:inexact_acc}. 
    Then
    \begin{equation}\label{eq:lemma_ass}
    \psi_t^{\ast} := \min \limits_x \psi_t(x)   \geq \frac{f(x_t)}{A_{t-1}}.
    \end{equation}
\end{lemma}

\begin{proof}
    We prove Lemma by induction. Let us start with $t=0$, we define $A_{-1}$ such that $\tfrac{1}{A_{-1}}=0$. Then $\tfrac{f(x_0)}{A_{-1}}=0$ and $\psi_0^{\ast}=0$, hence, $\psi_0^{\ast}\geq \tfrac{f(x_0)}{A_{-1}}$. Let us assume that $\frac{f(x_t)}{A_{t-1}} \leq \psi^{\ast}_t$
    and show that  $\frac{f(x_{t+1})}{A_{t}}\leq \psi^{\ast}_{t+1}$.
    By definition,
    \begin{gather*}
        \psi_t(x) =  \sum \limits_{i=2}^{3} \frac{ \bk^{t}_i}{i}\|x - x_0\|^i + \sum \limits_{j = 0}^{t - 1} \frac{\alpha_j}{A_j}l(x,x_{j+1}).
    \end{gather*}
    Next, we apply Lemma \ref{lm:argmin} with the following choice of parameters: $h(x) =  \sum \limits_{j = 0}^{t - 1}\frac{\alpha_j}{A_j}l(x,x_{j+1})$, $\theta_i = \bk^{t}_i$ for $i = 2, 3$. \\
    By \eqref{eq:estimating_seq}, $y_t = \argmin \limits_{x\in \mathbb{R}^n} \bar{h}(x)$, and we have
    \begin{gather*}
        \psi_t(x) \geq \psi_t^{\ast} + \frac{ \bk^{t}_2}{2}\|x - y_t\|^2 + \frac{ \bk^{t}_3}{6}\|x - y_t\|^3\\
        \stackrel{\eqref{eq:lemma_ass}}{\geq} \frac{f(x_t)}{A_{t - 1}} + \frac{ \bk^{t}_2}{2}\|x - y_t\|^2 + \frac{ \bk^{t}_3}{6}\|x - y_t\|^3,
    \end{gather*}
    where the last inequality follows from the assumption of the lemma.
    
    By the definition of $\psi_{t+1}(x)$, the above inequality, and convexity of $f$,we obtain
    \begin{gather*}
        \psi_{t+1}(y_{t+1}) =  \psi_t(y_{t+1}) + \tfrac{ \bk^{t+1}_2-\bk_2^{t}}{2}\|y_{t+1} - x_0\|^2 
            +  \tfrac{ \bk^{t+1}_3-\bk_3^{t}}{3}\|y_{t+1} - x_0\|^3
         + \frac{\alpha_t}{A_t}l(y_{t+1},x_{t+1})\\
        \overset{\eqref{eq:condition_acc}}{\geq} \frac{f(x_t)}{A_{t - 1}} + \frac{ \bk^{t+1}_2}{2}\|y_{t+1} - y_t\|^2 
            +  \frac{ \bk^{t}_3}{6}\|y_{t+1} - y_t\|^3  + \frac{\alpha_t}{A_t}l(y_{t+1},x_{t+1})\\
        \geq  \frac{1}{A_{t-1}}(f(x_{t+1}) + \langle \nabla f(x_{t+1}), x_t - x_{t+1}\rangle)\\
        + \frac{ \bk^{t+1}_2}{2}\|y_{t+1} - y_t\|^2 
            +  \frac{ \bk^{t+1}_3}{6}\|y_{t+1} - y_t\|^3
         + \frac{\alpha_t}{A_t}l(y_{t+1},x_{t+1}). 
    \end{gather*}
    \begin{remark}
    We highlight that in the first inequality we use the condition \eqref{eq:condition_acc} that is checked in the method. Because of induction part and fixed $\delta$ in previous proofs, the current step is regularized by previous $\kappa_t$ and $\delta_{t-1}$. From our perspective, it is strange that we are losing so much information by removing $\frac{ \bk^{t+1}_2-\bk_2^{t}}{2}\|y_{t+1} - x_0\|^2$. We look forward for a more clear proof without need of such condition in the method and smarter adaptation rules.
    \end{remark}
    
    Next, we consider the sum of two linear models from the last inequality:
    \begin{gather*}
    \frac{1}{A_{t-1}}(f(x_{t+1}) + \langle \nabla f(x_{t+1}), x_t - x_{t+1}\rangle) + \frac{\alpha_t}{A_t}l(y_{t+1},x_{t+1}) \\
    =\frac{1}{A_{t-1}}(f(x_{t+1}) + \langle \nabla f(x_{t+1}), x_t - x_{t+1}\rangle)
    + \frac{\alpha_t}{A_t} (f(x_{t+1}) 
    + \langle \nabla f(x_{t+1}), y_{t+1} - x_{t+1}\rangle)\\
        \stackrel{\eqref{eq:alphas}}{=} \frac{1 - \alpha_t}{A_t}f(x_{t+1}) + \frac{1 - \alpha_t}{A_t}\langle \nabla f(x_{t+1}), x_t - x_{t+1}\rangle + \frac{\alpha_t}{A_t}f(x_{t+1}) +\frac{\alpha_t}{A_t}\langle \nabla f(x_{t+1}), y_{t+1}- x_{t+1}\rangle\\
        \\
        \stackrel{\eqref{eq:u_t}}{=} \frac{f(x_{t+1})}{A_{t}} + \frac{1 - \alpha_t}{A_t}\langle \nabla f(x_{t+1}), \frac{v_t - \alpha_t y_t}{1 - \alpha_t}- x_{t+1}\rangle + \frac{\alpha_t}{A_t}\langle \nabla f(x_{t+1}), y_{t+1}- x_{t+1} \rangle \\
        = \frac{f(x_{t+1})}{A_{t}} + \frac{1}{A_t}\langle \nabla f(x_{t+1}), v_t- x_{t+1}\rangle + \frac{\alpha_t}{A_t}\langle \nabla f(x_{t+1}), y_{t+1}- y_t \rangle.
    \end{gather*}
    As a result, by \eqref{eq:estimating_seq}, we get
    \begin{equation}
    \label{eq:x_bar_def}
        \begin{gathered}
            \psi^{\ast}_{t+1} = \psi_{t+1}(y_{t+1})  \geq \frac{f(x_{t+1})}{A_{t}} + \frac{1}{A_t}\langle \nabla f(x_{t+1}), v_t- x_{t+1}\rangle 
            + \frac{ \bk^{t+1}_2}{2}\|y_{t+1} - y_t\|^2 \\
            +  \frac{ \bk^{t}_3}{6}\|y_{t+1} - y_t\|^3 +\frac{\alpha_t}{A_t}\langle \nabla f(x_{t+1}), y_{t+1}- y_t  \rangle.
        \end{gathered}
    \end{equation}
 
    To complete the induction step, we show, that the sum of all terms in the RHS except $\frac{f(x_{t+1})}{A_t}$  is non-negative. 
    
    Lemma \ref{lem:inexact_cubic_step_improvement} provides the lower bound for $\langle \nabla f(x_{t+1}), v_t - x_{t+1}\rangle$. Let us consider the case when the minimum in the RHS of \eqref{eq:cubic_step_improvement} is attained at the first term. 
    By Lemma \ref{lem:dual} with the following choice of the parameters
    $$z = y_t - y_{t+1}, ~~ v = \frac{\alpha_t}{A_t}\nabla f(x_{t+1}), ~~ \theta = \bk_i^{t},$$
    we have 
    \begin{equation}\label{eq:1case}
        \frac{\bk_2^{t+1}}{2}\|y_t -y_{t+1}\|^2 + \frac{\alpha_t}{A_t}\langle \nabla f(x_{t+1}), y_{t+1} - y_t \rangle \geq - \frac{1}{2}\left( \frac{\|\frac{\alpha_t}{A_t}\nabla f(x_{t+1})\|^2}{\bk_2^{t+1}} \right).
    \end{equation}
    Hence,
    \begin{gather*}
        \frac{f(x_{t+1})}{A_t} + \frac{1}{A_t}\langle \nabla f(x_{t+1}), v_t - x_{t+1}\rangle + \frac{ \bk_2^{t+1}}{2}\|y_{t+1} - y_t\|^2 + \frac{\alpha_t}{A_t}\langle \nabla f(x_{t+1}), y_{t+1} - {y_t}\rangle \\
        \stackrel{\eqref{eq:1case}}{\geq} \frac{f(x_{t+1})}{A_t} + \frac{1}{A_t}\langle \nabla f(x_{t+1}), v_t - x_{t+1}\rangle - \frac{\|\frac{\alpha_t}{A_t}\nabla f(x_{t+1})\|^i}{2\bk_2^{t+1}}  \\
        \stackrel{\eqref{eq:cubic_step_improvement}}{\geq}
        \frac{f(x_{t+1})}{A_t} + \frac{1}{A_t}\|\nabla f(x_{t+1})\|^2 \left( \frac{1}{4\delta_t}\right) - \frac{\|\frac{\alpha_t}{A_t}\nabla f(x_{t+1})\|^2}{2\bk_i^{t+1}} \\
         \geq \frac{f(x_{t+1})}{A_t}, 
    \end{gather*}
    where the last inequality holds by our choice of the parameters
    \begin{equation}\label{eq:bar_kappa}
        \bk_2^{{t+1}} \geq  \frac{2\delta_t\al_t^2}{A_t} .
    \end{equation}

    Next, we consider the case when the minimum in the RHS of \eqref{eq:cubic_step_improvement} is achieved on the second term. Again, by Lemma \ref{lem:dual} with the same choice of $z, v$ and with $\theta = \frac{\bk_{3}^{t}}{2}$, we have
    \begin{equation}
        \label{eq:2case}
        \frac{\bk_{3}^{t}}{6}\|y_t -y_{t+1}\|^3 + \frac{\alpha_t}{A_t}\langle \nabla f(x_{t+1}), y_{t+1} - {y_t} \rangle \geq - \frac{2}{3}\left( \frac{2\|\frac{\alpha_t}{A_t}\nabla f(x_{t+1})\|^{3}}{\bk_{3}^{t}} \right)^\frac{1}{2}.
    \end{equation}

   Hence, we get
    \begin{gather*}
        \frac{f(x_{t+1})}{A_t} + \frac{1}{A_t}\langle \nabla f(x_{t+1}), v_t - x_{t+1}\rangle + \frac{\bk_{3}^{t+1}}{6}\|y_{t+1} - y_t\|^3 + \frac{\alpha_t}{A_t}\langle \nabla f(x_{t+1}), y_{t+1} - {y_t}\rangle\\
        \stackrel{\eqref{eq:2case}}{\geq} \frac{f(x_{t+1})}{A_t} + \frac{1}{A_t}\langle \nabla f(x_{t+1}), v_t - x_{t+1}\rangle  - \frac{2}{3}\left( \frac{2\|\frac{\alpha_t}{A_t}\nabla f(x_{t+1})\|^{3}}{\bk_{3}^{t}} \right)^\frac{1}{2}\\
        \stackrel{\eqref{eq:cubic_step_improvement}}{\geq}
        \frac{f(x_{t+1})}{A_t} + \frac{1}{A_t}\|\nabla f(x_{t+1})\|^\frac{3}{2}
        \left( \frac{1}{3M}\right)^\frac{1}{2} - \frac{2}{3}\left( \frac{2\|\frac{\alpha_t}{A_t}\nabla f(x_{t+1})\|^{3}}{\bk_{3}^{t}} \right)^\frac{1}{2}\\
         \geq \frac{f(x_{t+1})}{A_t},
    \end{gather*}
    where the last inequality holds by our choice of $\bk_{3}^{t}$:
    \begin{equation}\label{eq:bar_kappa_p+1}
        \bk_{3}^t \geq \frac{8M}{3} \frac{\alpha_t^{3}}{A_t}.
    \end{equation}

    To sum up, by our choice of the parameters $\bk_{i}^{t}$, $i=2,3$, we prove the induction step.
    
\end{proof}

Finally, we are in a position to prove the convergence rate. For the proof, we denote $R$ as
\begin{equation}
\label{R_teorem3}
\|x_0 - x^{\ast}\| \leq R.
\end{equation}

\begin{theorem}\label{thm:acc_convergence_app}
Let $f(x)$ be a convex function, $f(x)$ has $L_2$-Lipschitz-continuous Hessian, $B_{t}$ is a $\delta_{v_{t}}^{x_{t+1}}$-inexact Hessian, there exists $\delta_{max}$ such that $\delta_{v_t}^{x_{t+1}}\leq \delta_{max}$ for all $t$, and $M\geq 2L_2$. 
    Algorithm  \ref{alg:inexact_acc} makes $T\geq 1$ iterations with parameters
    \begin{gather*}
        \bk_2^{{t+1}} =  \frac{4\delta_t\al_t^2}{A_t} \quad  \bk_{3}^{t+1} = \frac{8M}{3} \frac{\alpha_{t+1}^{3}}{A_{t+1}}.
    \end{gather*}
    to find $\varepsilon$-solution $x_T$ such that $f(x_{T}) - f(x^{\ast})\leq \varepsilon$. The value of $T$ is bounded by  
    \begin{equation*}
    \label{eq:acc_convergence}
     T = O(1) \max \lb \ls \tfrac{\delta_{T} R^2}{\varepsilon}\rs^{1/2};  \ls\tfrac{MR^3}{\varepsilon}\rs^{1/3}\rb, \text{where }R= \|x_0-x^{\ast}\|.
    \end{equation*}

\end{theorem}
\begin{proof}  
        From Lemmas \ref{lem:step} and \ref{lem:upper_seq}, we obtain that, for all $t \geq 1$,
    \begin{gather*}
       \frac{f(x_{t+1})}{A_t} \stackrel{\eqref{eq:lemma_ass}}{\leq} \psi_{t+1}^{\ast} \leq \psi_{t+1}(x^{\ast}) \stackrel{\eqref{eq:acc_upper_bound}}{\leq}
       \frac{f(x^{\ast})}{A_{t}} + \frac{ \bk^{t}_2}{2}\|x^{\ast} - x_0\|^2 + \frac{\bk^{t}_{3}}{3}\|x^{\ast} - x_0\|^3.
    \end{gather*}
    Now, we fix $\alpha_t$ to get upperbounds for $\bk^{t}_{3}$ and $\bk^{t+1}_{2}$. Let us take
        \begin{equation}
        \label{alpha_t}
            \alpha_t = \frac{3}{t+3}, ~ t \geq 1.
        \end{equation}
        Then, we have
        \begin{equation}\label{eq:A_t_bound}
            A_{T} = 
            \prod_{t=1}^{T}\left(1-\alpha_{t}\right)=\prod_{t=1}^{T} \frac{t}{t+3}=\frac{T !3 !}{(T+3) !}=\frac{6}{(T+1)(T+2)(T+3)}.
         \end{equation}
        So, we get the next upperbounds 
        \begin{equation}
        \label{eq:A_t_bound_acc}
                \frac{\alpha_{t}^{3}}{A_{t}} =\tfrac{9}{2} \frac{(t+1)(t+2)}{(t+3)^2} \leq 4.5,
        \end{equation}
        
        \begin{equation}
        \label{eq:delta_t_bound_acc}
                \frac{\delta_{t}\alpha_{t}^{2}}{A_{t}} =\tfrac{3}{2} \frac{\delta_t(t+1)(t+2)}{(t+3)} \leq \tfrac{3}{2}  \delta_t(t+1).
        \end{equation}
    Finally, we obtain the following convergence rate bound
    \begin{gather*}
       f(x_{T}) - f(x^{\ast})  \stackrel{\eqref{R_teorem3}}{\leq} A_T\left( \frac{ \bk^{T}_2}{2}R^{2} + \frac{\bk^{t}_{3}}{3}  R^{3}
       \right) \\
      =  A_T\left(
      \frac{4\delta_{T-1}\al_{T-1}^2}{A_{T-1}}R^{2}+ 12M R^{3}\right) 
    \stackrel{\eqref{eq:A_t_bound_acc}, \eqref{eq:delta_t_bound_acc}}{\leq}
      \frac{36\delta_T R^{2}}{(T+3)^{2}} 
    + \frac{72MR^{3}}{(T+1)^{3}}
      .
    \end{gather*}
    
    Thus, we obtain the statement of the theorem.
\end{proof}

\section{Quasi-Newton approximation}
In this section, we present the inexactness guaranties for different types of \QN approximations.
\subsection{\QN with history}
For this subsection, we assume that we have \QN approximation with history:
 \begin{equation*}
     s_i = z_{i+1}-z_{i},\qquad
     y_i = \nabla f(z_{i+1}) - \nabla f(z_{i}).
 \end{equation*}
 
Let us start with one of the most popular \QN method: L-BFGS \eqref{eq:l-bfgs}.
\begin{equation}
\label{eq:bfgs-ap}
B^{m+1}_{t} = B^{m}_{t} +  \tfrac{y_m y_m^{\top}}{y_m^{\top} s_m} -  \tfrac{B_t^m s_m (B^m_t s_m)^{\top}}{s_m^{\top} B_t^m s_m}.
\end{equation}
We want to get some results with respect to Assumption \ref{ass:direct_inexactness}.
Note, that it is enough to show that $\|B_x-\nabla^2 f(x)\|\leq \delta_{max}$ then automatically all directional inexactness are smaller than $\delta_{max}$.
\begin{lemma}
    \label{lm:bfgs}
Let $f(x)$ be a convex function with respect to the global minimizer $x^{\ast}$, and let $f(x)$ have an $L_1$-Lipschitz-continuous gradient and an $L_2$-Lipschitz-continuous Hessian. Suppose $B_{t}$ is a $m$-memory L-BFGS approximation from \eqref{eq:bfgs-ap}, then 
\begin{equation}
    \delta_{max}\leq m L_1,
\end{equation}
where 
\begin{equation}
    \label{eq:delta_max_QN}
    \|B_t-\nabla^2 f(x_t)\|\leq \delta_{max},\qquad \forall x \in \R^d.
\end{equation}
\end{lemma}
\begin{proof}
Let us start from the fact that $B_t \succeq 0$ if $B_0 \succeq 0$. From \eqref{eq:bfgs-ap}, we show by induction that if $B_t\succeq 0$ then $B_{t+1}\succeq 0$. 
By using Lemma $6.1$ from \cite{rodomanov2021rates}, we get that for all $h \in \R^d$, we have
\begin{align*}
    \la B_t h, h\ra - \frac{\la B_t u, h \ra^2}{\la B_t u, u\ra} 
    &= \min\limits_{\alpha \in R} \lb \la B_t h, h\ra - 2\al \la B_t u, h \ra + \al^2\la B_t u, u\ra \rb\\
    &= \min\limits_{\alpha \in R} \lb \la B_t (h-\al u), h  - \al  u \ra \rb\\
    \geq 0.
\end{align*}
Hence, $B_{t} -  \tfrac{B_t u d(B_t u)^{\top}}{u^{\top} B_t u} \succeq 0$. This proves that L-BFGS approximation \eqref{eq:bfgs-ap} is positive semidefinite and $B_t \succeq 0$. And we can prove the first part of the bound:
\begin{gather*}
\begin{split}
L_1 I &\succeq \nabla ^2 f(x_t)\\
B_t &\succeq 0
\end{split}
  \quad \Longrightarrow \quad B_t - \nabla ^2 f(x_t) \succeq - L_1 I
\end{gather*}
    The second part of the bound is based on Lemma 5.1 Option 1 from \cite{berahas2021quasi}.
    \begin{gather*}
        B^{m+1}_{t} = B^{m}_{t} +  \tfrac{y_m y_m^{\top}}{y_m^{\top} s_m} -  \tfrac{B_t^m s_m (B^m_t s_m)^{\top}}{s_m^{\top} B_t^m s_m},
    \end{gather*}
    where $s_m = z_{m+1}-z_{m},\qquad
     y_m = \nabla f(z_{m+1}) - \nabla f(z_{m})$. Hence, we get 
     \begin{align*}
        \lambda_{max}\ls B_t^{m+1}\rs &= \lambda_{max} \ls B^{m}_{t} +  \tfrac{y_m y_m^{\top}}{y_m^{\top} s_m} -  \tfrac{B_t^m s_m (B^m_t s_m)^{\top}}{s_m^{\top} B_t^m s_m}\rs\\
        &\leq  \lambda_{max} \ls B^{m}_{t} +  \tfrac{y_m y_m^{\top}}{y_m^{\top} s_m} \rs\\
        &\leq  \lambda_{max} \ls B^{m}_{t}\rs +  \lambda_{max} \ls \tfrac{y_m y_m^{\top}}{y_m^{\top} s_m} \rs\\
        &\leq  \lambda_{max} \ls B^{m}_{t}\rs +  \ls\tfrac{\|y_m\|^2}{y_m^{\top} s_m} \rs\\
        &\leq  \lambda_{max} \ls B^{m}_{t}\rs + L_1,
     \end{align*}
     where the last inequality holds because $f(x)$ is $L_1$-Lipschitz continuous and convex function.
     So, by summation, it means that $B^{m+1}_{t} \preceq \ls m L_1 +\lambda_{max}(\B_0)\rs I$. If we take $L_1 I\succeq B_0 \succeq 0$, we get that $B_t =B^{m}_{t}\preceq m L_1$, hence
     \begin{gather*}
\begin{split}
0 &\preceq \nabla ^2 f(x_t)\\
B_t &\preceq m L_1 I
\end{split}
  \quad \Longrightarrow \quad B_t - \nabla ^2 f(x_t) \preceq  m L_1 I.
\end{gather*}
Finally, $ \|B_t-\nabla^2 f(x_t)\|\leq m L_1$.
\end{proof}

If we a little bit change and damp the update of L-BFGS, like this
\begin{equation}
\label{eq:bfgs-damped}
B^{m+1}_{t} = B^{m}_{t} +  \myred{\tfrac{1}{m}}\tfrac{y_m y_m^{\top}}{y_m^{\top} s_m} -  \tfrac{B_t^m s_m (B^m_t s_m)^{\top}}{s_m^{\top} B_t^m s_m},
\end{equation}
hence we get \textit{damped L-BFGS} approximation, where the biggest inexactness of Classical L-BFGS is reduced. So, similarly to previous Lemma, we get the next result
\begin{lemma}
    \label{lm:bfgs_damped}
Let $f(x)$ be a convex function with respect to the global minimizer $x^{\ast}$, and let $f(x)$ have an $L_1$-Lipschitz-continuous gradient and an $L_2$-Lipschitz-continuous Hessian. Suppose $B_{t}$ is a $m$-memory \textit{damped L-BFGS} approximation from \eqref{eq:bfgs-damped}, then 
\begin{equation}
    \delta_{max}\leq L_1,
\end{equation}
where $\delta_{max}$ is from \eqref{eq:delta_max_QN}.
\end{lemma}
The proof is exactly the same as the proof of Lemma \ref{lm:bfgs}
By uniting these Lemmas with Theorems \ref{thm:cubic_convergence} and \ref{thm:acc_convergence}, we get the proof of Theorems \ref{thm:cubic_bfgs} and \ref{thm:accelerated_cubic_bfgs}.

\subsection{\QN with sampling}
For this subsection, we assume that we have \QN approximation with sampled HVP
 \begin{equation*}
     y_i = \nabla^2 f(x_{t})s_i.
 \end{equation*}
We extend the class of approximation to \textit{convex Broyden class}.
For symmetric positive semidefinite matrices $A$ and $G$, the DFP update is defined as
\begin{equation}
\label{eq:DFP_mem}
    \DFP(A,G, s) = \ls I - \frac{s s^{\top}A}{\la A s, s \ra } \rs^{\top} G 
    \ls I - \frac{s s^{\top}A}{\la A s, s \ra } \rs + \frac{A s (A s)^{\top} }{\la As, s\ra}.
\end{equation}
The BFGS update is defined as
\begin{equation}
\label{eq:BFGS_mem}
    \BFGS(A,G, s) = G -  \frac{G s (G s)^{\top} }{\la Gs, s\ra} + \frac{A s (A s)^{\top} }{\la As, s\ra}.
\end{equation}
We can form the convex Broyden class as a convex combination of BFGS and DFP updates.
\begin{equation}
\label{eq:Broyd_mem}
    \Broyd_{\upsilon}(A,G, s) = \upsilon \DFP(A,G,s) + (1-\upsilon) \BFGS(A,G,s),
\end{equation}
where $\upsilon \in [0,1].$
To get the $m$-memory approximation, we make $m$ next steps :
$$B_t^{i+1} = \Broyd_{\upsilon}(\nabla^2 f(x_t),B_t^i, s_i)$$
\begin{lemma}
    \label{lm:broyd_sampled}
Let $f(x)$ be a convex function with respect to the global minimizer $x^{\ast}$, and let $f(x)$ have an $L_1$-Lipschitz-continuous gradient and an $L_2$-Lipschitz-continuous Hessian. We fix $B_0 = 0$. Suppose $B_{t}$ is a $m$-memory approximation from convex Broyden class (DFP + BFGS) \eqref{eq:DFP_mem}, then 
\begin{equation}
    \delta_{max}\leq L_1,
\end{equation}
where $\delta_{max}$ is from \eqref{eq:delta_max_QN}.
\end{lemma}
\begin{proof}
Similarly to Lemma \ref{lm:bfgs}, $\BFGS(A,G, s) \succeq 0$ for BFGS update.
For DFP update, for any $h \in \R^d$, we have
\begin{align*}
   h^{\top}\DFP(A,G, s)h = \ls h - \frac{s s^{\top}Ah}{\la A s, s \ra } \rs^{\top} G 
    \ls h - \frac{s s^{\top}Ah}{\la A s, s \ra } \rs + \frac{h^{\top}A s (A s)^{\top}h }{\la As, s\ra}
    \geq 0,
\end{align*}
as $A \succeq 0 $ and $G \succeq 0$. Hence, $\DFP(A,G, s) \succeq 0$ and $\Broyd_{\upsilon}(A,G, s)\succeq 0$ as the convex combination. 

As a result, for the first bound, we get  
\begin{gather*}
\begin{split}
L_1 I &\succeq \nabla ^2 f(x_t)\\
B_t &\succeq 0
\end{split}
  \quad \Longrightarrow \quad B_t - \nabla ^2 f(x_t) \succeq - L_1 I
\end{gather*}
    The second part of the bound is based on Lemma 2.1 from \cite{rodomanov2021rates}.
    We take $A = \nabla^2 f(x)$, then $B_0= 0 \leq A$ hence all $B_t \leq A$ by Lemma 2.1. and we get the next bound
     \begin{gather*}
\begin{split}
0 &\preceq \nabla ^2 f(x_t)\\
B_t &\preceq \nabla ^2 f(x_t)
\end{split}
  \quad \Longrightarrow \quad B_t - \nabla ^2 f(x_t) \preceq   0 .
\end{gather*}
Finally, $ \|B_t-\nabla^2 f(x_t)\|\leq L_1$.
\end{proof}

So, we get a much bigger class of updates, the convex Broyden family, that can be applied as \QN with sampling in Theorems \ref{thm:cubic_bfgs} and \ref{thm:accelerated_cubic_bfgs}.

\section{Extra Experiments} \label{sec:extra_exp}
\add{We solve the following empirical risk minimization problem:
\begin{equation*}
    f(x)= \frac{1}{n}\sum_{i=1}^n \log(1+ \exp(-a_ib_i^Tx)) + \frac{\mu}{2}\|x\|^2.
\end{equation*}
For all the numerical experiments, we normalize each data point and get $\|a_i\|_2 = 1$ for all $i \in [1,...,n]$ to balance the problem and the data. In Figure \ref{fig1}, we consider the classification task on \texttt{real-sim} dataset for scaled hyperparameters on the convex case. Due to the complexity of Cubic Regularized Newton, it is out of our scope to show numerical experiments on \texttt{real-sim} dataset.
In Figures \ref{fig2} and \ref{fig3} we provide additional experiments using theoretical parameters for convex and strongly convex cases on datasets \texttt{a9a}, \texttt{gisette}, \texttt{MNIST}, \texttt{CIFAR-10}. For all the numerical experiments on  \texttt{MNIST} and \texttt{CIFAR-10} datasets, we consider the binary classification task where one class contains all the data points with labels less than 5, and the other class contain all the other data points. In Figure \ref{fig4} we provide additional experiments using tuned hyperparameters for all the Quasi-Newton and Cubic Regularized (Quasi-) Newton methods in strongly convex case. In Figure \ref{fig5} we provide additional experiments using tuned hyperparameters for all the Quasi-Newton and Cubic Regularized (Quasi-) Newton methods in the convex case. }

\begin{figure}[ht]
 
  \centering
  \includegraphics[width=0.45\linewidth]{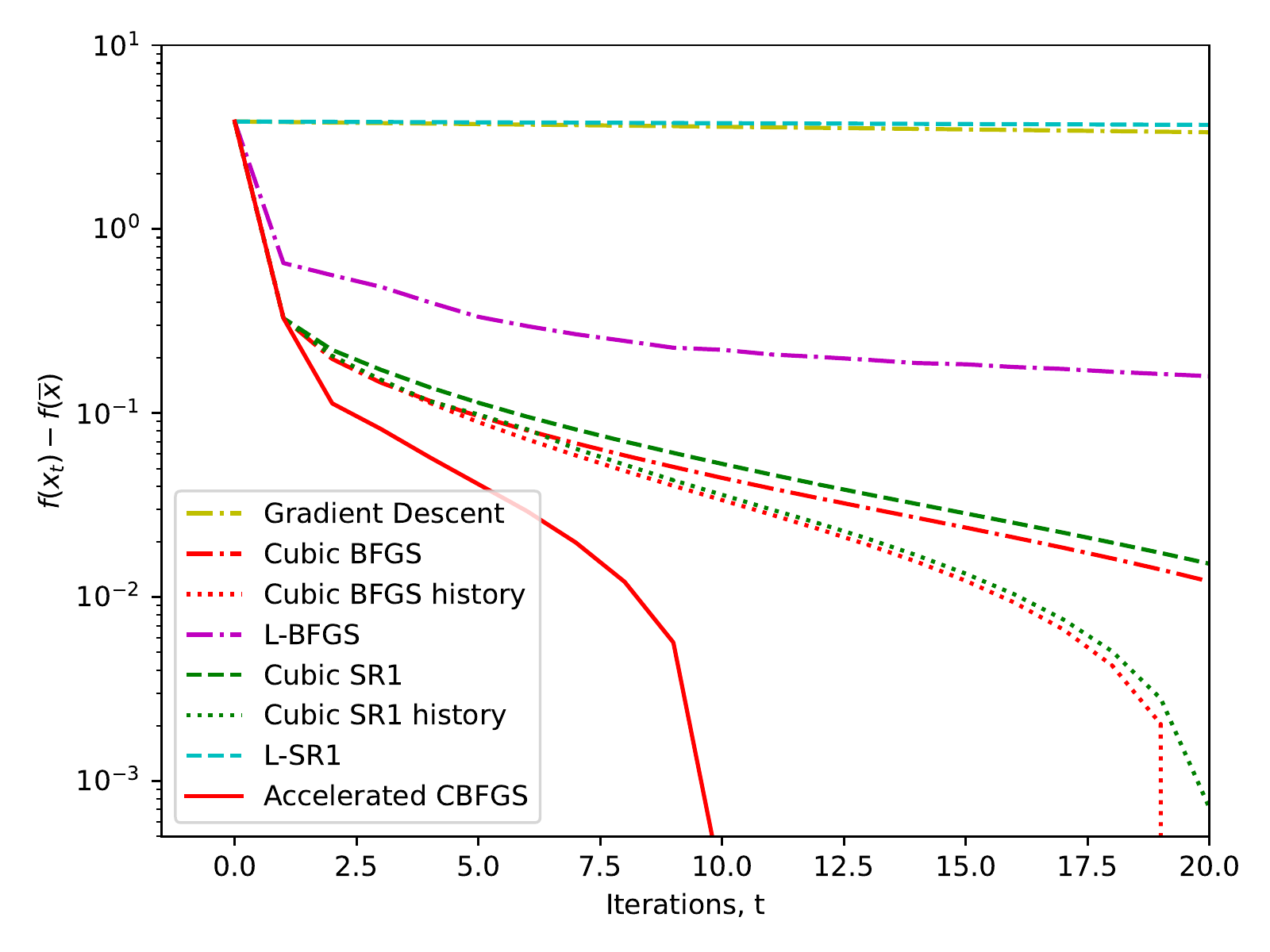}
  \includegraphics[width=0.45\linewidth]{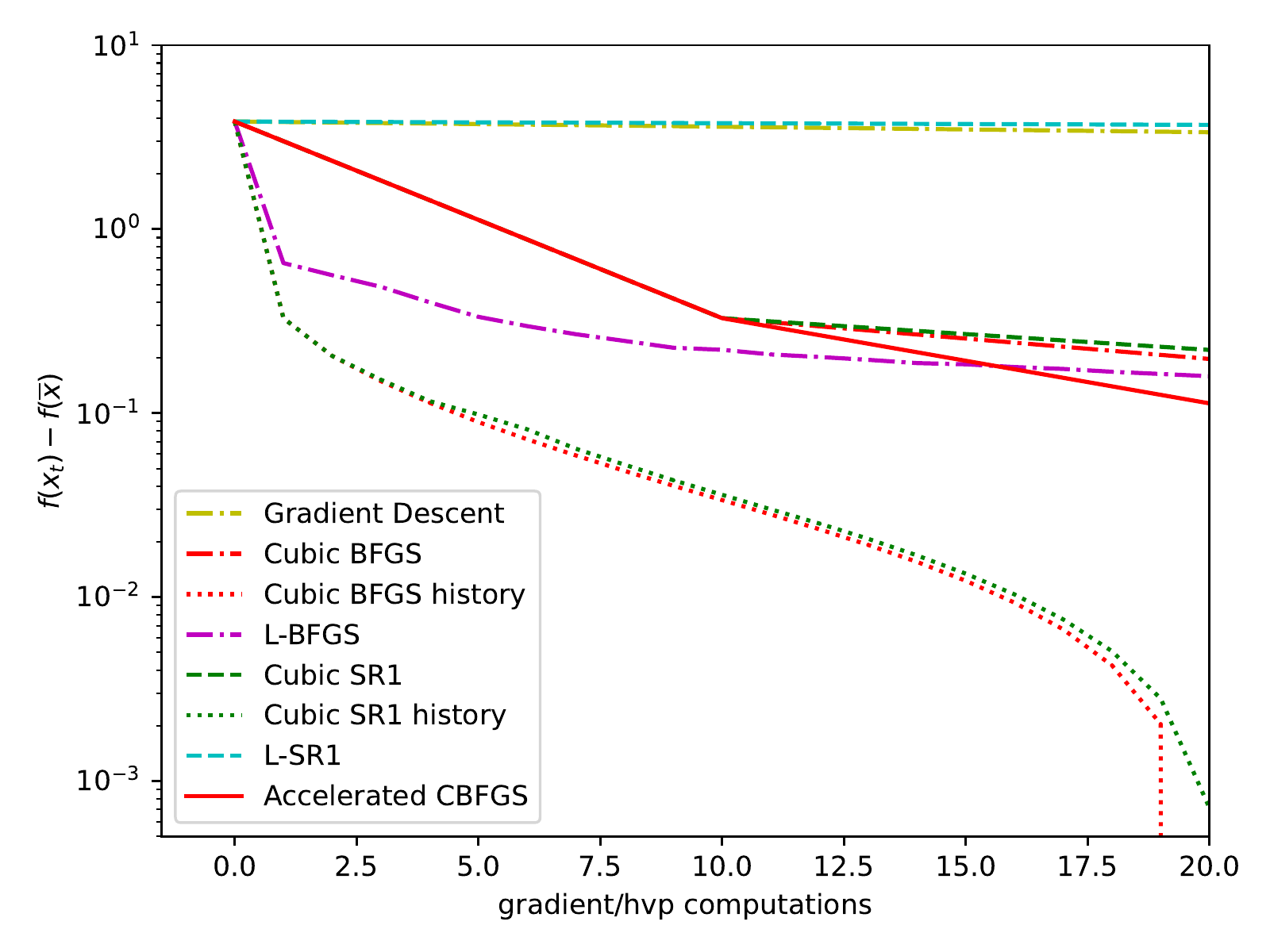}
  
  \caption{\add{Comparison of Quasi-Newton methods and Cubic Regularized (Quasi-) Newton methods for real-sim dataset on the convex case. Hyperparameters: $L_2 = 3.05 \cdot 10^{-5}$ for all the methods, $lr = 0.00013$ for L-SR1, $lr = 0.0116$ for L-BFGS.}}
\label{fig1}

\end{figure}
\subsection{Tuned Parameters.}
For all the methods we used eleven numbers spaced evenly on a log scale between $2^0$ and $2^{-10}$. After training all the methods with $L_2$ and $lr$ respectively in the range given above we selected the one that had the best performance to show the results for tuned parameters in all datasets. 
\add{Tuned hyperparameters used for \texttt{a9a} dataset in strongly-convex case shown in Figure \ref{a9a4}   are : $L_2 = 0.343116$ for Accelrated CBFGS, $lr = 0.125$ for L-BFGS, $lr = 0.00390625$ for L-SR1, $lr = 30$ for Gradient Descent and $L_2 = 0.0012708$ for the rest of the methods. 

Tuned hyperparameters used for  \texttt{gisette} dataset in strongly-convex case shown in Figure \ref{gisette4}   are $L_2 = 0.0078125$ for Cubic Regularized Newton, $L_2 = 0.0168$ for Accelerated CBFGS, $lr = 0.0078125$ for L-BFGS, $lr = 0.001953125$ for L-SR1, $lr = 10$ for Gradient Descent and $L_2 = 0.00012$ for the rest of the methods. From numerical experiments, we observed that for $L_2 < 0.0078125$ Cubic Regularized Newton method diverges and is not able to find the optimal solution.

Tuned hyperparameters used for  \texttt{MNIST} dataset in strongly-convex case shown in Figure \ref{mnist4} are : $L_2 = 0.001953$ for Cubic Regularized Newton, $L_2 = 0.17577$ for Accelrated CBFGS, $lr = 0.0078125$ for L-BFGS, $lr = 0.001953125$ for L-SR1, $lr = 30$ for Gradient Descent and $L_2 = 0.000016$ for the rest of the methods.}

\add{Tuned hyperparameters used for \texttt{CIFAR-10} dataset in strongly-convex case shown in Figure \ref{cf4} are :  $L_2 = 0.031$ for Accelrated CBFGS, $lr = 0.00390625$ for L-BFGS, $lr = 0.001953125$ for L-SR1, $lr = 10$ for gradient descent and $L_2 = 0.00031$ for the rest of the methods.
\begin{figure*}[ht]

\includegraphics[width=0.24\textwidth]{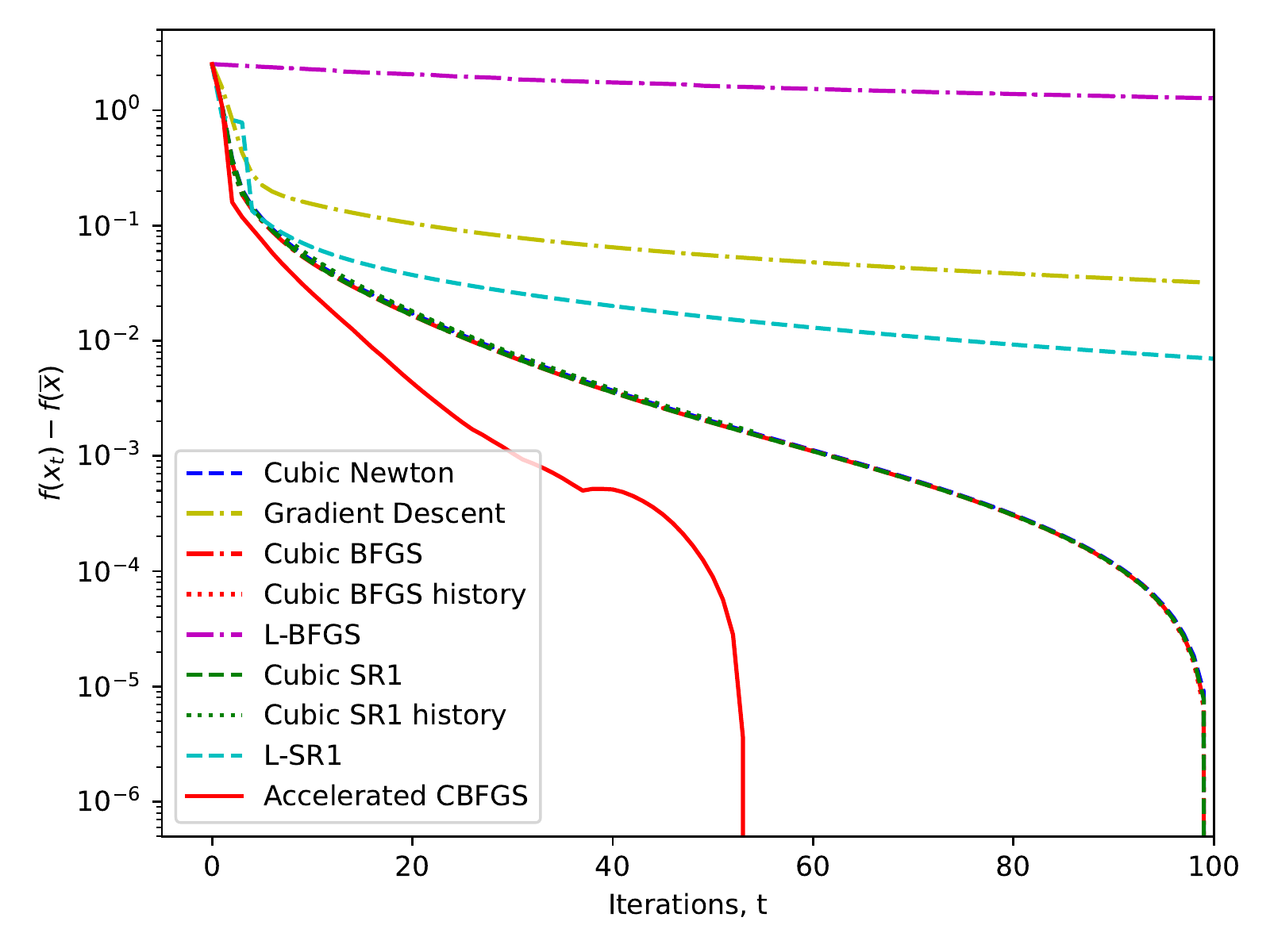}
\includegraphics[width=0.24\textwidth]{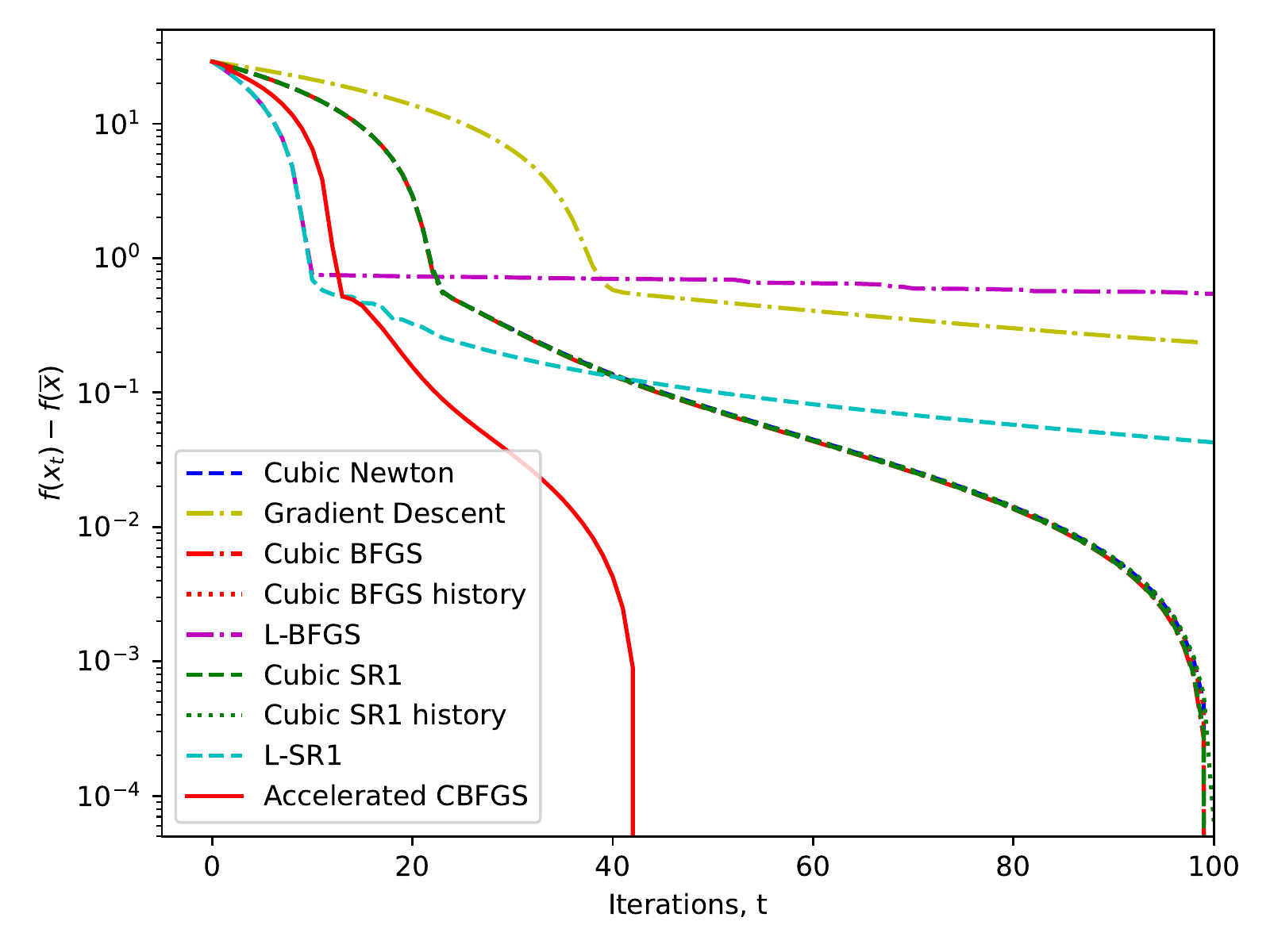}
\includegraphics[width=0.24\textwidth]{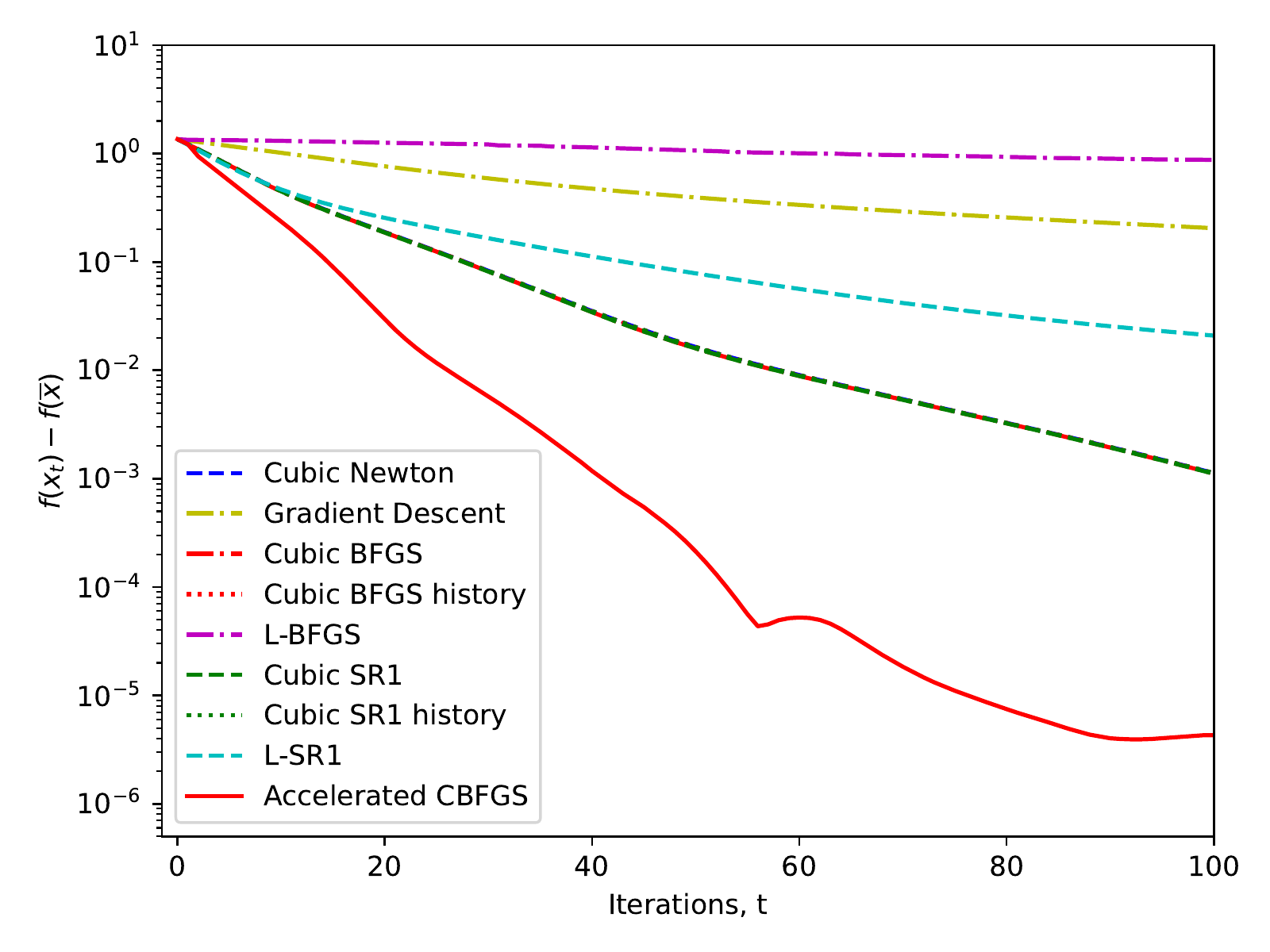}
\includegraphics[width=0.24\textwidth]{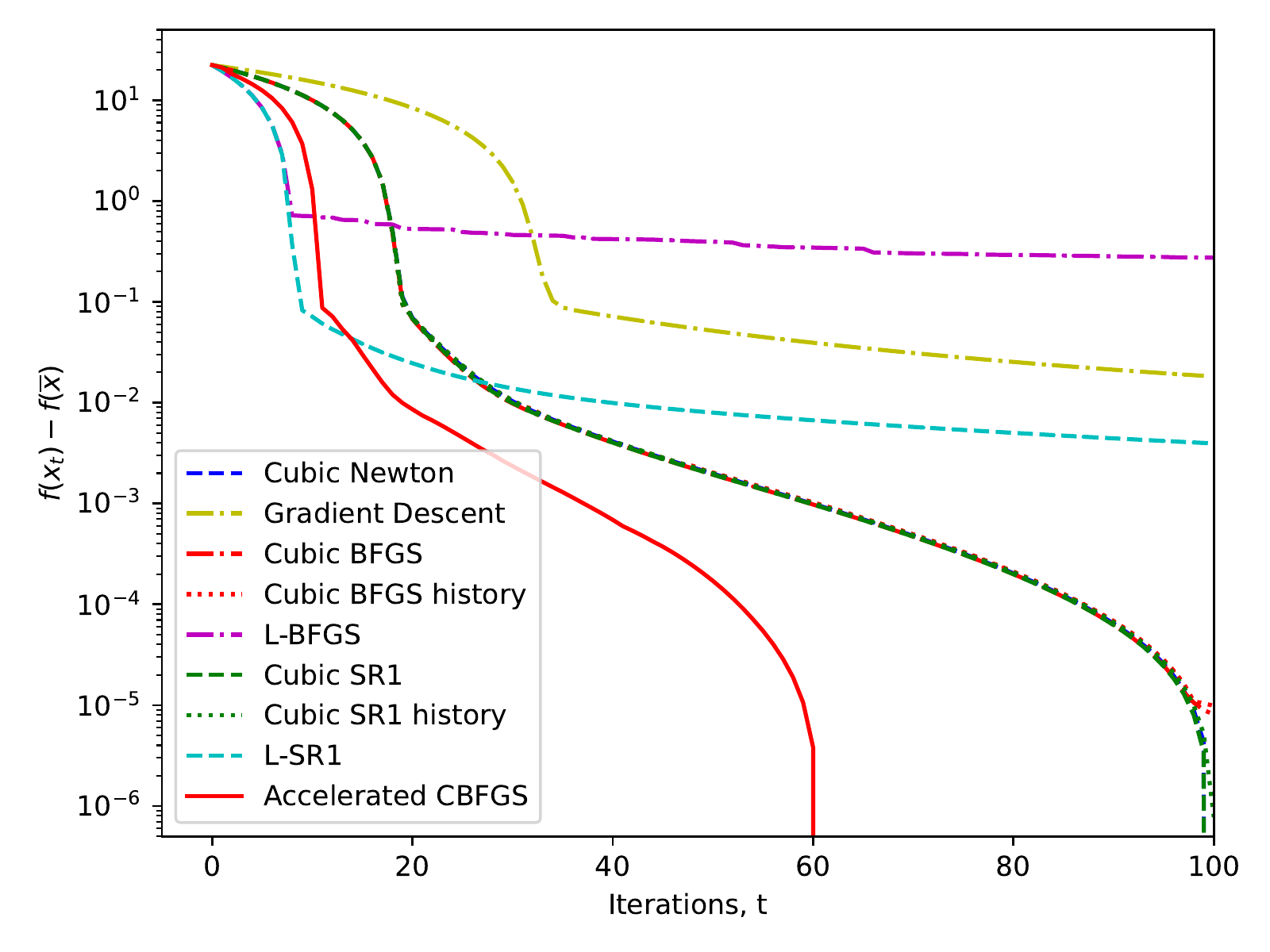}
\\
\subfigure[\texttt{a9a}]{\includegraphics[width=0.24\textwidth]{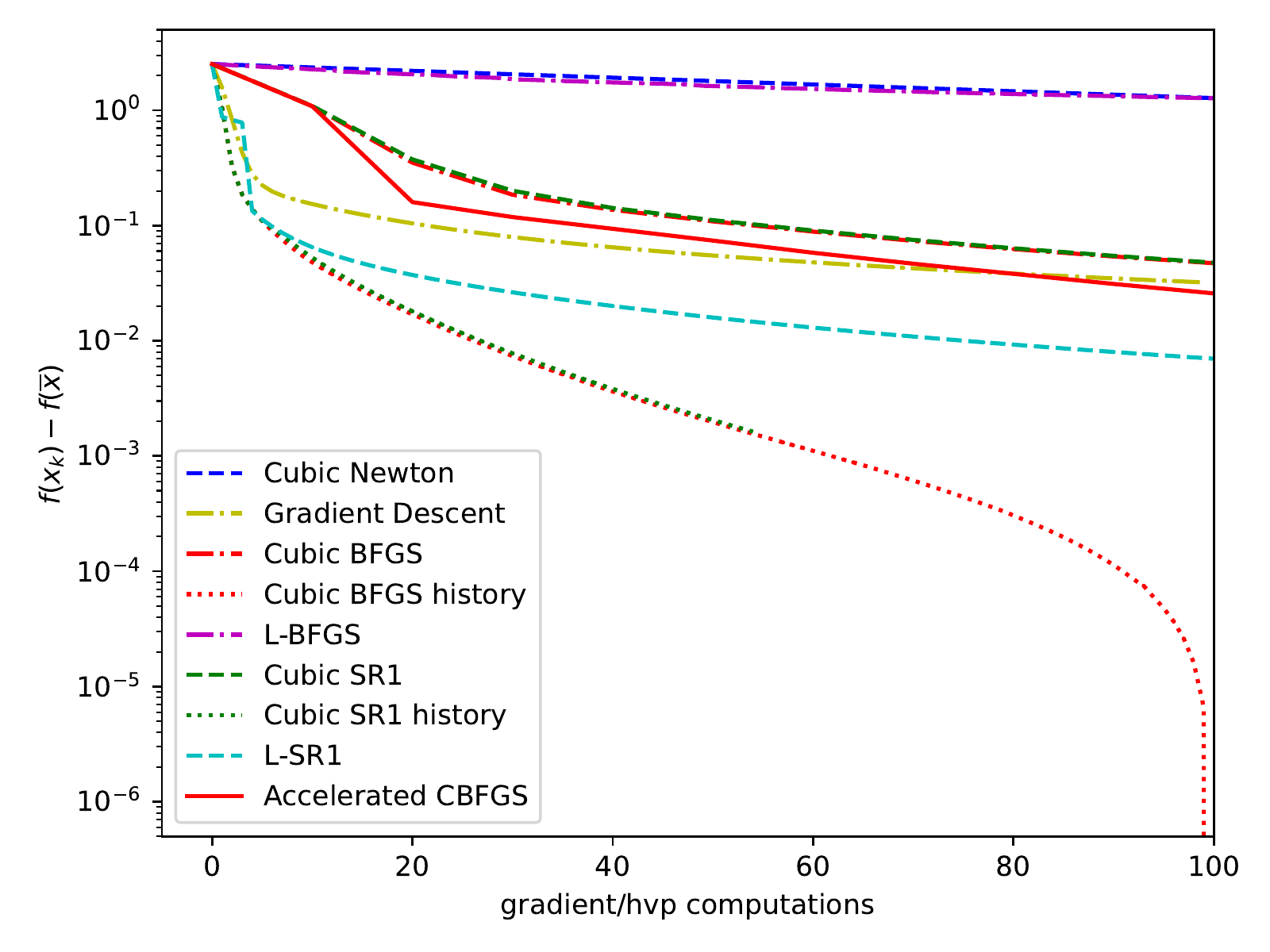}\label{a9a2}}
\subfigure[\texttt{gisette}]{\includegraphics[width=0.24\textwidth]{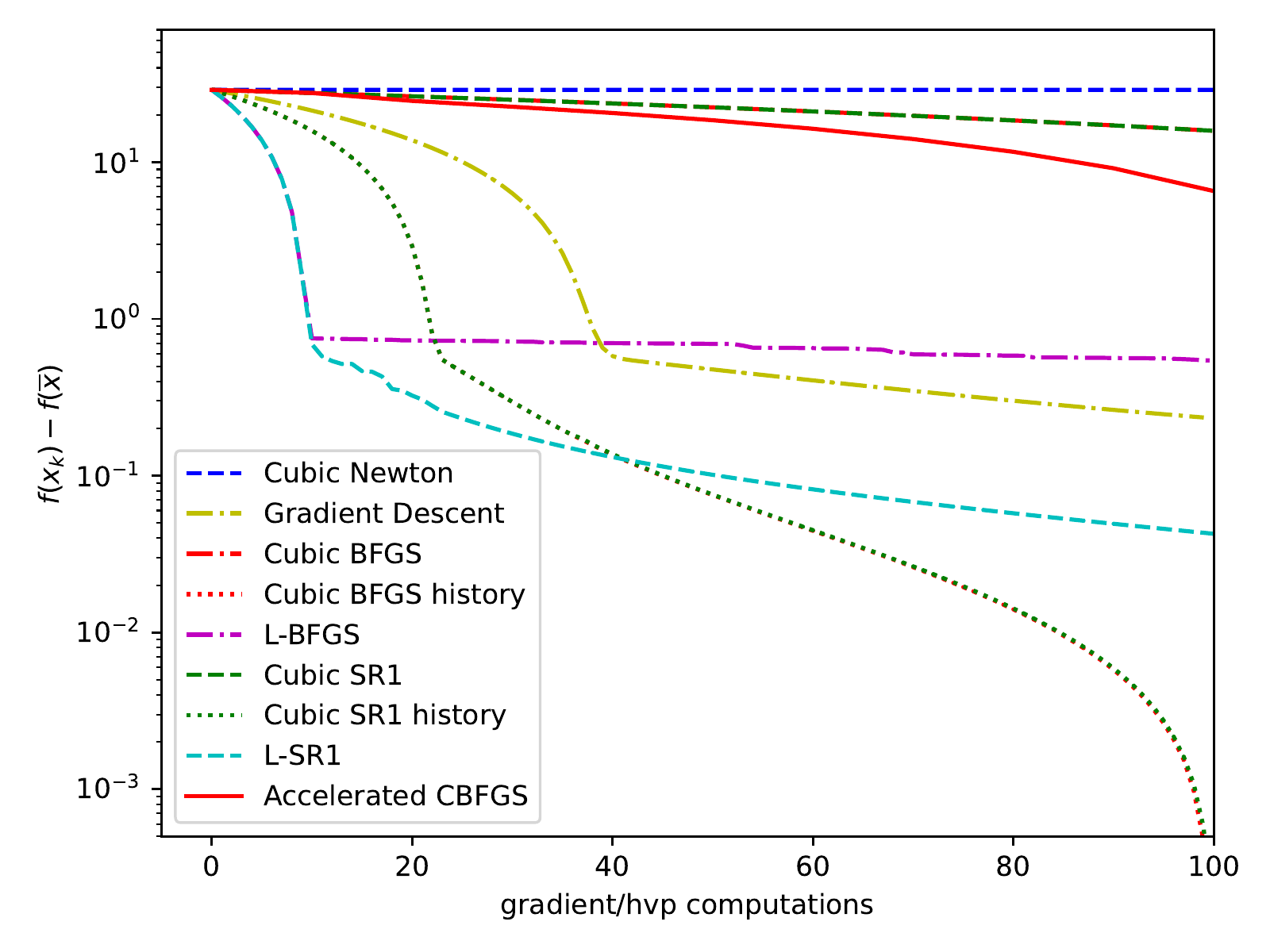}\label{gisette2}}
\subfigure[\texttt{MNIST}]{\includegraphics[width=0.24\textwidth]{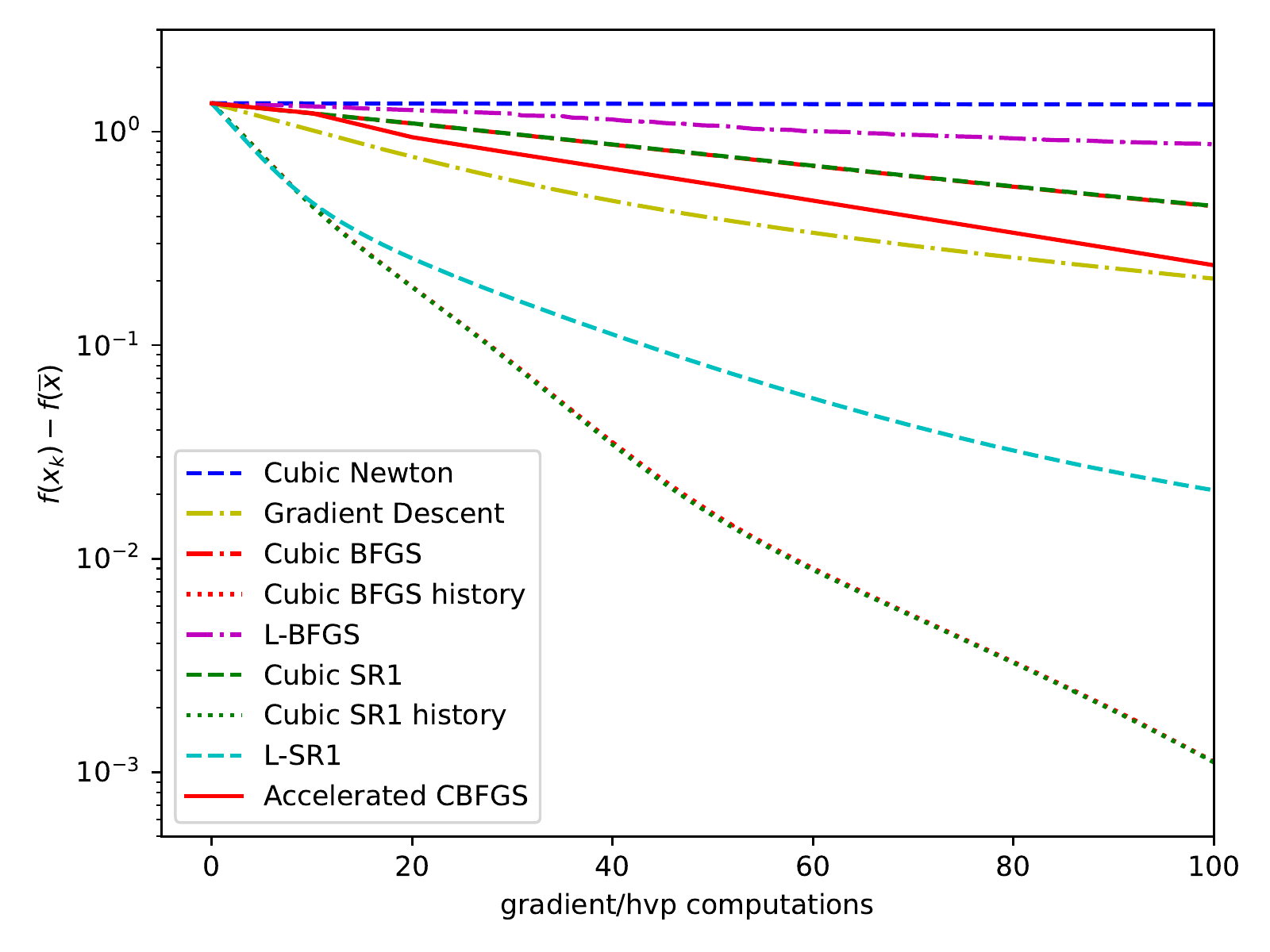}\label{mnist2}}
\subfigure[\texttt{CIFAR-10}]{\includegraphics[width=0.24\textwidth]{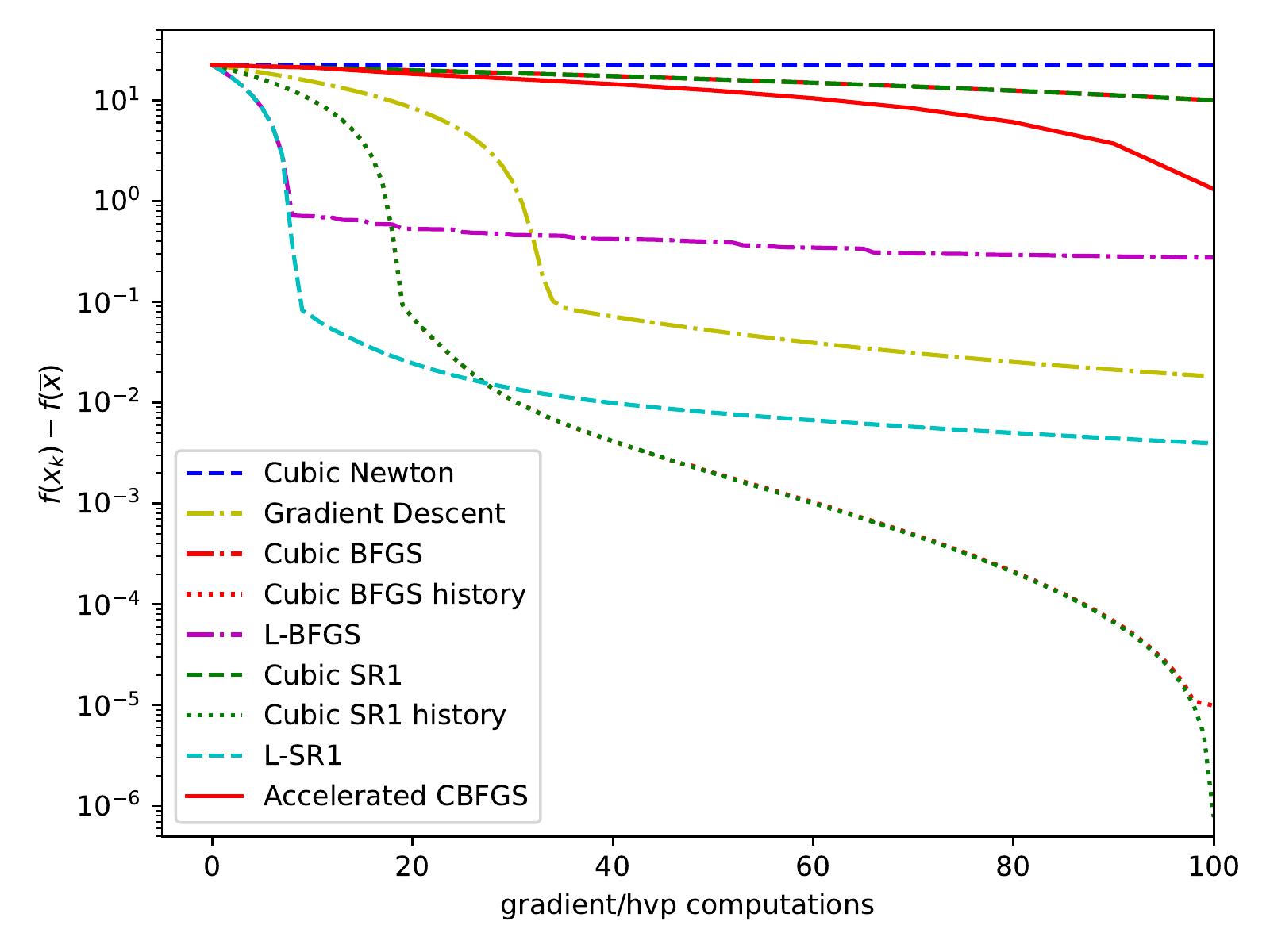}\label{cf2}}

\caption{\add{Comparison of Quasi-Newton methods and Cubic Regularized (Quasi-) Newton methods using theoretical parameters in strongly convex case on datasets \texttt{a9a}, \texttt{gisette}, \texttt{MNIST}, \texttt{CIFAR-10} respectively.}}
\label{fig2}
\vskip-10pt
\end{figure*}

\begin{figure*}[ht]
\includegraphics[width=0.24\textwidth]{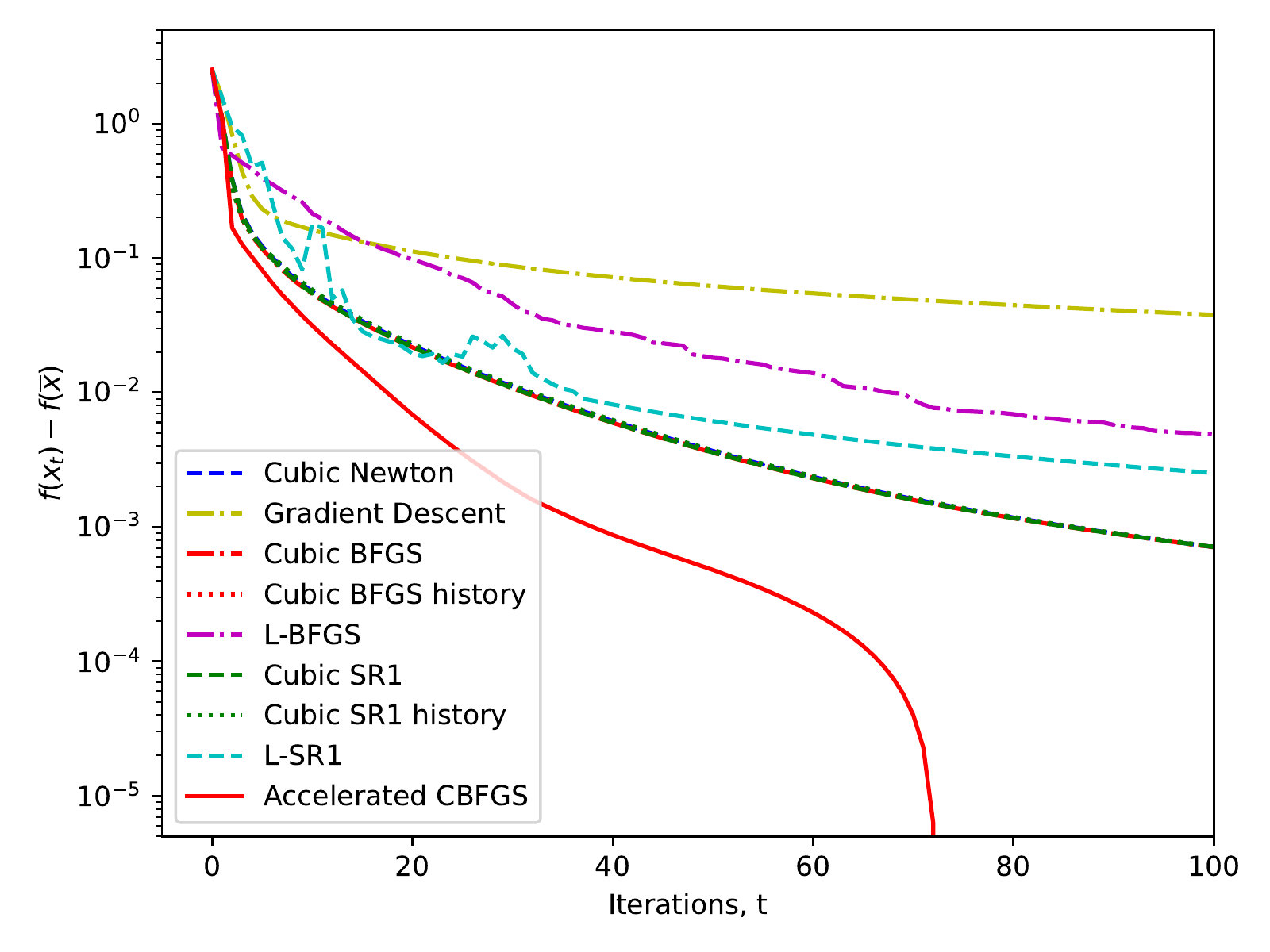}
\includegraphics[width=0.24\textwidth]{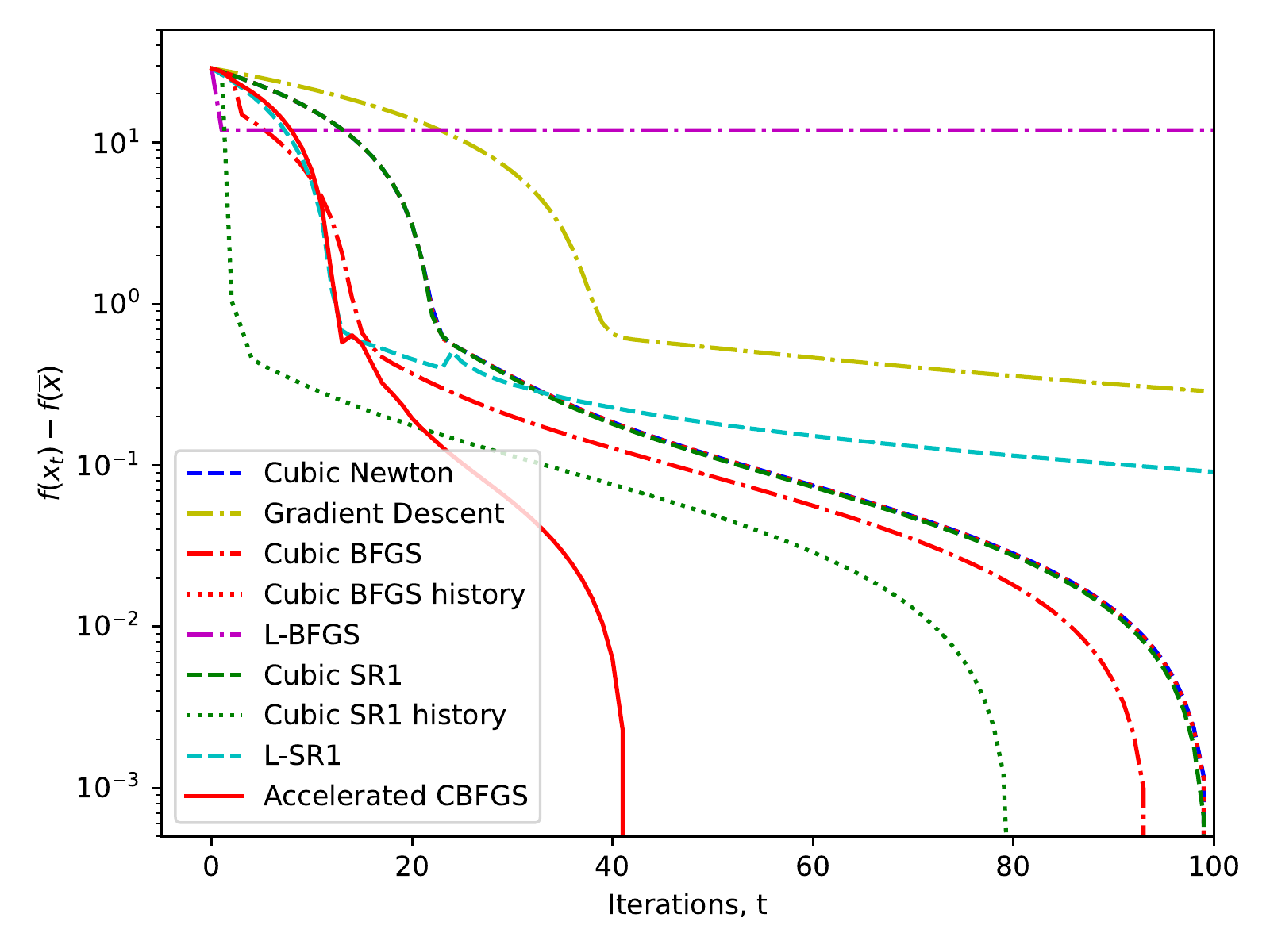}
\includegraphics[width=0.24\textwidth]{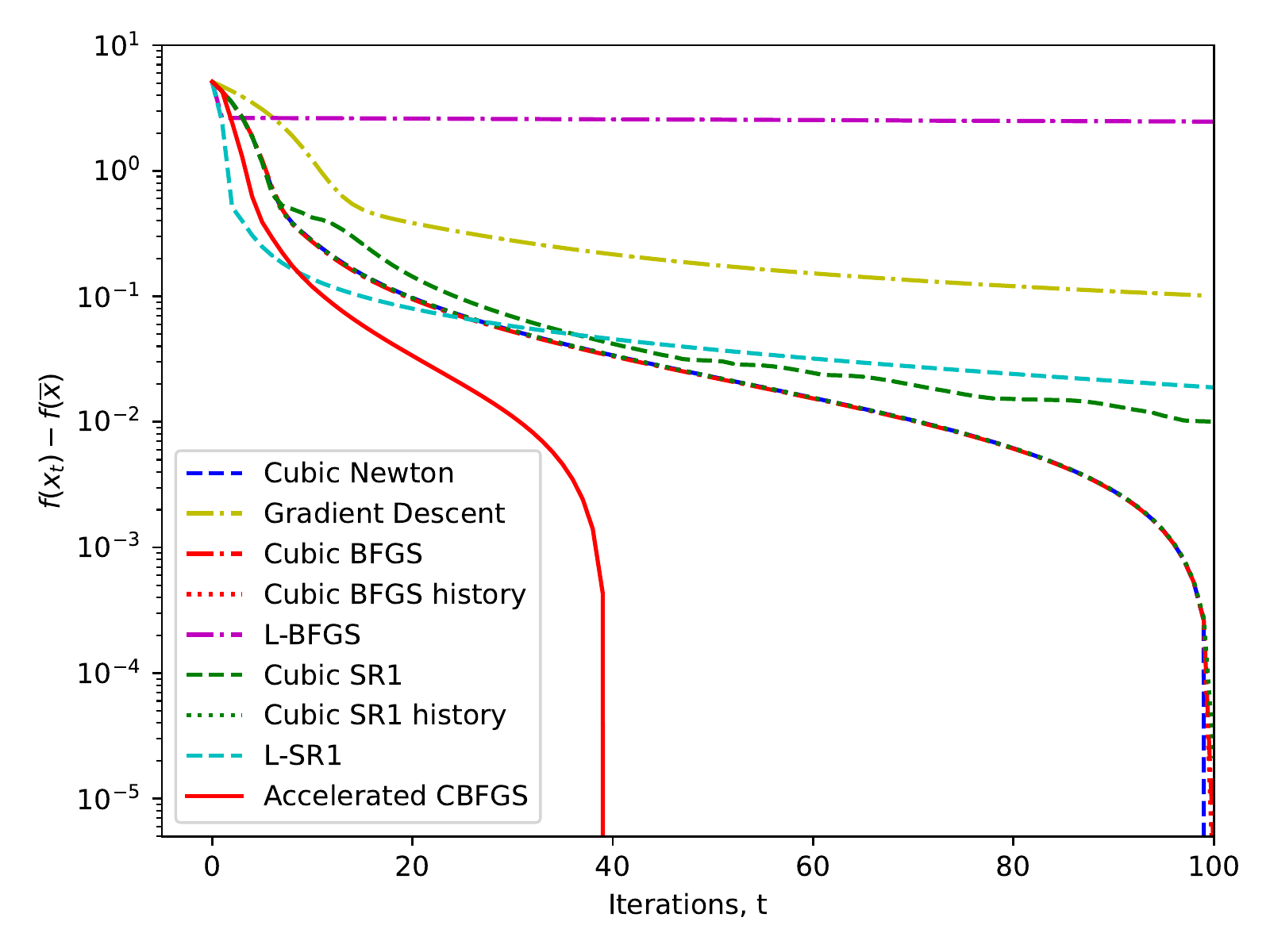}
\includegraphics[width=0.24\textwidth]
{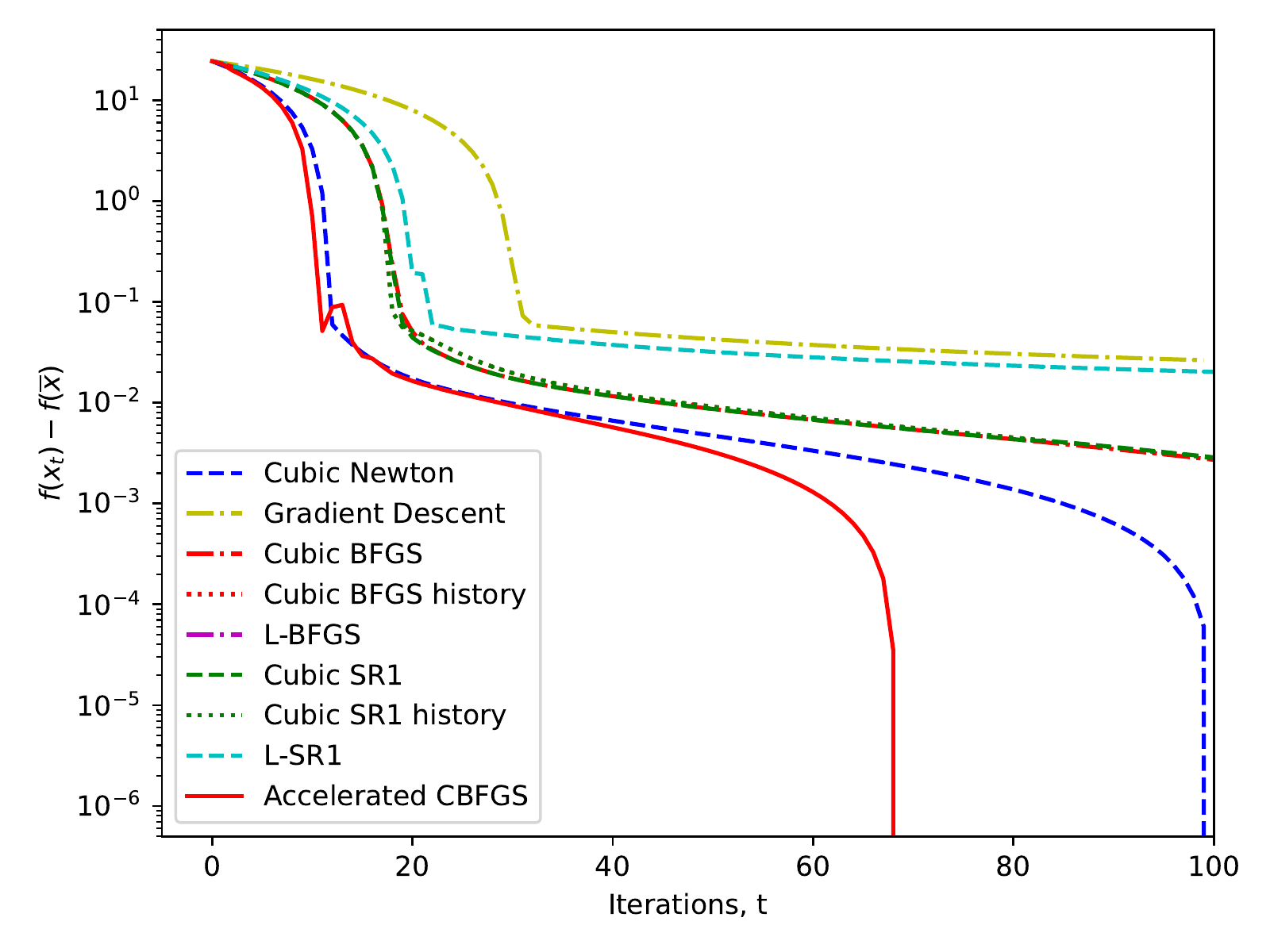}
\\
\subfigure[\texttt{a9a}]{\includegraphics[width=0.24\textwidth]{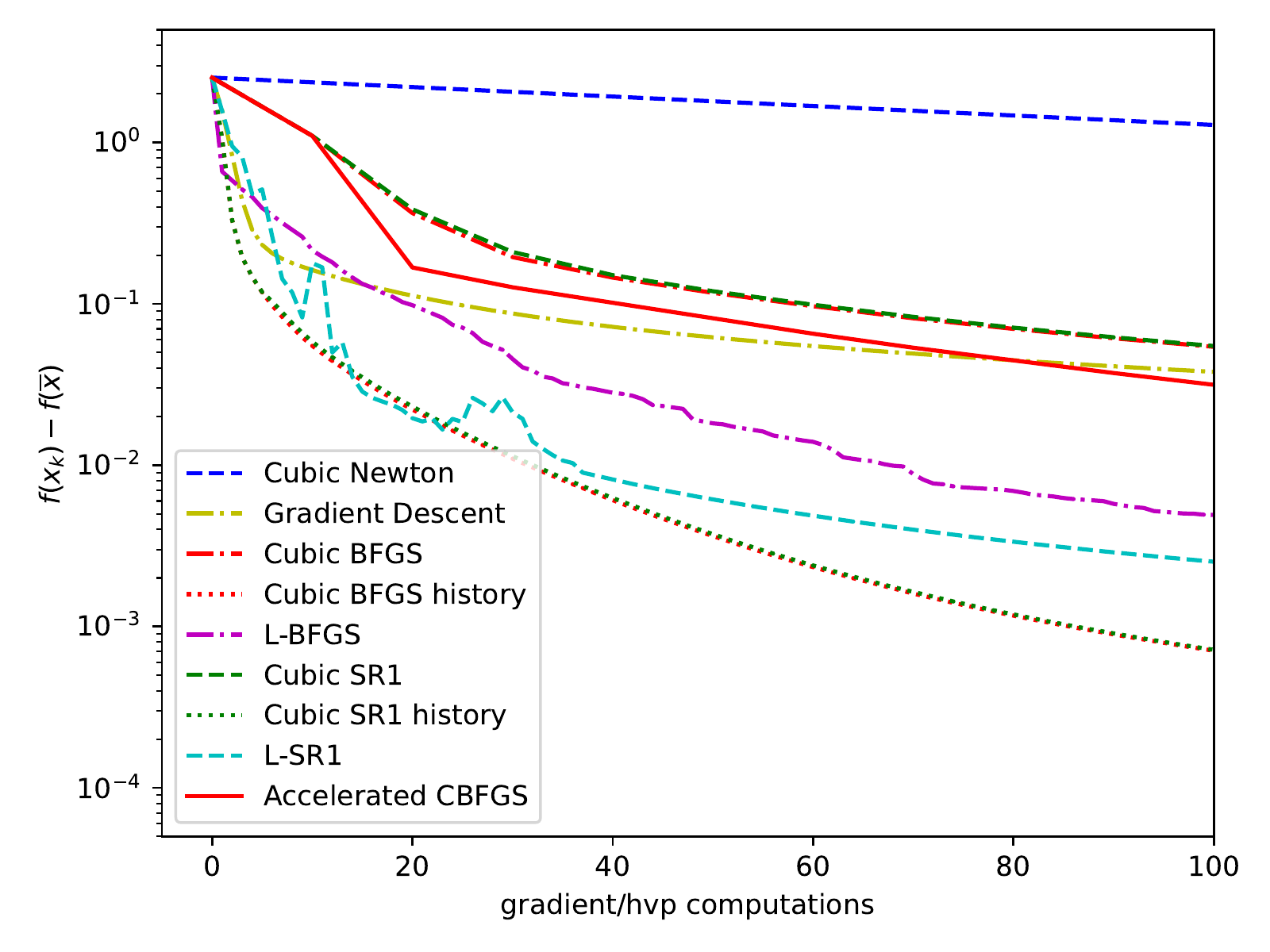}\label{a9a3}}
\subfigure[\texttt{gisette}]{\includegraphics[width=0.24\textwidth]{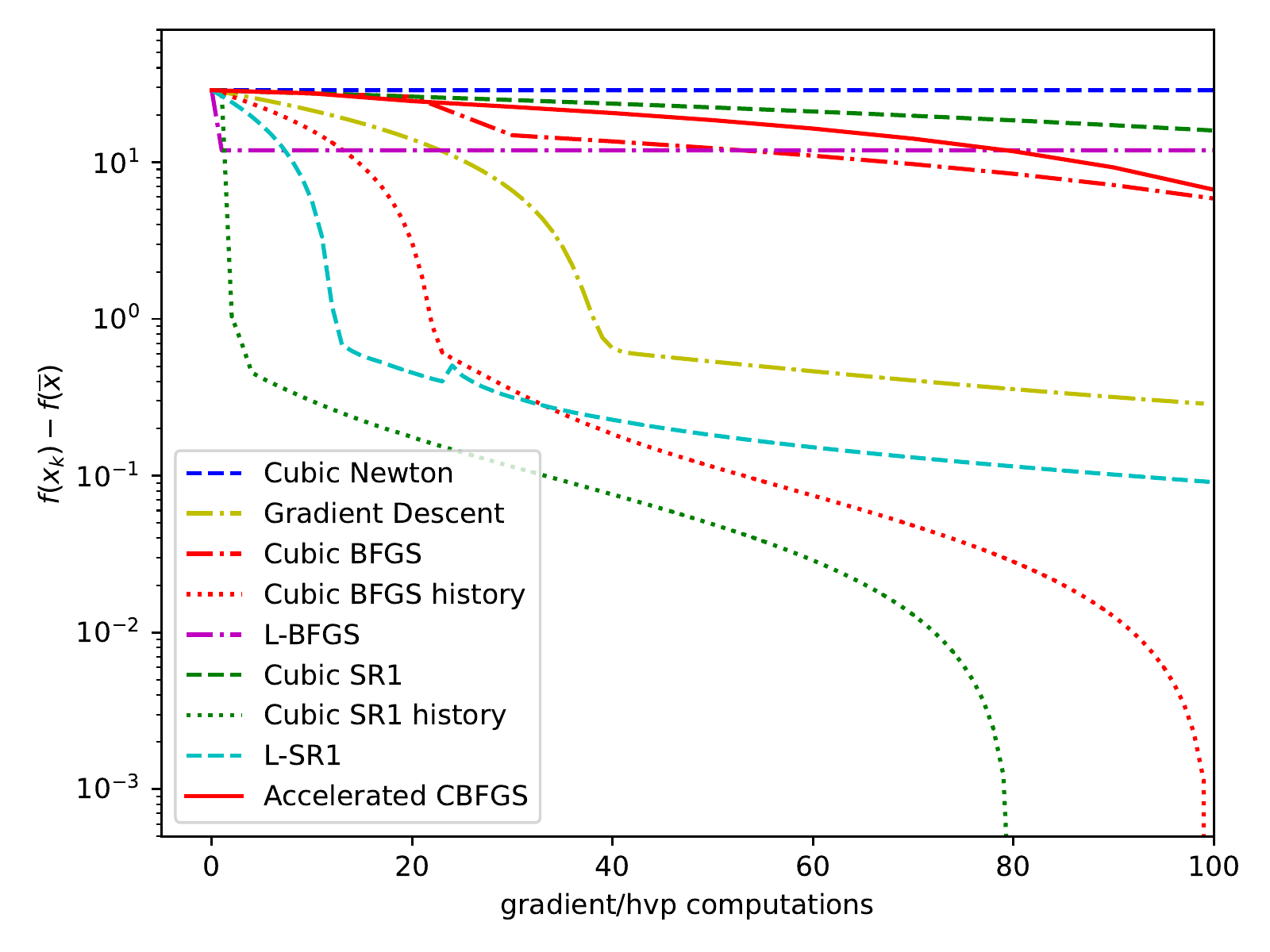}\label{gisette3}}
\subfigure[\texttt{MNIST}]{\includegraphics[width=0.24\textwidth]{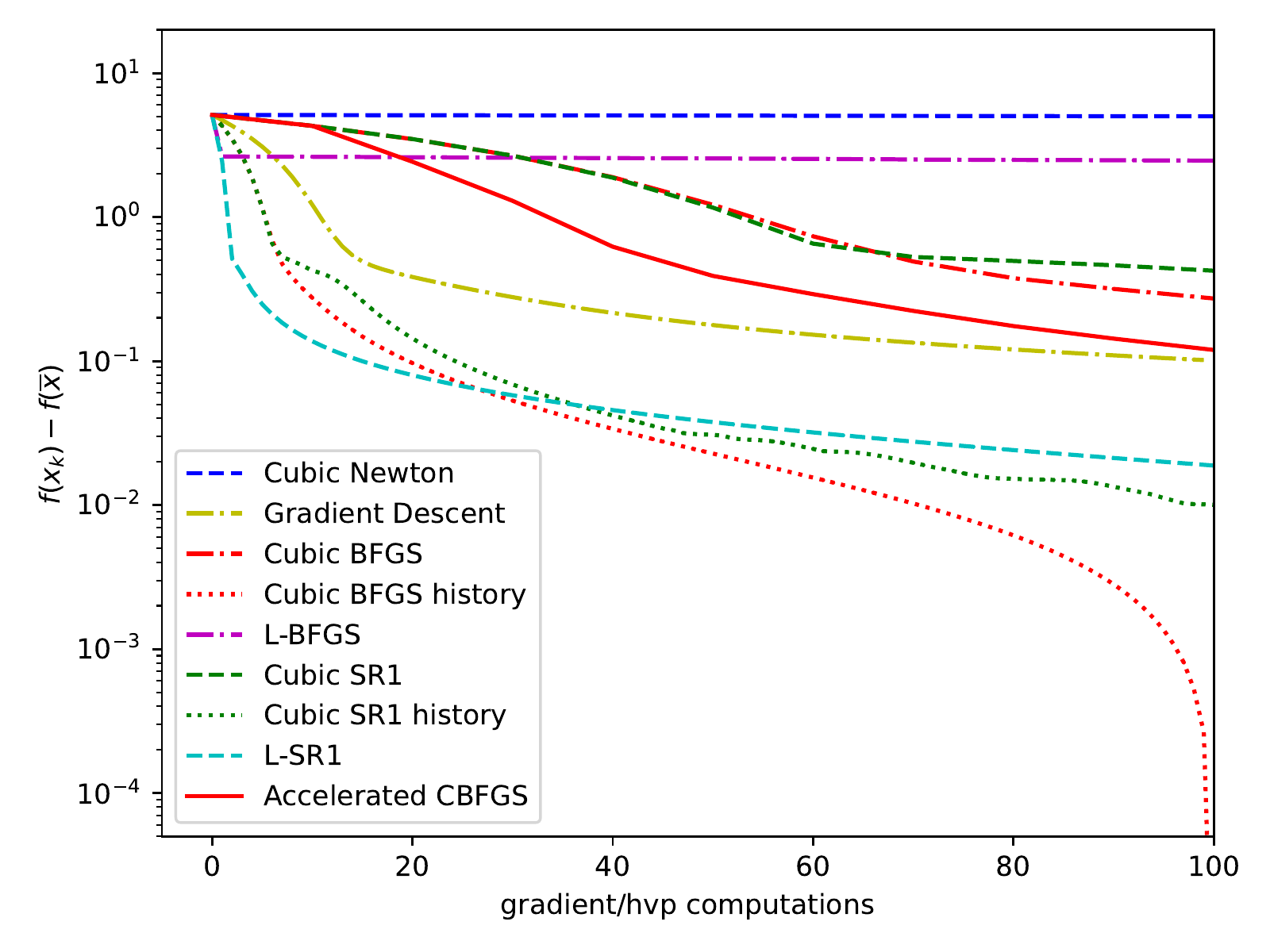}\label{mnist3}}
\subfigure[\texttt{CIFAR-10}]{\includegraphics[width=0.24\textwidth]{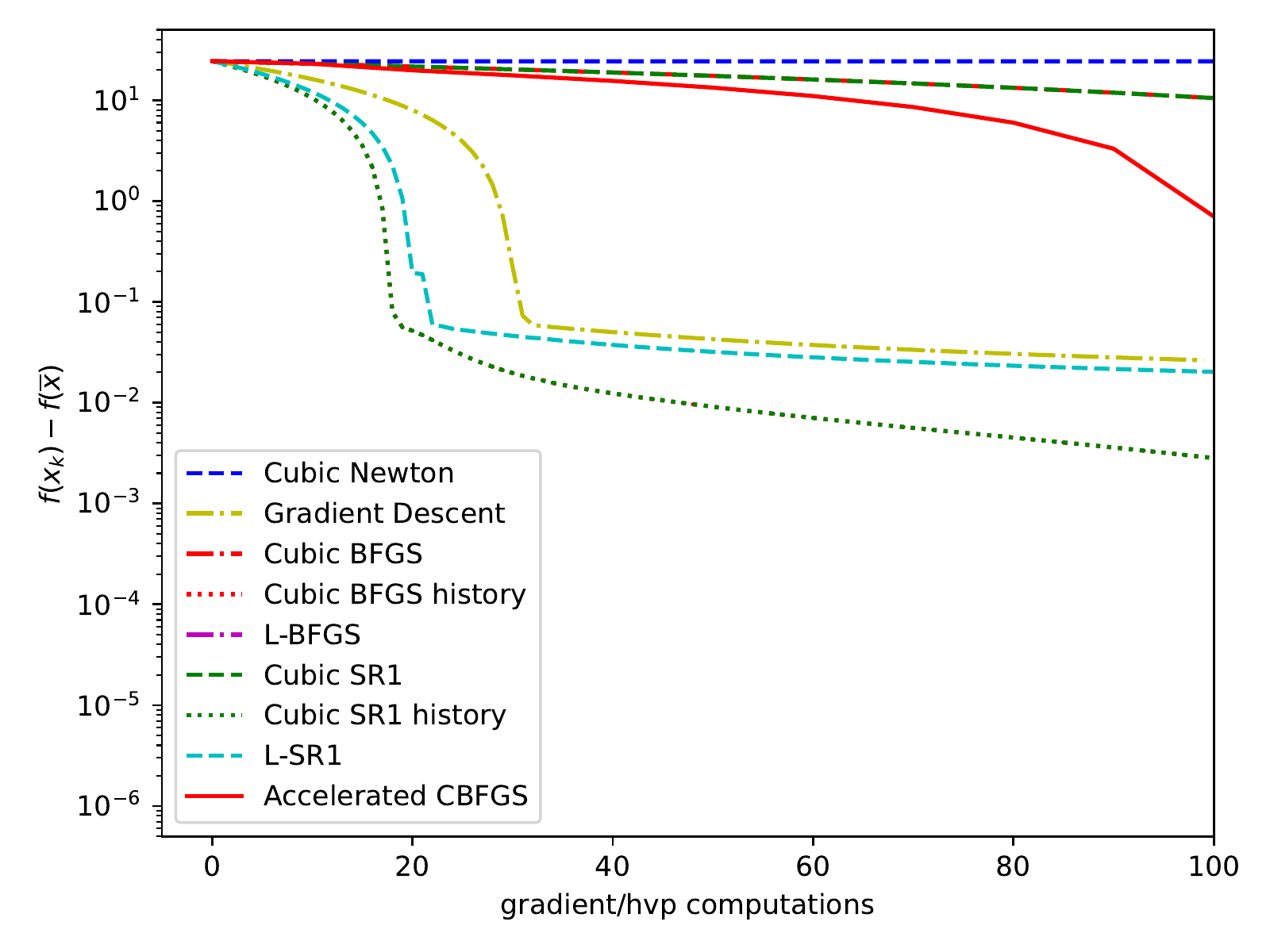}\label{cf3}}
\caption{\add{Comparison of Quasi-Newton methods and Cubic Regularized (Quasi-) Newton methods using theoretical parameters in the convex case on datasets \texttt{a9a}, \texttt{gisette}, \texttt{MNIST}, \texttt{CIFAR-10} respectively.}}
\label{fig3}

\end{figure*}

\begin{figure*}[ht]
\includegraphics[width=0.24\textwidth]{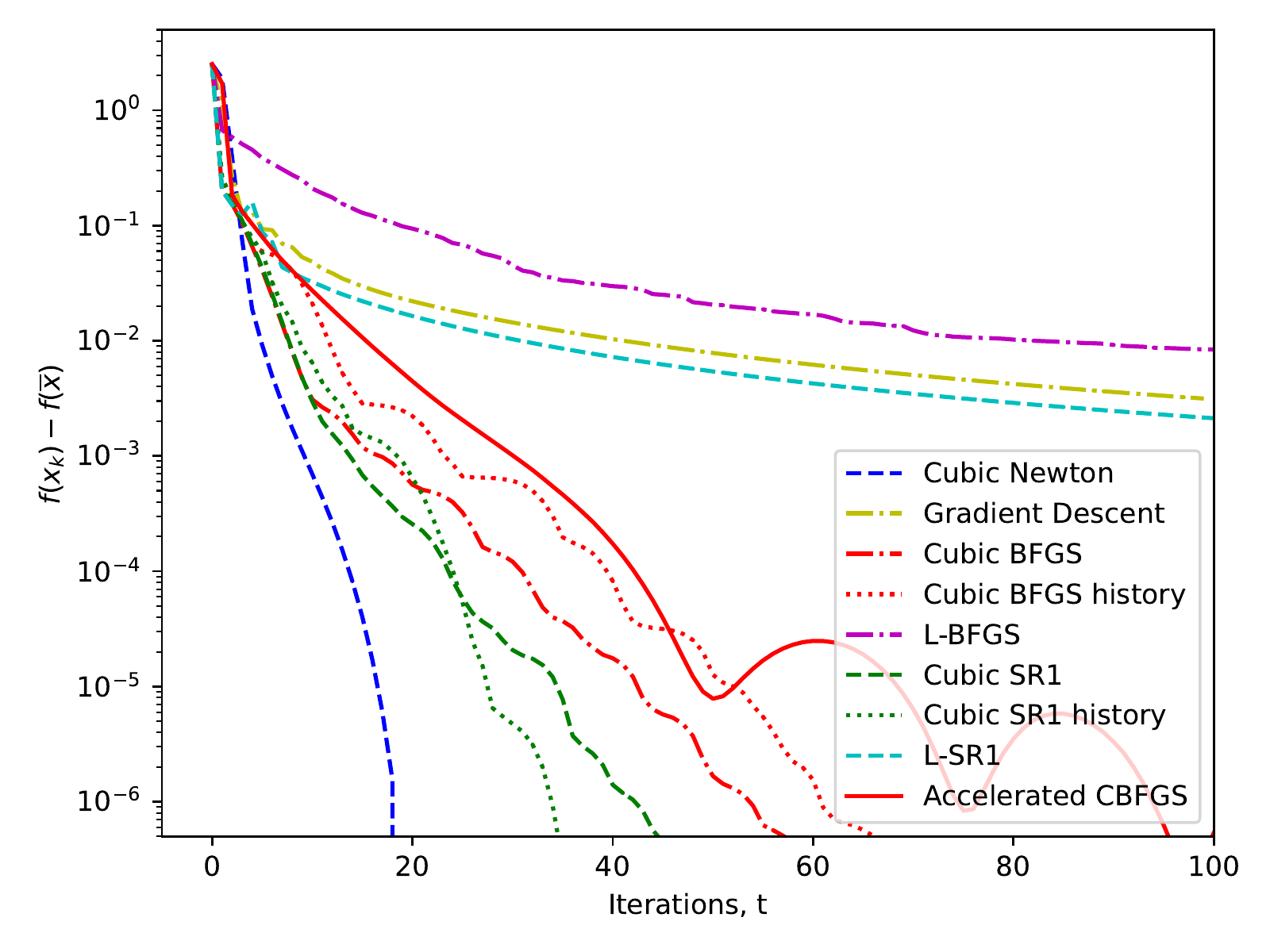}
\includegraphics[width=0.24\textwidth]{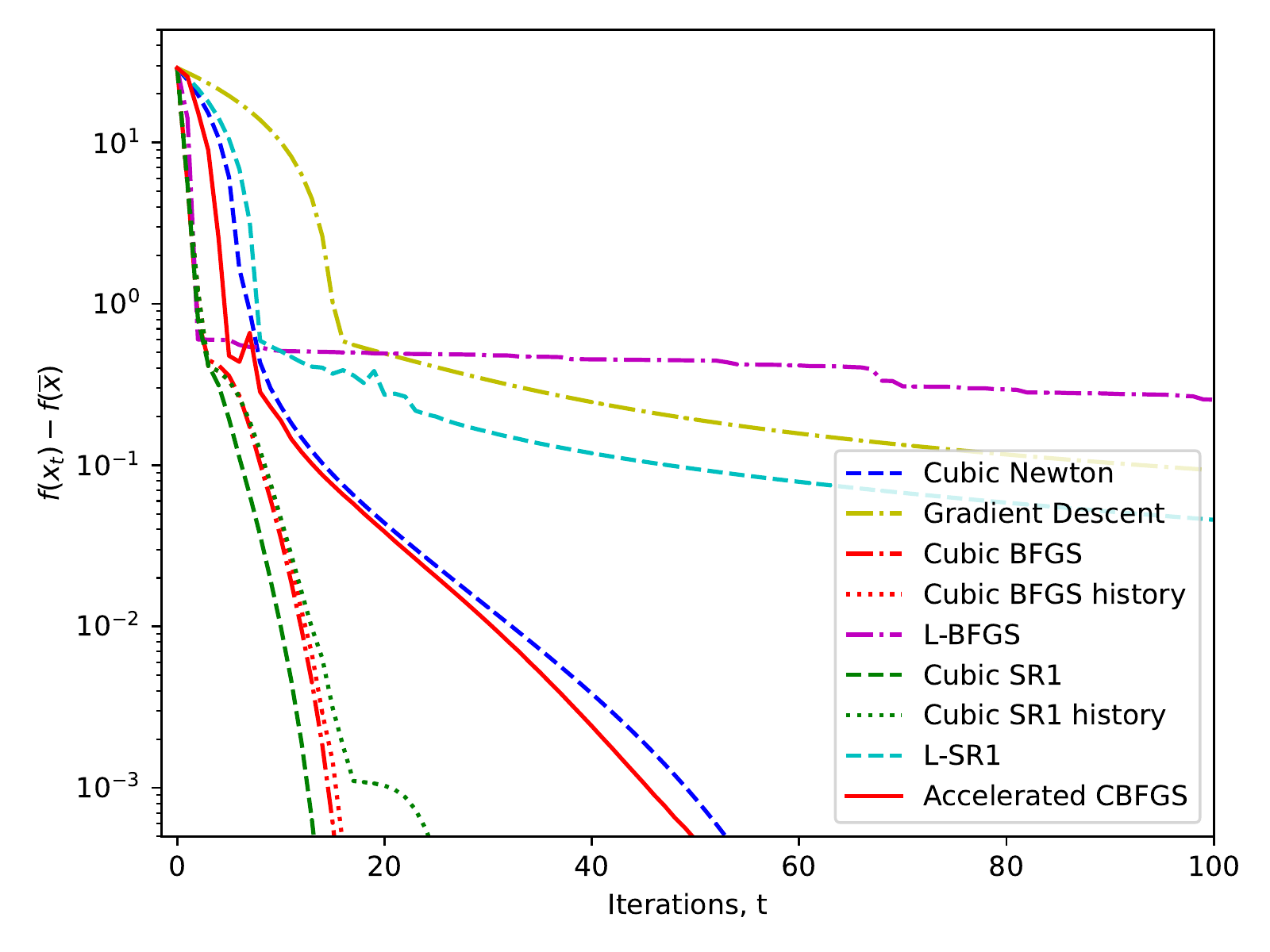}
\includegraphics[width=0.24\textwidth]
{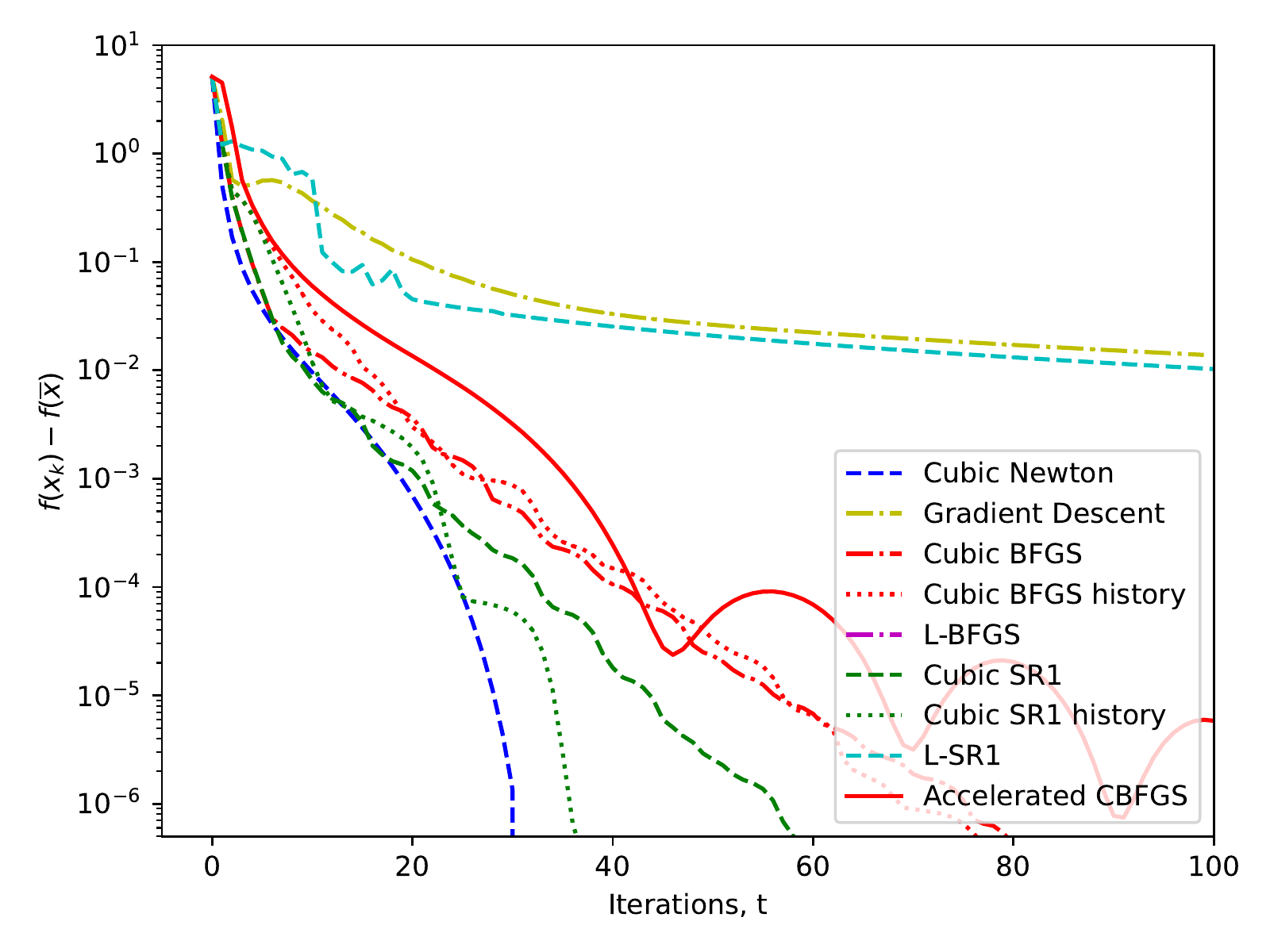}
\includegraphics[width=0.24\textwidth]
{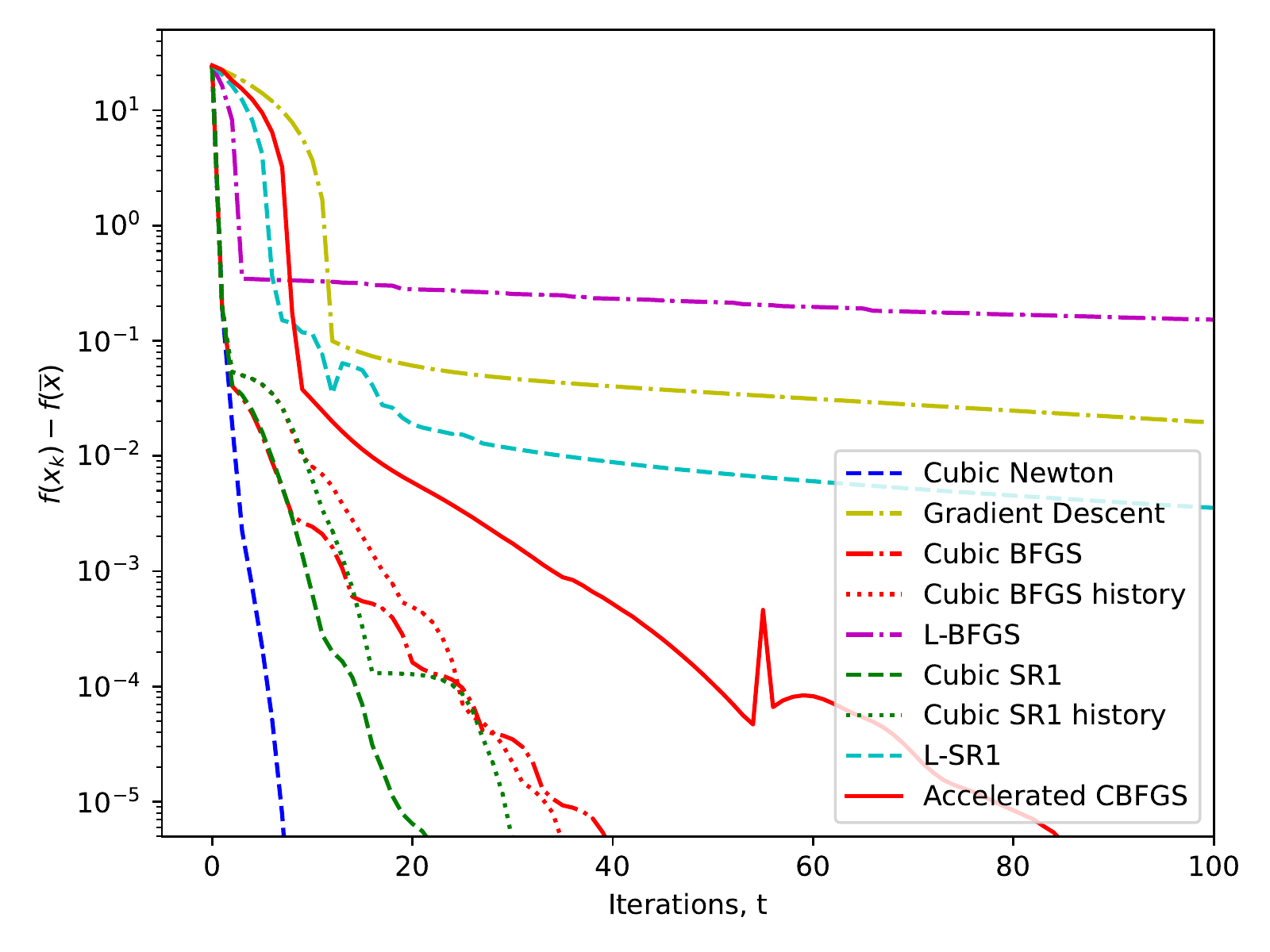}
\\
\subfigure[\texttt{a9a}]{\includegraphics[width=0.24\textwidth]{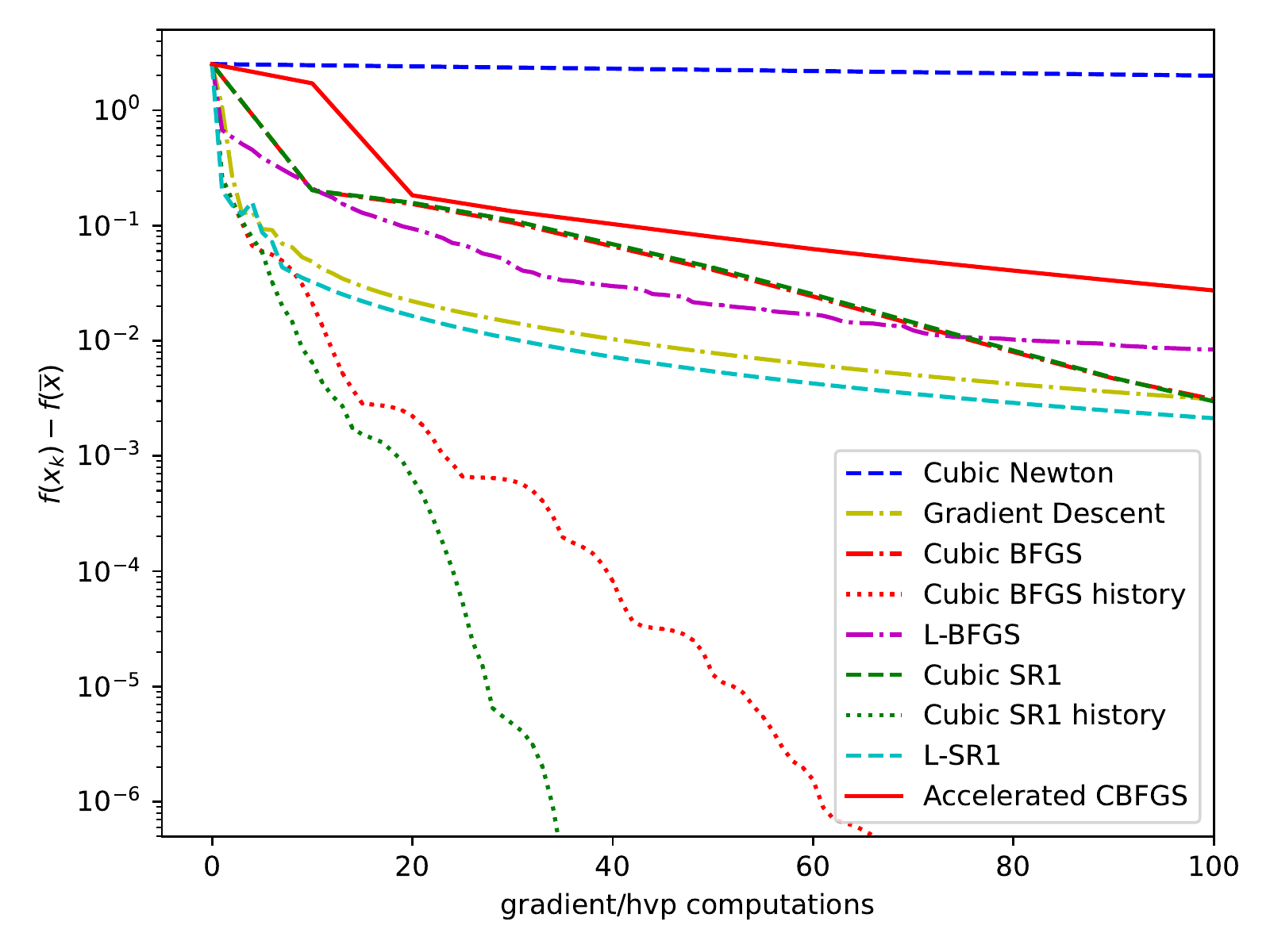}\label{a9a4}}
\subfigure[\texttt{gisette}]{\includegraphics[width=0.24\textwidth]{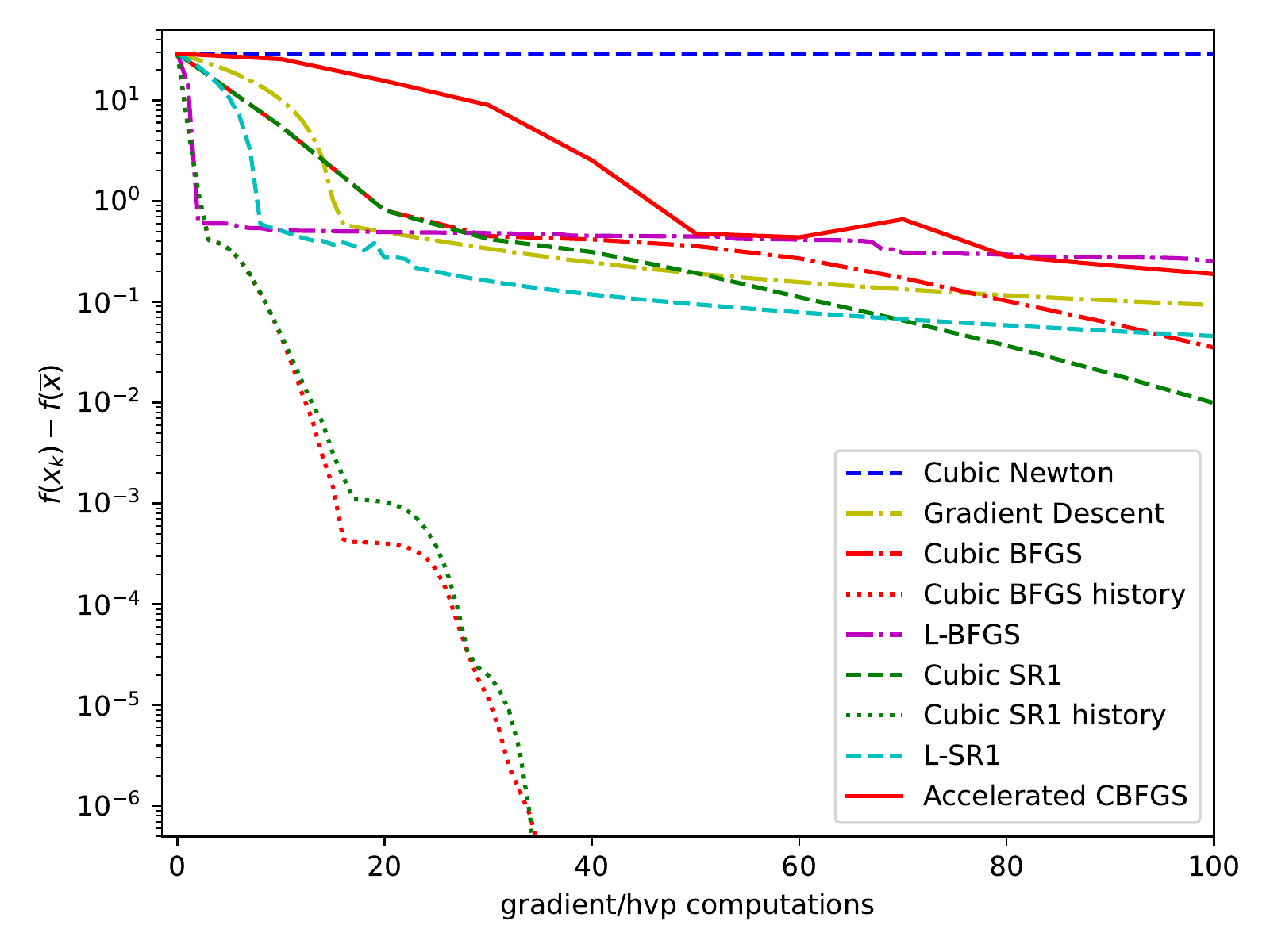}\label{gisette4}}
\subfigure[\texttt{MNIST}]{\includegraphics[width=0.24\textwidth]{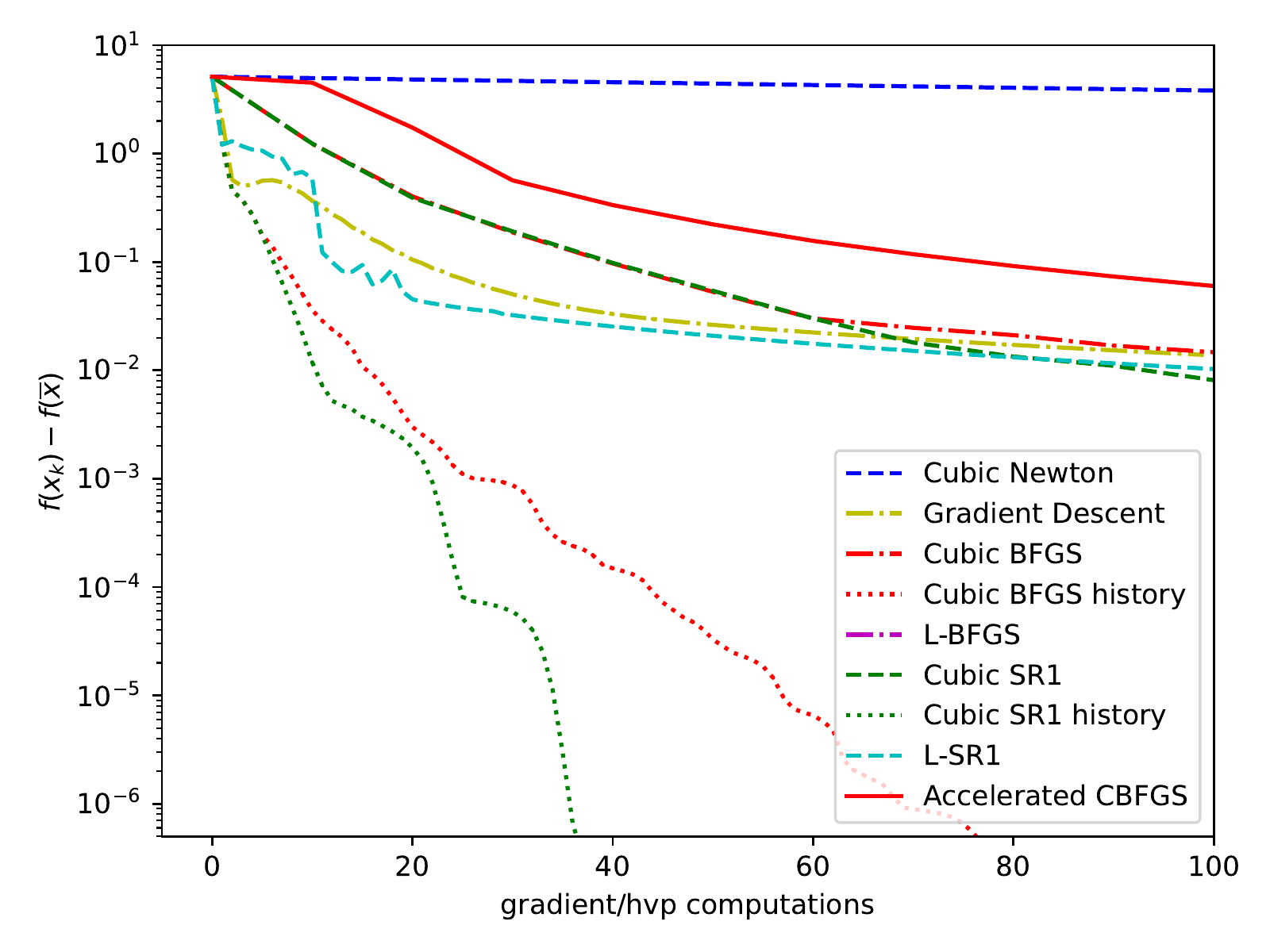}\label{mnist4}}
\subfigure[\texttt{CIFAR-10}]{\includegraphics[width=0.24\textwidth]{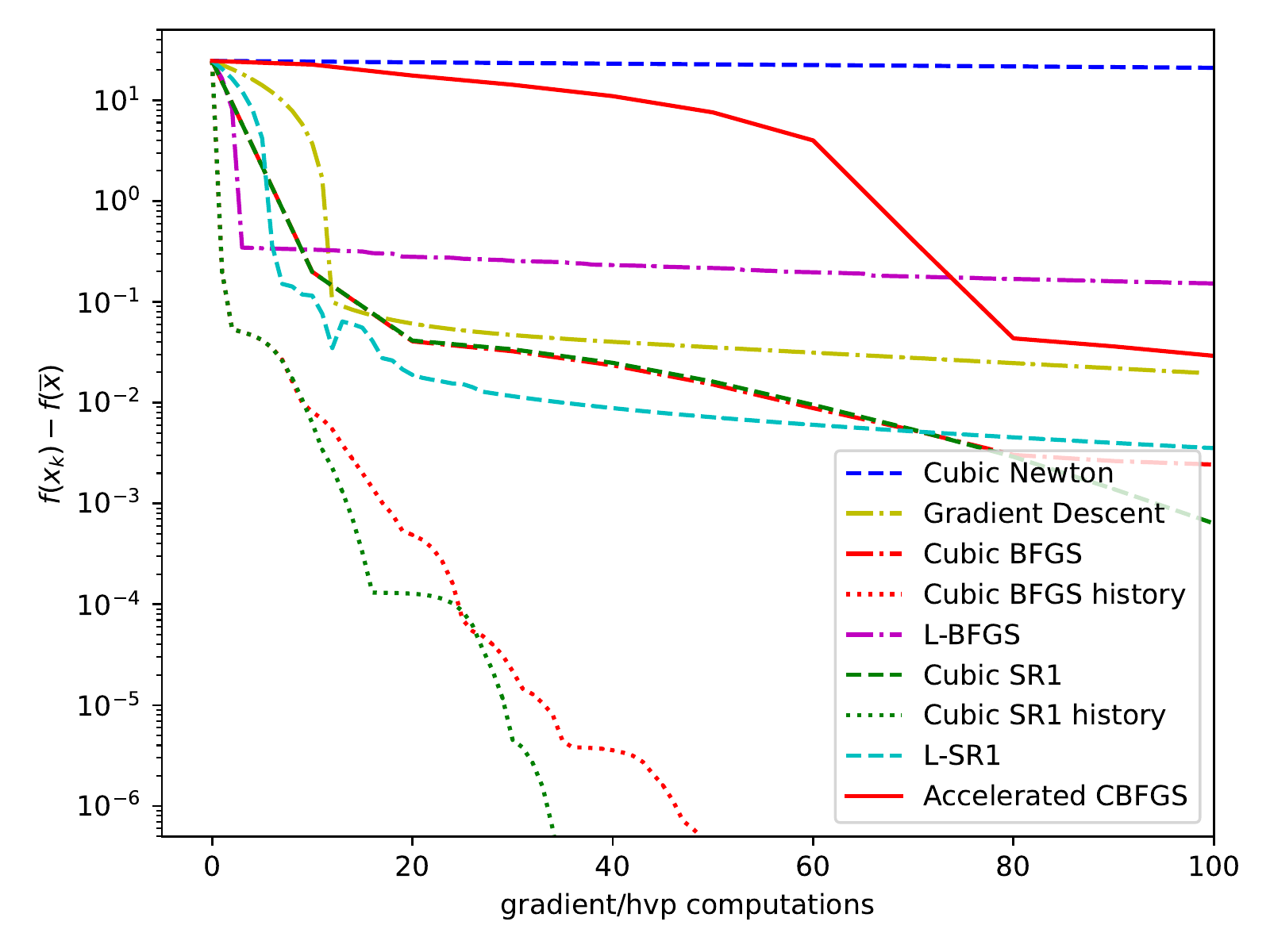}\label{cf4}}
\caption{Comparison of Quasi-Newton methods and Cubic Regularized (Quasi-) Newton methods using scaled parameters in strongly convex case on datasets  \texttt{a9a}, \texttt{gisette}, \texttt{MNIST}, \texttt{CIFAR-10} respectively.}
\label{fig4}
\vskip-10pt
\end{figure*}

\begin{figure*}[ht]
\includegraphics[width=0.24\textwidth]{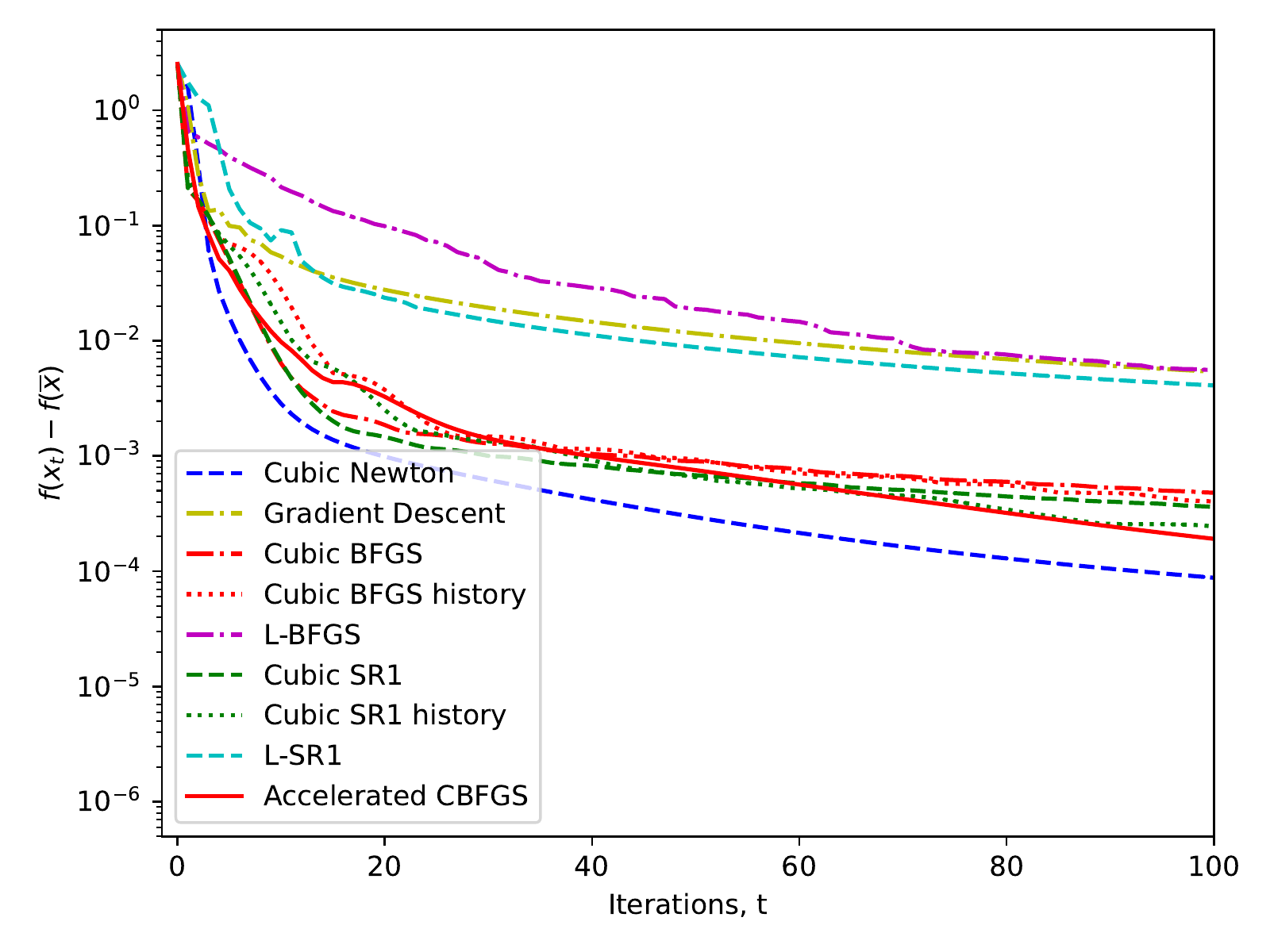}
\includegraphics[width=0.24\textwidth]{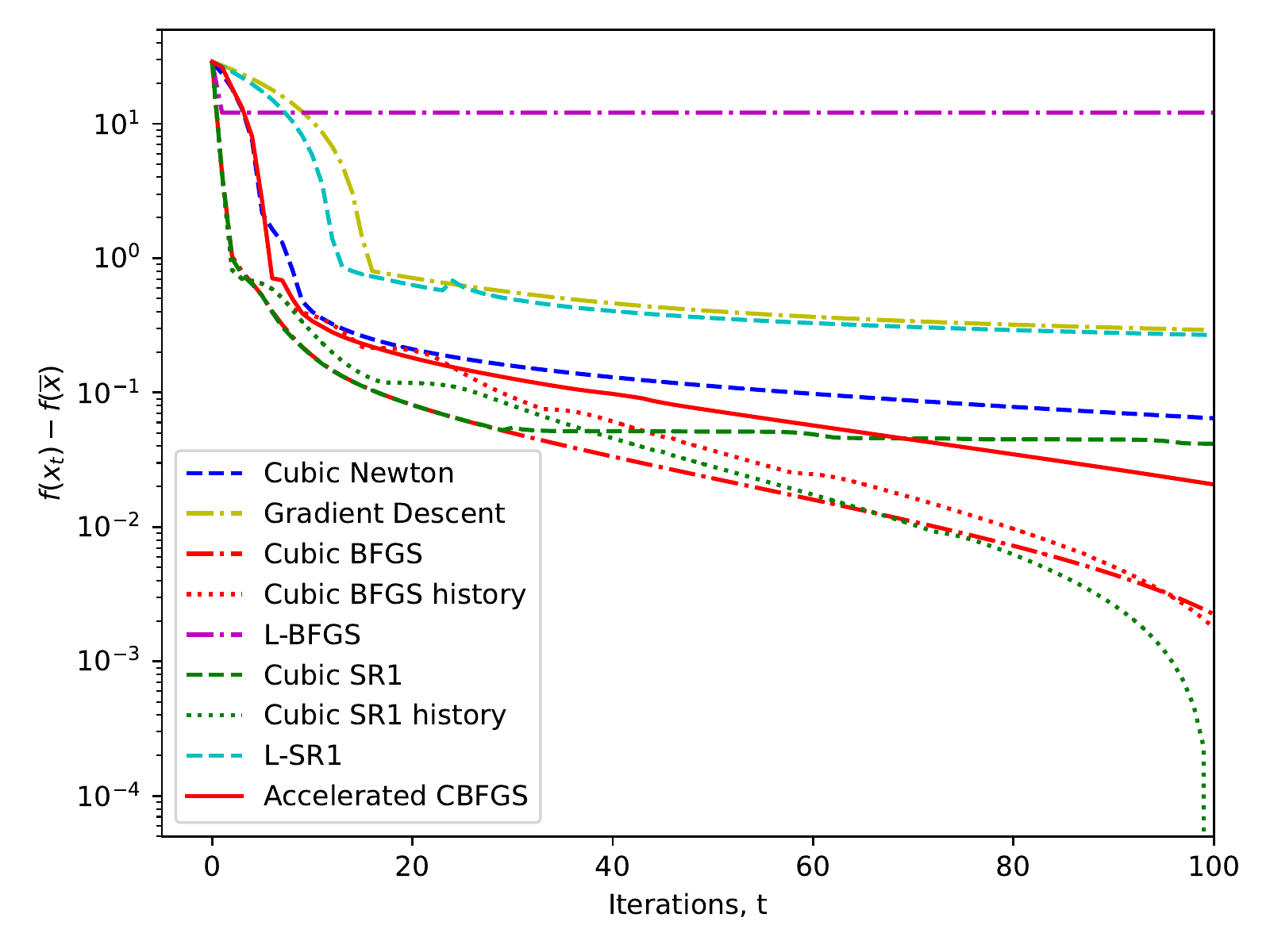}
\includegraphics[width=0.24\textwidth]
{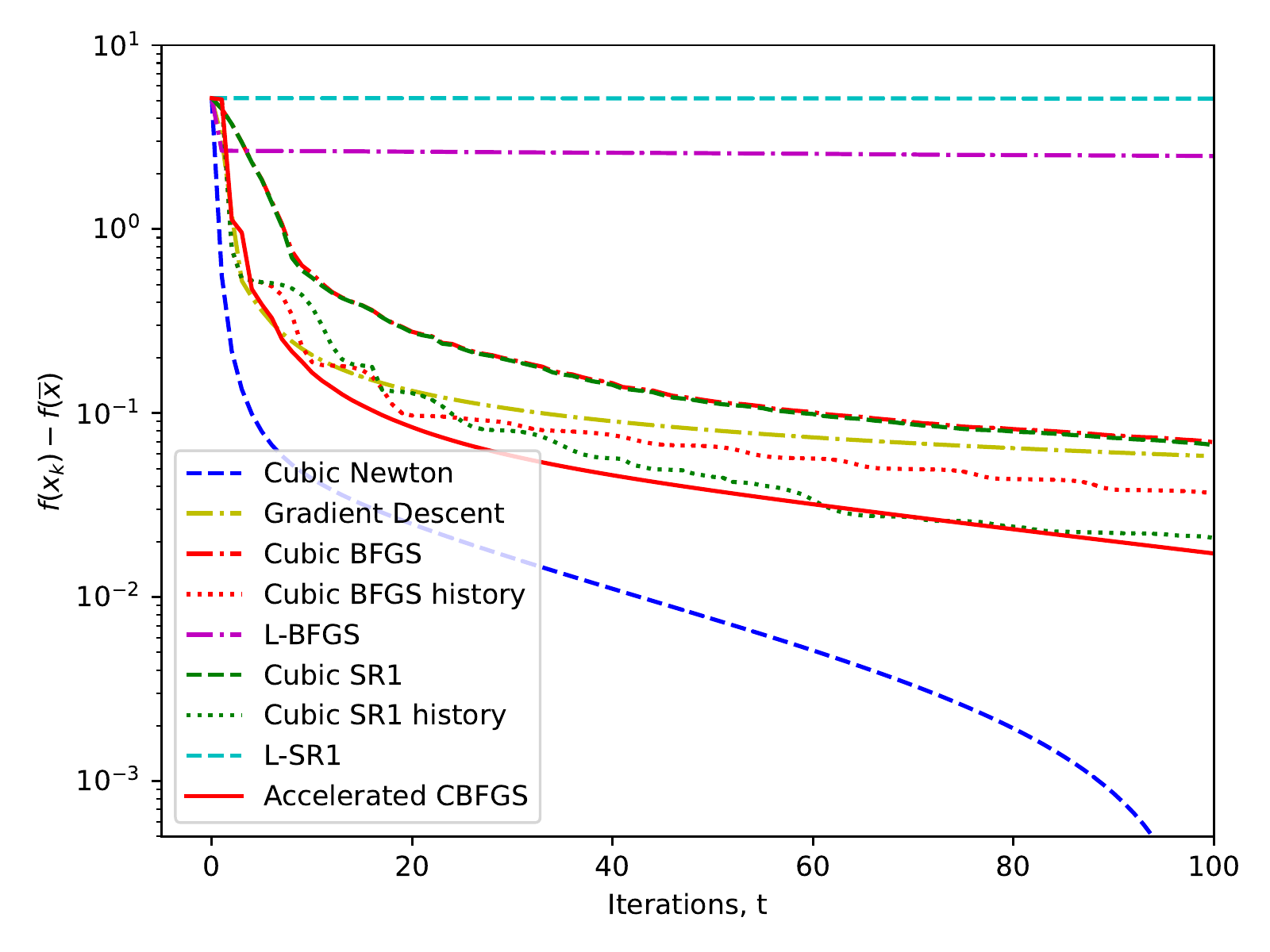}
\includegraphics[width=0.24\textwidth]
{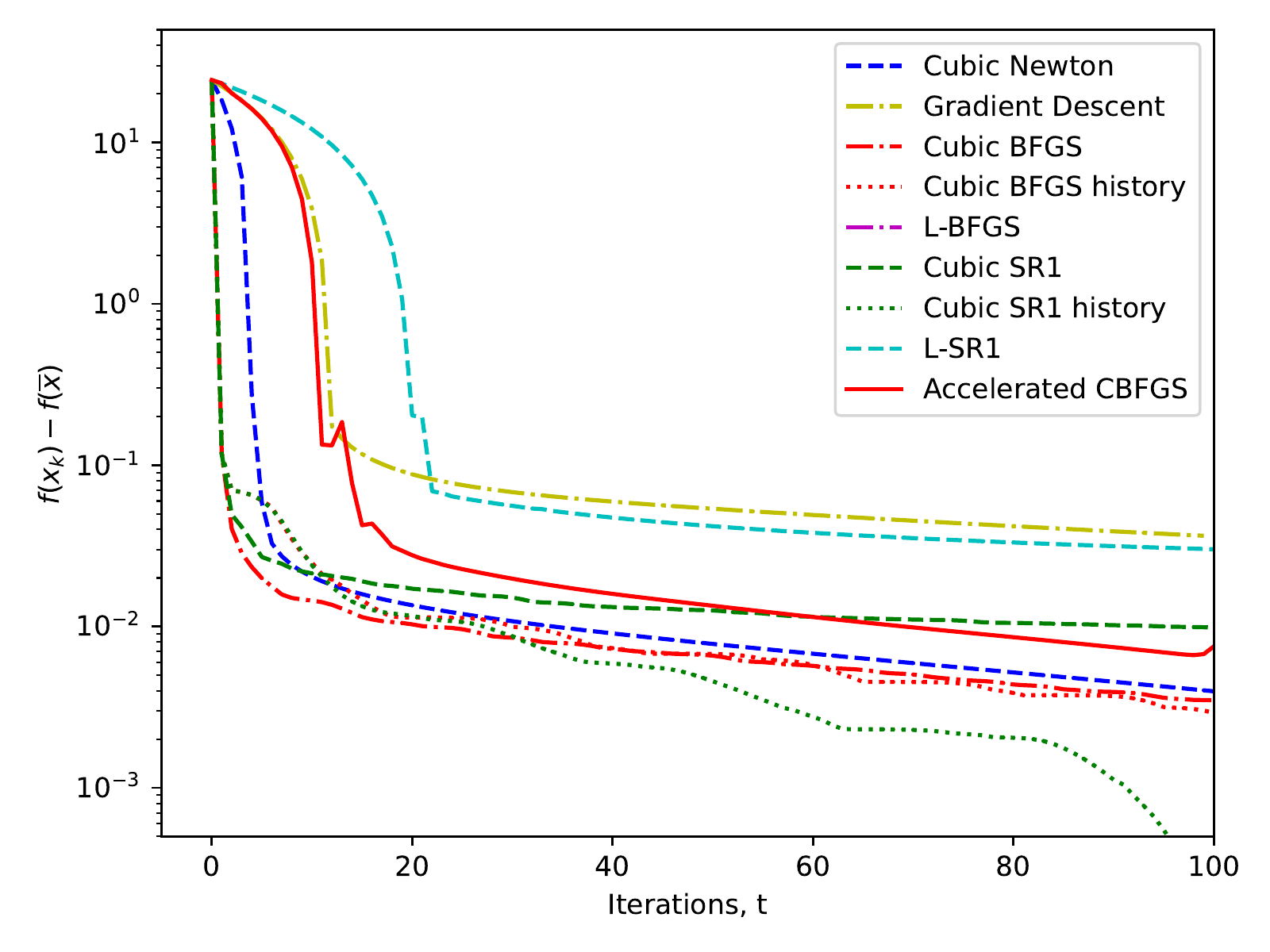}
\\
\subfigure[\texttt{a9a}]{\includegraphics[width=0.24\textwidth]{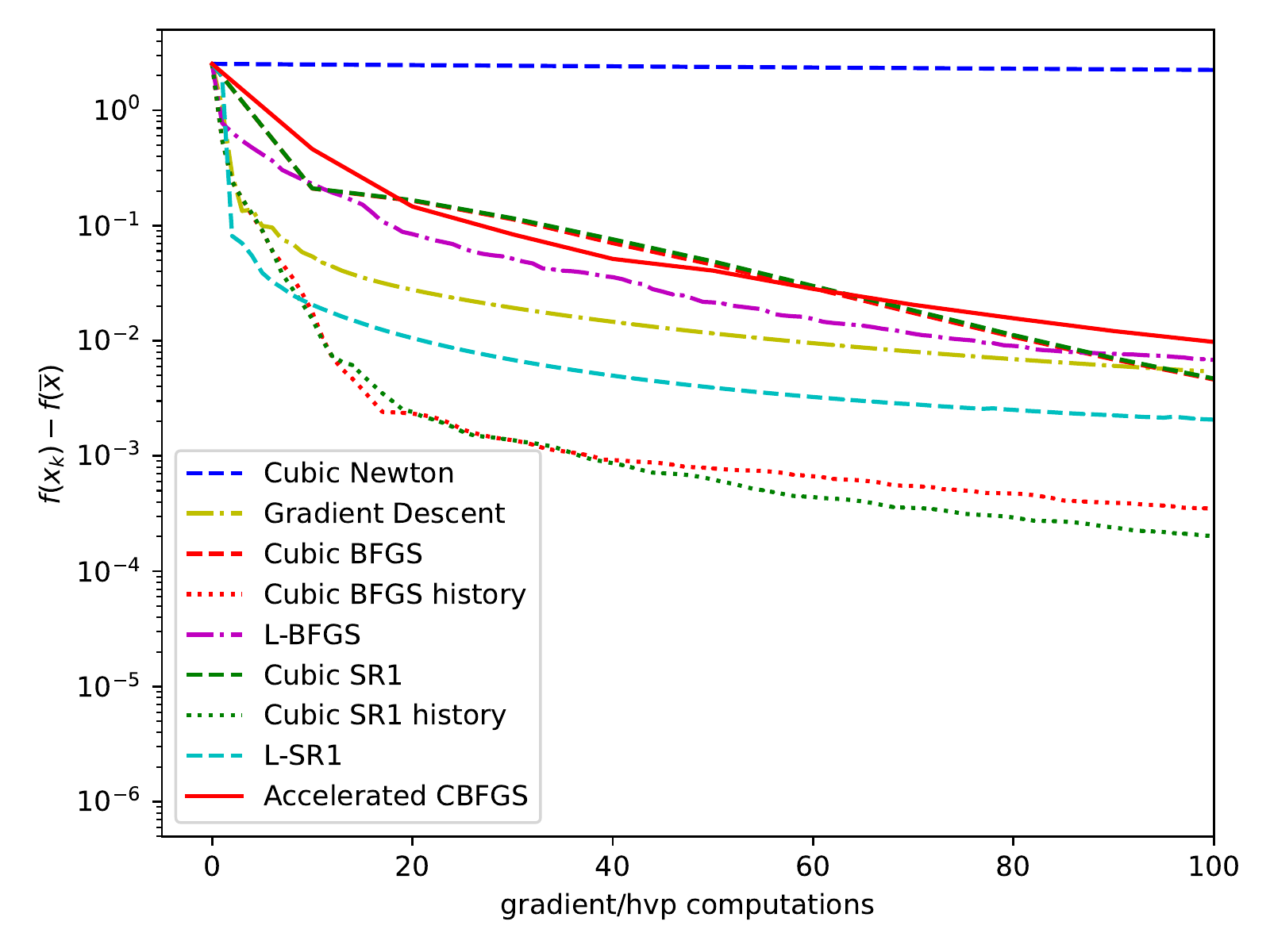}\label{a9a5}}
\subfigure[\texttt{gisette}]{\includegraphics[width=0.24\textwidth]{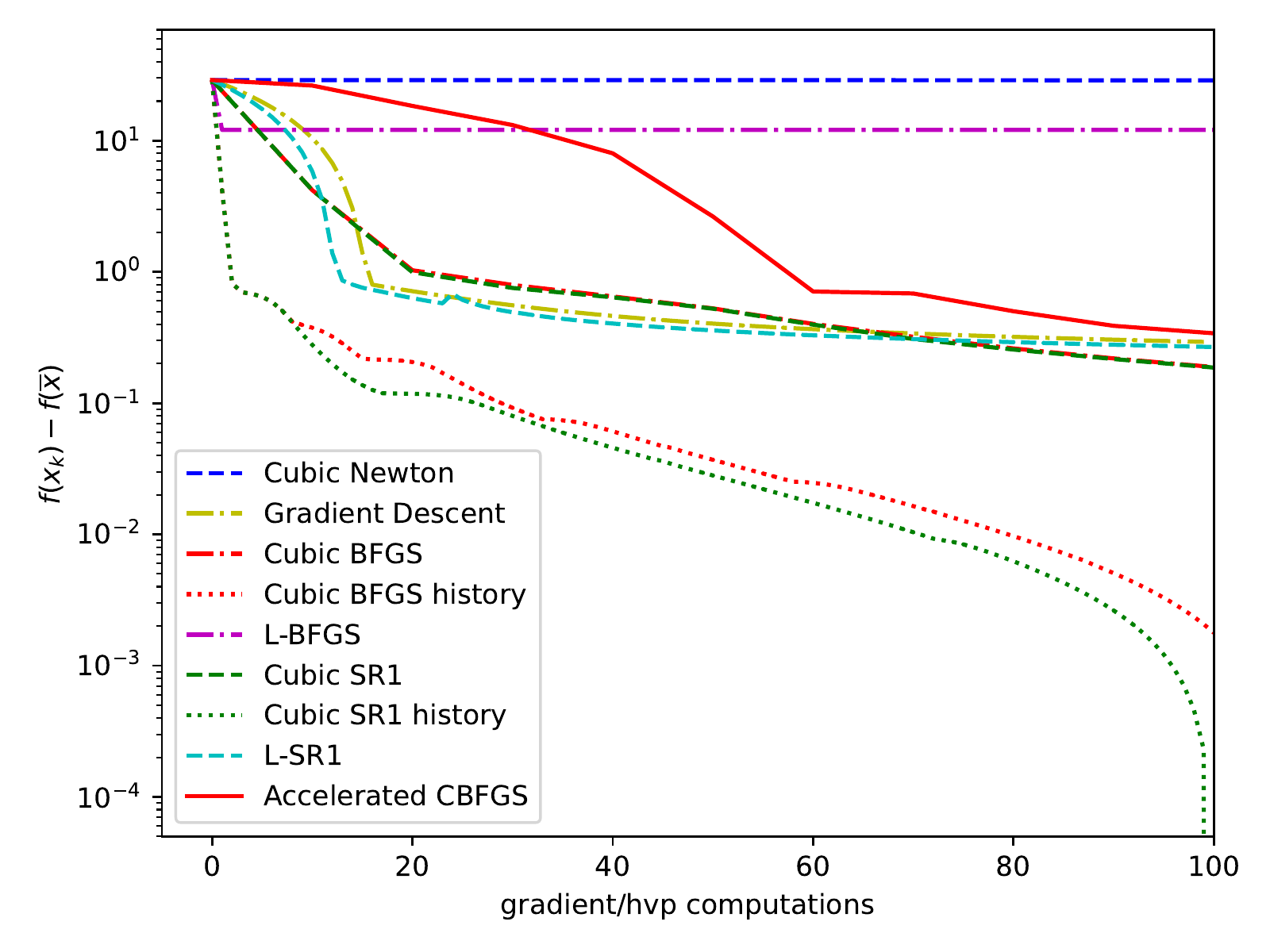}\label{gisette5}}
\subfigure[\texttt{MNIST}]{\includegraphics[width=0.24\textwidth]{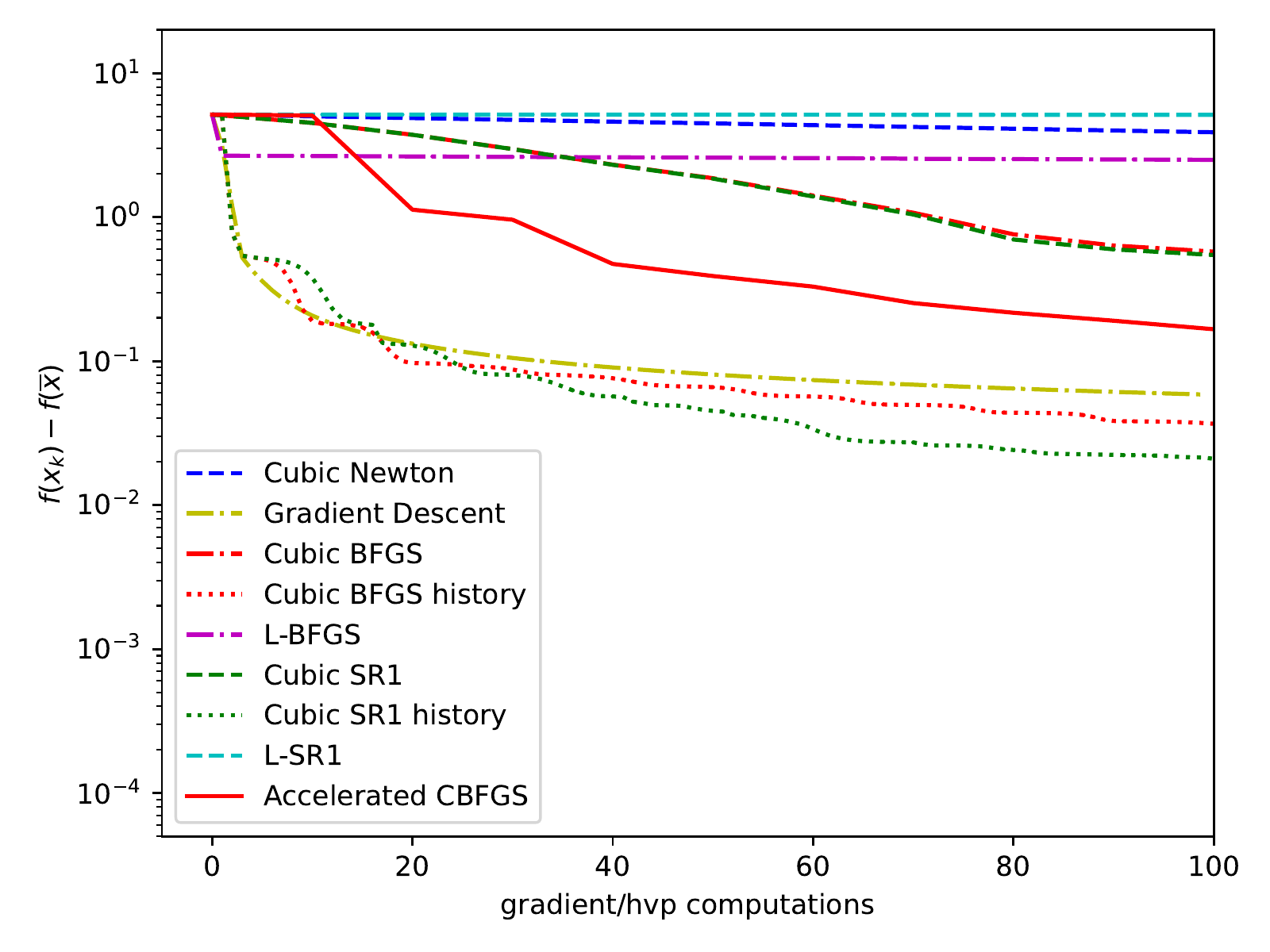}\label{mnist5}}
\subfigure[\texttt{CIFAR-10}]{\includegraphics[width=0.24\textwidth]{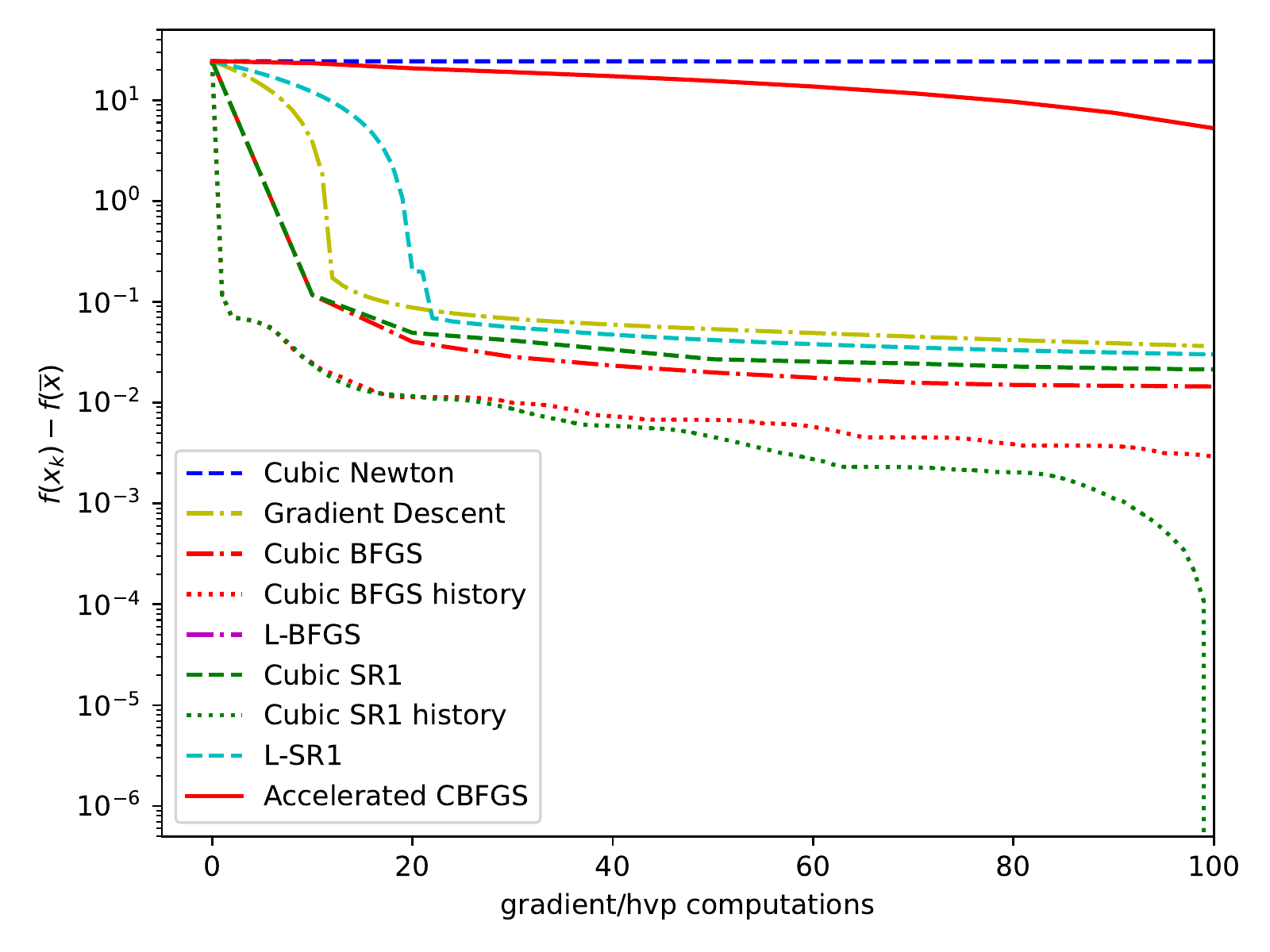}\label{cf5}}
\caption{Comparison of Quasi-Newton methods and Cubic Regularized (Quasi-) Newton methods using scaled parameters in the convex case on datasets  \texttt{a9a}, \texttt{gisette}, \texttt{MNIST}, \texttt{CIFAR-10} respectively.}
\label{fig5}
\vskip-10pt
\end{figure*}

Tuned hyperparameters used for \texttt{a9a} dataset in the convex case shown in Figure \ref{a9a5} are $L_2 = 0.03408$ for Accelerated CBFGS, $lr = 0.125$ for L-BFGS, $lr = 0.00390625$ for L-SR1, $lr = 30$ for Gradient Descent and $L_2 = 0.001704$ for the rest of the methods. 

Tuned hyperparameters used for \texttt{gisette} dataset in the convex case shown in Figure \ref{gisette5} are $L_2 = 0.0055243$ for Cubic Regularized Newton, $L_2 = 0.0224$ for Accelerated CBFGS, $lr = 0.125$ for L-BFGS, $lr = 0.00125$ for L-SR1, $lr = 10$ for gradient descent and $L_2 = 0.00014$ for the rest of the methods.

Tuned hyperparameters used for \texttt{MNIST} dataset in the convex case shown in Figure \ref{mnist5} are $L_2 = 0.001953125$ for Cubic Regularized Newton, $L_2 = 0.0035$ for Accelerated CBFGS, $lr = 0.0006$ for L-BFGS, $lr = 0.00263$ for L-SR1, $lr = 20$ for gradient descent and $L_2 = 0.00035$ for the rest of the methods.

Tuned hyperparameters used for \texttt{CIFAR-10} dataset in the convex case shown in Figure \ref{cf5} are $L_2 = 0.005$ for Cubic Regularized Newton, $L_2 = 0.15$ for Accelerated CBFGS, $lr = 0.0006$ for L-BFGS, $lr =  0.0006$ for L-SR1, $lr = 30$ for gradient descent and $L_2 = 0.0003$ for the rest of the methods.}

\newpage

\section{Solving subproblem } \label{app:subproblem}
In order to solve the subproblem of the Quasi-Newton Cubic step \eqref{eq:cubic_subproblem_solution} we need to find 
$$(\bar{B}_k^m)^{-1} \nabla f_k,$$
where $m$ is the memory size and $\bar{B}_k^m = B_k^m + \delta I + \frac{L}{2}\tau I$. For simplicity, let us denote $B_k^m = B_m$ and $\nabla f(x_k) = g$.

The main goal here is to avoid $d\times d$ matrix inversion and storing full Hessian approximations $B_m$. Quasi-Newton updates usually have low-rank form, so we utilize this fact to apply Woodbury matrix identity \cite{woodbury1949stability, woodbury1950inverting}:
$$ 
(A + UCV)^{-1} = A^{-1} - A^{-1}U(C^{-1} + V A^{-1}U)^{-1}V A^{-1},
$$
where $A, C, U, V$ are conformable matrices and $A$ is $n\times n$, $C$ is $k \times k$, $U$ is $n\times k$ and $V$ is $k\times n$ . 

\paragraph{L-BFGS}

By updating the Hessian approximation $B_k^m$ using L-BFGS formula \eqref{eq:l-bfgs} we can write
\begin{equation}
    B_m = B_{0} + \sum_{i=0}^{m-1} \alpha_i u_i u_i^T + \sum^{m-1}_{i=0} \beta_i v_i v_i^T,
\end{equation}
where $u_i = y_{i}$, $v_i = B_{i}s_{i}$, $\alpha_i = \frac{1}{y_{i}^T s_{i}}$ and $\beta_i = -\frac{1}{s_{i}^T B_{i} s_{i}} =  -\frac{1}{s_{i}^T v_i}$. \\
From Woodbury matrix identity and BFGS we can write  
$$
(\bar{B}_m)^{-1} g = g^T (\bar{B}_0 + W^T C W)^{-1} g = (\bar{B}_0)^{-1} g 
 - (\bar{B}_0)^{-1} W^T (C^{-1} + W (\bar{B}_0)^{-1} W^T)^{-1} W (\bar{B}_0)^{-1} g,
$$
where $\bar{B}_0 + W^T C W = \bar{B}_0 + \sum_{i=0}^{m-1} \alpha_i u_i u_i^T + \sum^{m-1}_{i=0} \beta_i v_i v_i^T$ and $\bar{B}_0$ is $d \times d$, $W$ is $2m \times d$, $C$ is $d \times d$.

\paragraph{L-SR1 update.}
Practically, L-SR1 update is very promising, that is why we include L-SR1 for practical comparison.

In the case of using  L-SR1 formula to update the Hessian approximation $B_m$ we write 
\begin{equation}
    B_m = B_0 + \sum_{i=0}^{m-1} \alpha_i u_i u_i^T,
\end{equation}
where $u_i = y_i - v_i$, $u_i = B_is_i$ and $\alpha_i= \frac{1}{u_i^Ts_i}$.
By using Woodbury's identity we can write
$$
(\bar{B}_m)^{-1} g =  (\bar{B}_0 + W^T C W)^{-1} g = (\bar{B}_0)^{-1} g 
 - (\bar{B}_0)^{-1} W^T (C^{-1} + W (\bar{B}_0)^{-1} W^T)^{-1} W (\bar{B}_0)^{-1} g,
$$
where $\bar{B}_0 + W^T C W = \bar{B}_0 + \sum_{i=0}^{m-1} \alpha_i u_i u_i^T$ and $\bar{B}_0$ is $d \times d$, W is $m\times d$, $C$ is $d\times d$. 

\paragraph{Complexities}

We assume that we initialized $B_0 = cI$ for some constant $c > 0$.
If we calculate the inverse $(C^{-1} + W (\bar{B}_0)^{-1} W^T)^{-1}$ on each step of the line-search the total complexity of line-search will be $\tilde O(m^3)$. But it is unnecessary, we can use one SVD decomposition of $C^{-1} + W(\bar{B}_0^{\tau_0})^{-1} W^T$ for the starting point $\tau_0$:
$$C^{-1} + W(\bar{B}_0^{\tau_0})^{-1} W^T = QZ_0Q^{-1},$$
where $Z_0$ is a diagonal $2m \times 2m$ matrix and $Q$ is orthogonal $2m \times 2m$ matrix for L-BFGS and $Z_0$ is a diagonal $m \times m$ matrix and $Q$ is orthogonal $m \times m$ matrix for L-SR1 .

Then, to calculate the inverse for $\tau_1$ we do:

$$(QZ_0Q^{-1} + \tfrac{L}{2}\tau_1 I - \tfrac{L}{2}\tau_0 I)^{-1} = (QZ_0Q^{-1} + \tfrac{L}{2}(\tau_1-\tau_0) QIQ^{-1})^{-1} = Q (Z_0 - (\tau_1-\tau_0)I)^{-1}Q^{-1}.$$

Thus, each  calculation of $(\bar{B}_m^{\tau_k})^{-1} \nabla f_k$ for new $\tau_k$ from the line-search takes only $O(m^2)$ computations, and the total complexity of line-search procedure to reach accuracy $\tilde{\e}$ is $O(m^2)\log \tilde{\e}^{-1}$. The total complexity of the solution of Quasi Newton Cubic Regularized subproblem is $(m^2d + m^2\log \tilde{\e}^{-1})$, where the first term comes from the calculation of pairs $u_i, v_i$ and constants $\alpha_i, \beta_i$ for $0 \leq i \leq m-1.$ 

\section{Alternative variant of Adaptive Inexact Cubic Regularized for star-convex and $\mu$-strongly star-convex functions.}
In this section, we assume more general class of functions for Adaptive Inexact \CRN methods. But the inexactness criteria is different from the main paper and a little bit harder to adapt to it. This Section was in the main paper in the first version of the paper. 

\begin{assumption}
Let $x^{\ast}$ be a minimizer of the function $f$. For $\mu\geq 0$, the function $f$ is \textbf{$\mu$-strongly star-convex} with respect to $x^{\ast}$ if for all 
$x \in \R^d$
and $\forall \al\in [0,1]$
\begin{equation}
    \label{eq:star-convex-strong}
     f\ls \al x + (1-\al) x^{\ast}\rs \leq   \al f(x) + (1-\al) f(x^{\ast})  \\
       - \frac{\al (1-\al)\mu}{2}\|x-x^{\ast}\|^2.
\end{equation}
If $\mu = 0$ then the function $f$ is \textbf{star-convex} with respect to $x^{\ast}$.
\end{assumption}

\begin{assumption}
\label{ass:dupdlow}
A positive semidefinite matrix $B_x \in \R^{d\times d}$ is an $(\delta^{up},\delta_{low})$-inexact Hessian for the function $f(x)$ at the point $x \in \R^d$ if 
\begin{equation}
\label{eq:delta_alt}
    - \delta_{low} I \preceq \nabla^2 f(x)- B_x \preceq  \delta^{up} I.
\end{equation}
\end{assumption}
This assumption is used for analysis and convergence theorems of star-convex and $\mu$-strongly star-convex functions.

Note, that in \cite{ghadimi2017second} the authors used $(\delta^{up},\delta_{low})$-inexact Hessian with $\delta^{up} = 0$, and in \cite{agafonov2020inexact} the authors used  $(\delta^{up},\delta_{low})$-inexact Hessian with $\delta_{low} = \delta^{up}=\delta$. Hence, our assumption is a generalization of previous approaches. Later in this section, we will clarify what is the main difference between $\delta_{low}$ and $\delta^{up}$ and why it is helpful to differ them.

Let us show that regularized inexact Taylor approximation with $(\delta^{up},\delta_{low})$-inexact Hessian is close to the function $f(x)$ by finding upper and lower bounds.

\begin{lemma}
\label{lem:cubic_bound_alt}
For the function $f(x)$ with $L_2$-Lipschitz-continuous Hessian and $B_x$ is $(\delta^{up},\delta_{low})$-inexact Hessian , for any $x,y \in \R^d$ we have
\begin{align}
\label{eq:func_upper_alt}
    f(y) \leq \phi_{x}(y) + \frac{L_2}{6}\|y-x\|^3 + \frac{\delta^{up}}{2}\|y-x\|^2,\\
    \label{eq:func_lower_alt}
    \phi_{x}(y) \leq f(y) + \frac{L_2}{6}\|y-x\|^3 + \frac{\delta_{low}}{2}\|y-x\|^2.
\end{align}
\end{lemma}
\begin{proof}
One can get the upper-bound \eqref{eq:func_upper_alt} from \eqref{eq:delta_alt}
\begin{equation*}
\begin{gathered}
    f(y) - \phi_{x}(y) \leq 
    f(y) - \Phi_{x}(y) + \Phi_{x}(y) - \phi_{x}(y) \leq \frac{L_2}{6}\|y-x\|^3 + \Phi_{x}(y) - \phi_{x}(y) \\
    \leq \frac{L_2}{6}\|y-x\|^3 + \tfrac{1}{2} \la (\nabla^2f(x) - B_x) (y-x), y-x \ra \stackrel{\eqref{eq:delta_alt}}{\leq} 
    \frac{L_2}{6}\|y-x\|^3 + \frac{\delta^{up}}{2}\|y-x\|^2.
\end{gathered}
\end{equation*}
The lower-bound \eqref{eq:func_lower_alt} comes from \eqref{eq:delta_alt}
\begin{equation*}
\begin{gathered}
    \phi_{x}(y) - f(y) \leq 
    \Phi_{x}(y) - f(y) + \phi_{x}(y)-\Phi_{x}(y) \leq \frac{L_2}{6}\|y-x\|^3 + \phi_{x}(y) - \Phi_{x}(y) \\
    \leq \frac{L_2}{6}\|y-x\|^3 + \tfrac{1}{2} \la (B_x-\nabla^2f(x)) (y-x), y-x \ra \stackrel{\eqref{eq:delta_alt}}{\leq} 
    \frac{L_2}{6}\|y-x\|^3 + \frac{\delta_{low}}{2}\|y-x\|^2 .
\end{gathered}
\end{equation*}
\end{proof}

The first inequality is an upper-bound for the function $f(x)$, hence we can minimize right-hand side and get an optimization method. There are two possible stopping conditions. The main comes from the Lemma \ref{lem:cubic_bound_alt}, we can check inequality \eqref{eq:func_upper_alt} for $y = x_{t+1}, x = x_{t}$, and we denote $h_t = x_{t+1}-x_{t}$:
\begin{equation}
    \label{eq:delta_adaptive_condition1_alt}
    f(x_{t+1}) \leq f(x_{t})+ \la\nabla f(x_{t}),h_t\ra + \frac{1}{2}\la B_{x_{t}} h_t,h_t\ra \\
    + \frac{M}{6}\|h_t\|^3 + \frac{\delta_t^{up}}{2}\|h_t\|^2.
\end{equation}

By using the optimality condition of the subproblem \eqref{eq:inexact_cubic_operator}
\begin{equation}
    \label{eq:subproblem_optimality_condition_alt}
    \nabla f(x_{t})+ B_{x_{t}} h_t + \frac{M}{2}\|h_t\|h_t + \delta_t^{up}h_t=0,
\end{equation}
we simplify \eqref{eq:delta_adaptive_condition1_alt} to get the second condition
\begin{equation}
    \label{eq:delta_adaptive_condition2_alt}
    f(x_{t+1}) \leq f(x_{t})+ \frac{1}{2}\la\nabla f(x_{t}),h_t\ra - \frac{M}{12}\|h_t\|^3.
\end{equation}
We use the second condition in the Algorithm \ref{alg:adaptive_cubic_alt} because it is much easier to compute.

\begin{algorithm} 
		\caption{Adaptive Inexact Cubic Newton} 
        \label{alg:adaptive_cubic_alt}
		\begin{algorithmic}[1]
			\STATE \textbf{Requires:} Initial point $x_0 \in \R^d$, constant $M$ s.t. $M \geq L>0$, initial inexactness $\delta^{up}_0$, increase multiplier $\gamma_{inc}$, decrease multiplier $\gamma_{dec}$.
			\FOR {$t=0,1,\ldots, T$}
                \STATE $x_{t+1} = S_{M,\delta^{up}_t}(x_{t})$
                \WHILE{$f(x_{t+1})> f(x_{t})+ \frac{1}{2}\la\nabla f(x_{t}),x_{t+1}-x_{t}\ra - \frac{M}{12}\|x_{t+1}-x_{t}\|^3$}
			     \STATE $\delta^{up}_t= \delta^{up}_t\gamma_{inc}$
			     \STATE $x_{t+1} = S_{M,\delta^{up}_t}(x_{t})$
                \ENDWHILE
                \STATE $\delta^{up}_{t+1} = \delta^{up}_t\gamma_{dec}$
			\ENDFOR
            \STATE \textbf{Return} $x_{T+1}$
		\end{algorithmic}
	\end{algorithm}

Now, we present the convergence theorem of inexact \CRN for star-convex and $\mu$-strongly star-convex functions. For simplicity, we denote $\delta_t = \delta^{up}_t+\delta^t_{low}$. 

\begin{theorem}
\label{th:inexact_cubic_additional_alt}
Let $f(x)$ be a $\mu$-strongly star-convex function (Option A) or a star-convex function (Option B) with respect to the global minimizer $x^{\ast}$, $f(x)$ has $L_2$-Lipschitz-continuous Hessian, $B_{t}$ is a $(\delta_t^{up},\delta^t_{low})$-inexact Hessian, and $M\geq L_2$, then after the total number of iteration $T\geq 1$ of the Algorithm \ref{alg:adaptive_cubic_alt}, we get
\begin{gather}
    \text{Option A} \qquad  f(x_{t+1})-f(x^{\ast}) = (1-\alpha_t)\ls f(x_{t})-f(x^{\ast})\rs, \quad \forall t \in [0, \ldots, T], \label{eq:strongly-convergence_alt}\\
    \text{where} \qquad \al_{t} = \min \lb \frac{1}{3},\frac{\mu}{3(\delta_t^{up}+\delta_{low}^t)} , \sqrt{\frac{\mu}{2MR_0}}, \rb\label{eq:al_k_defenition_alt}\\
    \text{Option B} \quad  f(x_{T+1})-f(x^\ast) 
        \leq \frac{9 MR^3 }{T^2} + \frac{3R^2 }{(T+1)T^2}\sum\limits_{t=0}^{T} \ls  (\delta^{up}_t+\delta^t_{low}) (t+1)\rs,
    \label{eq:convex-convergence_alt}
\end{gather}
where $\al_0 = 1$, $\al_t \in (0;1)$, $A_0 = 1$, and
\begin{equation*}
            A_t = 
                \begin{cases}
                    1 , &t = 0\\
                    \prod \limits_{i=1}^t (1-\alpha_i), &t\geq 1.
                \end{cases}
        \end{equation*}
For simplicity, we denote
\begin{equation}
    \label{eq:R_alt}
    R_0 = \|x_0 - x^{\ast}\|,\qquad
    R = \max\limits_{x \in \mathcal{L}} \|x - x^{\ast}\|,
\end{equation}
where $R$ represents the diameter of the level set $\mathcal{L}~=~\lb x \in \R^d : f(x) \leq f(x_0) \rb$,
 and
\begin{equation}
    \label{eq:delta_t_def_alt}
    \delta_t = \delta^{up}_t+\delta^t_{low}
\end{equation}
\end{theorem}

Let us discuss the results of the theorem. The term in convergence rate with $M\geq L_2$ corresponds to the classical Cubic Regularized Method. 
The convergence rate with $\delta$ can be interpreted in several ways. There are two cases: we can not control inexactness, so the method has to adapt to the existing inexactness of Hessians, in other case we can control inexactness and the method can make the approximation as precise as it wants to converge faster: 

\textbf{1) Uncontrollable inexactness.} 
 \begin{corollary}
Adaptive Inexact Cubic Newton with Uncontrollable Inexactness, such that $\delta_{t} \leq \delta_{\max}$, performs $T\geq 1$ iterations to find $\varepsilon$-solution $x_T$ such that $f(x_{T}) - f(x^{\ast})\leq \varepsilon$. $T$ is bounded by \\
(Option A)  
\begin{equation*}
T =  O(1)\max \lb 1;\frac{\delta_{\max}}{\mu};\sqrt{\frac{MR_0}{\mu}} \rb \log \ls \frac{f(x_0)-f(x^{\ast})}{\varepsilon}\rs,
\end{equation*}
(Option B) 
\begin{equation*}
    T = O(1) \max\lb \frac{\delta_{\max}R^2}{\varepsilon};  \sqrt{\frac{MR^3}{\varepsilon}}\rb.
\end{equation*}
\end{corollary}
The convergence rate by $\delta_{\max}$ corresponds to the convergence of the gradient descent. Note, that if we take very bad approximation of $B_x$ that equals to all $0$ and function has $L_1$-Lipschitz-continuous gradient, then $\delta_{low}=\delta^{up}=L_1$ and we get the same convergence rate as gradient descent.

\textbf{2) Controllable inexactness.} 
\begin{corollary}
Adaptive Inexact Cubic Newton with Controllable Inexactness, such that $\delta^t_{low} = 0, \delta_t^{up} \leq \sqrt{\mu MR_0}$ for (Option A) or $\delta^t_{low} = 0, \delta_t^{up} \leq \tfrac{MR}{t+1}$ for (Option B), performs $T\geq 1$ iterations to find $\varepsilon$-solution $x_T$ such that $f(x_{T}) - f(x^{\ast})\leq \varepsilon$. $T$ is bounded by \\
(Option A) 
\begin{equation*}
    T =  O(1)\max \lb 1; \sqrt{\frac{MR_0}{\mu}} \rb \log \ls \frac{f(x_0)-f(x^{\ast})}{\varepsilon}\rs,
\end{equation*}
(Option B) 
\begin{equation*}
    T = O\ls  \sqrt{\frac{MR^3}{\varepsilon}}\rs.
\end{equation*}
\end{corollary}
So, we get the same convergence as the classical Cubic Regularized Newton.

To sum up, we propose new Adaptive Inexact Cubic Newton under new inexactness assumptions. It opens up new possibilities of choosing approximation $B_{x}$ and control $\delta^{up}$ and $\delta_{low}$. Note, if we want we can create such $B_{x}$ that $\delta^{up}=0$, then we don't need to choose this parameter inside the steps of the method. On the other hand, we can choose $B_{x}$ such that $\delta_{low}=0$, then we can control level of the errors by $\delta^{up}$ and make an Adaptive Inexact \CRN that can control $\delta^{up}$ on desired level. 

Next, we prove the convergence theorem for the Algorithm \ref{alg:adaptive_cubic_alt}.

\begin{proof}
\begin{equation*}
    \begin{gathered}
        f(x_{t+1})-f(x^\ast) \stackrel{\eqref{eq:delta_adaptive_condition1_alt}}{\leq}  \min \limits_{y \in \mathbb{R}^n} \left \{  \phi_{x_{t}}(y) + \frac{M}{ 6}\|y - x_{t}\|^3 + \frac{ \delta^{up}_t}{2} \|y - x_{t}\|^2  \right \}-f(x^\ast)\\
        \stackrel{\eqref{eq:func_lower_alt}}{\leq}  \min \limits_{y \in \mathbb{R}^n} \left \{  f(y) + \frac{M}{ 3}\|y - x_{t}\|^3 + \frac{ \delta^{up}_t+\delta_{low}^t}{2} \|y - x_{t}\|^2  \right \}-f(x^\ast)\\
        \stackrel{\eqref{eq:delta_t_def_alt}}{=}  \min \limits_{y \in \mathbb{R}^n} \left \{  f(y) + \frac{M}{ 3}\|y - x_{t}\|^3 + \frac{ \delta_t}{2} \|y - x_{t}\|^2  \right \}-f(x^\ast)\\
        \stackrel{y=x_{t} + \al_{t}(x^{\ast} - x_{t})}{\leq}  f((1-\al_{t}) x_{t} + \al_{t} x^{\ast} ) + \al_{t}^3\frac{M}{ 3}\|x^{\ast} - x_{t}\|^3 + \al_{t}^2\frac{\delta_t}{2} \| x^{\ast} - x_{t}\|^2  -f(x^\ast)
        \end{gathered}
\end{equation*}
Note, that from the second inequality, we also get that the method is monotone and $f(x_{t+1})\leq f(x_{t})$.
        Next, the proof splits for two options A and B. Let us start with option A and prove the equation \eqref{eq:strongly-convergence_alt}. So, for the last inequality, we use the definition of $\mu$-strongly star-convexity.
        \begin{equation*}
    \begin{aligned}
        f(x_{t+1})-f(x^\ast)  \leq &f((1-\al_{t}) x_{t} + \al_{t} x^{\ast} )-f(x^\ast) + \al_{t}^3\frac{M}{ 3}\|x_{t} - x^{\ast}\|^3 + \al_{t}^2\frac{\delta_t}{2} \| x_{t} - x^{\ast}\|^2 \\
        \stackrel{\eqref{eq:star-convex-strong}}{\leq}  &(1-\al_{t})f(x_{t}) + \al_{t} f(x^{\ast})-f(x^\ast) \\
        &- \frac{\al_{t} (1-\al_{t})\mu}{2}\|x_{t}-x^{\ast}\|^2 
        + \al_{t}^3\frac{M}{ 3}\|x_{t}-x^{\ast}\|^3 + \al_{t}^2\frac{\delta_t}{2} \|x_{t}-x^{\ast}\|^2 \\
        = &(1-\al_{t})\ls f(x_{t}) - f(x^{\ast})\rs\\ 
         &- \frac{\al_{t}}{2}\|x_{t}-x^{\ast}\|^2 \ls (1-\al_{t})\mu 
        - \al_{t}^2\frac{2M}{ 3}\|x_{t}-x^{\ast}\| - \al_{t}\delta_t \rs
    \end{aligned}
\end{equation*}
    By the definition of $\al_{t}$ from \eqref{eq:al_k_defenition_alt}, one can see that
    \begin{equation*}
    \ls \frac{\mu}{3} - \al_{t}\mu \rs + 
    \ls \frac{\mu}{3} - \al_{t}^2\frac{2M}{ 3}\|x_{t}-x^{\ast}\| \rs +
    \ls \frac{\mu}{3} - \al_{t}\delta_t\rs \geq 0.
\end{equation*}
Hence, we finally prove \eqref{eq:strongly-convergence_alt}
$$
f(x_{t+1})-f(x^{\ast})  \leq (1-\al_{t})\ls f(x_{t}) - f(x^{\ast})\rs.
$$

Now, we start proving option B for star-convex functions to get \eqref{eq:convex-convergence_alt}.

\begin{align}
        f(x_{t+1})&-f(x^\ast)  \leq f((1-\al_{t}) x_{t} + \al_{t} x^{\ast} )-f(x^\ast) + \al_{t}^3\frac{M}{ 3}\|x_{t} - x^{\ast}\|^3 + \al_{t}^2\frac{\delta_t}{2} \| x_{t} - x^{\ast}\|^2 \notag\\
        &\stackrel{\eqref{eq:star-convex-strong}}{\leq}  (1-\al_{t})(f(x_{t}) - f(x^{\ast}))
        + \al_{t}^3\frac{M}{ 3}\|x_{t}-x^{\ast}\|^3 + \al_{t}^2\frac{\delta_t}{2} \|x_{t}-x^{\ast}\|^2.
        \label{eq:convex_bound_0_alt}
\end{align}
Note, that $\al_0 = 1$, then for the first iteration we have
\begin{equation}
\label{eq:first_iteration_bound_alt}
        f(x_{1})-f(x^\ast)  \leq  
        \frac{M}{ 3}\|x_0-x^{\ast}\|^3 + \frac{ (\delta^{up}_0+\delta_{low}^0)}{2} \|x_0-x^{\ast}\|^2.
\end{equation}
Now by dividing both sides of \eqref{eq:convex_bound_0_alt} by $A_{t}$ and using the fact that $A_{t}=A_{t-1} (1-\al_{t})$, we get
\begin{align}
        &\frac{f(x_{t+1})-f(x^\ast)}{A_{t}}  \leq  \frac{(1-\al_{t})(f(x_{t}) - f(x^{\ast}))}{A_{t}}
        + \frac{\al_{t}^3}{A_{t}}\frac{M}{ 3}\|x_{t}-x^{\ast}\|^3 + \frac{\al_{t}^2}{A_{t}}\frac{\delta_t}{2} \|x_{t}-x^{\ast}\|^2 \notag \\
        &=  \frac{f(x_{t}) - f(x^{\ast})}{A_{t-1}}
        + \frac{\al_{t}^3}{A_{t}}\frac{M}{ 3}\|x_{t}-x^{\ast}\|^3 + \frac{\al_{t}^2}{A_{t}}\frac{\delta_t}{2} \|x_{t}-x^{\ast}\|^2, \quad \forall t\in \lb 1, \ldots, T \rb \label{eq:convex_bound_1_alt}
\end{align}

Finally, we sum up \eqref{eq:convex_bound_1_alt} for all $t\in \lb 1, \ldots, T\rb$,  apply \eqref{eq:first_iteration_bound_alt}, and as a result, we get
\begin{equation*}
    \frac{f(x_{T+1})-f(x^\ast)}{A_{T}} 
        \leq  
        \sum\limits_{t=0}^{T} \ls\frac{\al_t^3}{A_{t}}\frac{M}{ 3}\|x_t-x^{\ast}\|^3 + \frac{\al_t^2}{A_{t}}\frac{\delta_t}{2} \|x_t-x^{\ast}\|^2\rs.
\end{equation*}
By multiplying both parts on $A_T$, we get
\begin{align}
    f(x_{T+1})-f(x^\ast) 
        &\leq  
        \sum\limits_{t=0}^{T} \ls  \frac{\al_t^2A_{T}}{A_{t}}\frac{ \delta_t}{2}\|x_t-x^{\ast}\|^2  + \frac{\al_t^3A_{T}}{A_{t}}\frac{M}{ 3}\|x_t-x^{\ast}\|^3\rs \notag\\
        &\stackrel{\eqref{eq:R_alt}}{\leq}\sum\limits_{t=0}^{T} \ls  \frac{\al_t^2A_{T}}{A_{t}}\frac{ \delta_tR^2}{2}  + \frac{\al_t^3A_{T}}{A_{t}}\frac{MR^3}{ 3}\rs \notag\\
        &= \frac{R^2 A_{T}}{2}\sum\limits_{t=0}^{T} \ls  \frac{\delta_t\al_t^2}{A_{t}}\rs  + \frac{MR^3 A_{T}}{ 3}\sum\limits_{t=0}^{T} \ls \frac{\al_t^3}{A_{t}}\rs.
        \label{eq:sum_a_A_alt}
\end{align}
Now, we fix $\alpha_t$ to upperbound both sums from \eqref{eq:sum_a_A_alt}. Let us take
        \begin{equation*}
            \alpha_t = \frac{3}{t+3}, ~ t \geq 1.
        \end{equation*}
        Then, we have
        \begin{equation*}
            A_{T} = 
            \prod_{t=1}^{T}\left(1-\alpha_{t}\right)=\prod_{t=1}^{T} \frac{t}{t+3}=\frac{T !3 !}{(T+3) !}=\frac{6}{(T+1)(T+2)(T+3)}.
         \end{equation*}
        Finally, by upperbounding two parts of the sum from \eqref{eq:sum_a_A_alt}, 
        \begin{equation*}
                \sum_{t=0}^{T} \frac{\alpha_{t}^{3}}{A_{t}} =\tfrac{9}{2}\sum_{t=0}^{T}  \frac{(t+1)(t+2)}{(t+3)^2} \leq 4.5 (T+1),
        \end{equation*}
        
        \begin{equation*}
                \sum_{t=0}^{T} \frac{\delta_{t}\alpha_{t}^{2}}{A_{t}} =\tfrac{3}{2}\sum_{t=0}^{T}  \frac{\delta_t(t+1)(t+2)}{(t+3)} \leq \tfrac{3}{2}\sum_{t=0}^{T}  \delta_t(t+1),
        \end{equation*}
we prove \eqref{eq:convex-convergence_alt}.
\end{proof}

\section{Inexact Cubic Newton methods}

    \begin{table}[H]
    \centering
    \caption{Comparison between different inexact cubic Newton papers.}
        \begin{tabular}{|c|c|c|c|c|}
        \hline
        \multirow{2}{10em}{\centering paper} & \multirow{2}{10em}{\centering inexactness} & \multirow{2}{5em}{\centering control of \\  inexactness} & convergence rate \\
        &  &   & (convex case)\\
        \hline
        \multirow{3}{10em}{\centering \citep{ghadimi2017second}} & \multirow{2}{11em}{\centering $-\delta I \preceq \nabla^2 f(x) - B_x   \preceq 0$} & \multirow{2}{3em}{\centering\xmark} & $\frac{\delta R^2}{T} + \frac{L_2R^3}{T^2}$\\[7pt]\cline{2-4}
        &\multirow{2}{18em}{\centering $-\delta I  \preceq \nabla^2 f(x) - B_x \preceq -\delta I/2$} & \multirow{2}{3em}{\centering\xmark}  & $\frac{\delta R^2}{T^2} + \frac{L_2R^3}{T^3}$ \\[7pt] 
        \hline
        \multirow{3}{10em}{\centering \citep{agafonov2020inexact}} & \multirow{3}{18em}{\centering $ \|\nabla^2 f(x) - B_x\| \leq \delta $}  & \multirow{2}{3em}{\centering\xmark} & $\frac{\delta R^2}{T} + \frac{L_2R^3}{T^2}$\\[7pt]\cline{3-4}
        & &\multirow{2}{3em}{\centering\xmark} & $\frac{\delta R^2}{T^2} + \frac{L_2R^3}{T^3}$\\[7pt]
        \hline
        \multirow{3}{10em}{\centering [This paper]} & \multirow{3}{18em}{\centering $\|(\nabla^2 f(x) - B_x)(y-x)\| \leq \delta_{x}^{y}\|y-x\|$} & \multirow{2}{3em}{\centering\cmark} & $\frac{\delta_T R^2}{T} + \frac{L_2R^3}{T^2}$\\[7pt]\cline{3-4}
        &  & \multirow{2}{3em}{\centering\cmark} & $\frac{\delta_T R^2}{T^2} + \frac{L_2R^3}{T^3}$\\[7pt]
        \hline
        {\centering [This paper,additional]} & {\centering $- \delta_{low} I \preceq \nabla^2 f(x)- B_x \preceq  \delta^{up} I$} & {\centering\cmark} & $\frac{\hat{\delta}^{up} R^2}{T} + \frac{L_2R^3}{T^2}$\\
        \hline
        \end{tabular}
        \vspace{-10pt}
        \label{tab:1}
    \end{table}

    All these inexact \CRN methods achieve the convergence rates of the same order but under different assumptions on Hessian inexactness. In our paper, we defined two new inexactness concepts. 
    \begin{enumerate}
        \item In the main part, we define $\delta_x^y$-inexact Hessian (Assumption \ref{ass:direct_inexactness}) for non-accelerated and accelerated method. This allows to construct an algorithm, which can adapt to inexactness. In \citep{ghadimi2017second, agafonov2020inexact} Hessian approximation should be close to the true Hessian in whole space (see Table \ref{tab:1}).  In our approach, the Hessian approximation should be close to the Hessian only in the direction of the step $x_{k+1} - x_{k}$. Hence, one can show that $\delta_x^y\leq \delta$, and in practice $\delta_x^y$ can be much smaller than the classical $\delta$.
        \item  For non-accelerated method, we also propose an alternative version  $(\delta_{low},\delta^{up} )$ Hessian inexactness, which is a generalization of Hessian inexactness from \cite{ghadimi2017second, agafonov2020inexact} (see Assumption \ref{ass:dupdlow}) but it can be adaptive and more applicable in practice.

    \end{enumerate}

    One of the main differences between our paper and other inexact Cubic Newton algorithms is ability to \textit{control the inexactness}. It can be viewed from two angles. 
    \begin{itemize}
        \item Let us consider the case when we can control Hessian inexactness. For example, it can be finite-sum or stochastic setup. On each iteration, we can increase the batch size to decrease Hessian error. For Quasi-Newton methods with sampling one can perform more number of Hessian-vector products to increase the accuracy of Hessian approximation to the desired value. Thus, inexactness is controllable. In Algorithms 1, 2 we provided two techniques to validate the step. So, given $\delta_t = \frac{1}{t}$ we can increase Hessian approximation's accuracy and achieve $O(1/T^3)$ convergence rate.
        \item On other hand, during the work of algorithms (Alg. 1, 2) the method learns the true $\delta$. Other approaches \cite{ghadimi2017second, agafonov2020inexact} are not able to adapt to inexactness. \citep{agafonov2020inexact} sets up a particular $\delta$ before the work of algorithm. It is also true for Hessian approximation assumption for the accelerated method in \citep{ghadimi2017second}.  
    \end{itemize}

\end{document}